\documentclass[11 pt, final]{article}
\title{A Dehn function for $Sp(2n,\mathbb{Z})$}
\author{David Bruce Cohen}

\usepackage{fullpage}
\usepackage[latin1]{inputenc}
\usepackage{tikz}
\usetikzlibrary{shapes,arrows}
\usepackage{graphicx}
\usepackage{amsmath}
\usepackage{amssymb}
\usepackage{amsthm}
\usepackage{epsfig,pinlabel}
\usepackage{amsfonts}
\usepackage{amscd}

\DeclareMathOperator{\Sp}{Sp}
\DeclareMathOperator{\SpH}{Sp\ensuremath{H}}
\DeclareMathOperator{\fsp}{\ensuremath{\mathfrak{sp}}}
\DeclareMathOperator{\GL}{GL}
\DeclareMathOperator{\SL}{SL}
\DeclareMathOperator{\SO}{SO}
\DeclareMathOperator{\ZZ}{\ensuremath{\mathbb{Z}}}

\DeclareMathOperator{\RR}{\ensuremath{\mathbb{R}}}
\DeclareMathOperator{\QQ}{\ensuremath{\mathbb{Q}}}
\DeclareMathOperator{\TT}{\ensuremath{\mathbb{T}}}
\DeclareMathOperator{\NN}{\ensuremath{\mathbb{N}}}

\DeclareMathOperator{\he}{\ensuremath{\hat{e}}}
\DeclareMathOperator{\hu}{\ensuremath{\hat{u}}}
\DeclareMathOperator{\huZ}{\ensuremath{\hat{u}_{Z}}}
\DeclareMathOperator{\diag}{diag}
\DeclareMathOperator{\Diag}{Diag}
\DeclareMathOperator{\Sym}{Sym}

\DeclareMathOperator{\Cparab}{\ensuremath{C_{\text{parab}}}}
\DeclareMathOperator{\mH}{\ensuremath{\mathcal{H}}}
\DeclareMathOperator{\Ab}{Ab}
\DeclareMathOperator{\bv}{\ensuremath{\overline{v}}}
\DeclareMathOperator{\bvp}{\ensuremath{\overline{v^{\prime}}}}
\DeclareMathOperator{\Lip}{Lip}
\DeclareMathOperator{\diam}{diam}
\DeclareMathOperator{\Hom}{Hom}
\DeclareMathOperator{\Id}{Id}
\DeclareMathOperator{\Area}{Area}
\DeclareMathOperator{\Sol}{Sol}

\DeclareMathOperator{\Isom}{Isom}

\DeclareMathOperator{\Shssr}{\ensuremath{\Sigma_{H_{S,S^{\prime}}(\RR)}}}

\DeclareMathOperator{\Shssz}{\ensuremath{\Sigma_{H_{S,S^{\prime}}(\ZZ)}}}
\DeclareMathOperator{\SSp}{\ensuremath{\Sigma_{\Sp(2p;\ZZ)}}}

\theoremstyle{plain}
\newtheorem{theorem}{Theorem}[section]
\newtheorem{lemma}[theorem]{Lemma}

\newtheorem*{nolabeltheorem}{Theorem}
\newtheorem{conjecture}[theorem]{Conjecture}

\newtheorem{proposition}[theorem]{Proposition}
\newtheorem{corollary}[theorem]{Corollary}
\theoremstyle{definition}
\newtheorem*{definition}{Definition}

\begin{document}

\maketitle
\begin{abstract}
Gromov conjectured that any irreducible lattice in a symmetric space of rank at least $3$ should have at most polynomial Dehn function. We prove that the lattice $\Sp(2p;\ZZ)$ has quadratic Dehn function when $p\geq 5$. By results of Broaddus, Farb, and Putman, this implies that the Torelli group in large genus is at most exponentially distorted.
\end{abstract}
\section{Introduction}
\label{section:Introduction}
\subsection{Statement of main theorem}
\label{subsection:dehnfunctions}
Before stating our main theorem, we must recall the definition of the Dehn function of a finitely presented group.
\paragraph{Dehn functions.} Let $\langle \mathcal{S}|\mathcal{R}\rangle$ be a finite presentation for a group $G$ and let $F(\mathcal{S})$ denote the free group on $\mathcal{S}$.  If $w\in F(\mathcal{S})$ represents the identity in $G$, then $w$ can be written as a product of conjugates of relators, i.e.,
$$w=\prod_{j=1}^{k}w_{j}r_{j}w_{j}^{-1},$$
where the $r_{j}$ are elements of $\mathcal{R}$, and the $w_{j}$ are elements of $F(\mathcal{S})$. The area of $w$ is defined to be the smallest number of relators needed in this sort of expression, i.e.,
$$\Area(w)=\inf\{k:w=\prod_{j=1}^{k}w_{j}r_{j}w_{j}^{-1}\}.$$
\begin{definition}
\label{def:groupdehnfunction}
The Dehn function of $G$ is the function
$$\delta_{G}:\NN\rightarrow \NN$$
whose value at $n$ is the maximum area of any word of length at most $n$:
$$\delta_{G}(n)=\sup\{\Area(w)|\ell(w)\leq n\}.$$
\end{definition}
Of course, this depends on the choice of presentation, but its growth rate does not.  To make this precise, we need the following definition.
\begin{definition}
Let $f,g$ be non-decreasing functions $\NN\rightarrow [0,\infty)$.  We say that $f\preceq g$ if there is a constant $C$ such that
$$f(n)\leq Cg(Cn+C)+Cn+C$$
If $f\preceq g$ and $g\preceq f$, we say that $g\simeq f$.  Observe that $\simeq$ is an equivalence relation.
\end{definition}
If $\delta_{1}$ and $\delta_{2}$ are Dehn functions of two different presentations of $G$, then $\delta_{1}\simeq\delta_{2}$.  Consequently, it makes sense to speak of a group having quadratic Dehn function or exponential Dehn function, and in fact these properties are invariant on passing to finite index subgroups.  We refer to an upper bound on the Dehn function as an isoperimetric inequality.

\paragraph{The symplectic group.}
For $p\in \NN$, the integer symplectic group $\Sp(2p;\ZZ)$ is the group of $2p\times 2p$ integer matrices which preserve a standard skew symmetric bilinear pairing on $\ZZ^{2p}$.  The following is our main theorem.
\begin{theorem}
\label{theorem:MAINTHM}
For $p\geq 5$, the group $\Sp(2p;\ZZ)$ has quadratic Dehn function.
\end{theorem}

For comparison, Young\cite{RY} proved the following theorem.
\begin{theorem}[Young]
\label{theorem:RY}
If $p\geq 5$, then $\SL(p;\ZZ)$ has quadratic Dehn function.
\end{theorem}

\paragraph{Distortion in Torelli.}
The Torelli group is the subgroup $\mathcal{I}(\Sigma)\subset Mod(\Sigma)$ of the mapping class group of a surface $\Sigma$ which acts trivially on $H_{1}(\Sigma)$, where $\Sigma$ has at most $1$ boundary component (see \cite{mcgprimer} for a comprehensive introduction to mapping class groups of surfaces).  If $H\subset G$ are groups with finite generating sets $S_{H}$ and $S_{G}$ respectively, then the distortion function of $H$ in $G$ is the number of generators of $H$ required to represent a word of length $n$ in the generators of $G$ which lands in $H$.  This function depends on the presentations of $H$ and $G$, but all choices of presentation yield $\simeq$-equivalent distortion functions. Broaddus, Farb, and Putman proved that for genus at least three, the Torelli subgroup is at least exponentially and at most doubly exponentially distorted in $Mod(\Sigma)$\cite[Theorem 1.1]{torelli}. They further showed that if $\Sp(2g;\ZZ)$ has quadratic Dehn function, then $\mathcal{I}(\Sigma)$ is at most exponentially distorted\cite[Remark 3.1]{torelli}, so our theorem has the following corollary.

\begin{corollary}
\label{corollary:torelli}
For a surface $\Sigma$ of genus at least five, the Torelli group $\mathcal{I}(\Sigma)$ is exponentially distorted in the mapping class group $Mod(\Sigma)$.
\end{corollary}

\subsection{Gromov's conjecture}
\label{subsection:gromov}
Our main theorem verifies a special case of a conjecture of Gromov about lattices in symmetric spaces. We now explain what this conjecture is, why it is widely believed, and how previous authors have attacked it. First though, we must recall some basic definitions about Dehn functions (for Riemannian manifolds), symmetric spaces, lattices, and horoballs.

\paragraph{Dehn functions for Riemannian Manifolds.}
Let $X$ be a complete, simply connected Riemmanian manifold.  Then given a piecewise smooth loop $\gamma:S^{1}\rightarrow X$, we can find a piecewise smooth function $f:D^{2}\rightarrow X$ whose restriction $\partial f$ to the boundary is $\gamma$.  The area of $\gamma$ is defined to be the smallest area of any such filling, i.e.,
$$\Area(\gamma)=\inf\{\Area(f)|\partial f=\gamma\}$$
\begin{definition}
\label{def:spacedehnfunction}
The Dehn function of $X$ is the function
$$\delta_{X}:[0,\infty)\rightarrow [0,\infty]$$
whose value at $r$ is the largest area of any loop of length $r$.
$$\delta_{X}(r)=\sup\{\Area(\gamma):\ell(\gamma)=r\}$$ 
\end{definition}

\paragraph{Remark:} The function $\delta_{X}$ is traditionally called the filling function of $X$ or the isoperimetric inequality of $X$.  However, the term ``isoperimetric inequality" has increasingly been used for any upper bound $f(r)$ on the area of loops of length at most $r$, regardless of whether this bound is sharp.  Our terminology follows \cite[Definition 1.1]{lp} and \cite[\S 2.2]{RY}.  If $G$ is a finitely presented group acting on $X$ properly, cocompactly and by isometries, then $\delta_{X}\simeq \delta_{G}$.  For instance, the Euclidean plane $E^{2}$ clearly has quadratic Dehn function, and $\ZZ^{2}$ has a free cocompact action on $E^{2}$, so $\ZZ^{2}$ has quadratic Dehn function.\\

\paragraph{Symmetric spaces.}
A symmetric space of noncompact type is a Riemannian manifold of the form $X=G/K$ where $G$ is a semisimple real Lie group and $K$ a maximal compact subgroup. The rank of a symmetric space is the largest dimension of any Euclidean subspace. Examples include the following.
\begin{itemize}
\item The space $\SL(p;\RR)/\SO(p)$ which parametrizes positive definite metrics on $\RR^{p}$ is a symmetric space of rank $p-1$.
\item The $p$-th Cartesian power of the hyperbolic plane, $(H^{2})^{p}$, is a symmetric space of rank $p$.
\item The space $\Sp(2p;\RR)/U(p)$ which parametrizes marked symplectic lattices in $\RR^{2p}$ is a symmetric space of rank $p$.
\end{itemize}
All of these spaces are CAT(0), meaning that any two geodesic rays emanating from a point will diverge at least as quickly as two rays in Euclidean space which make the same angle (see \cite{bh} for a thorough exposition of CAT(0) spaces).  This allows us to fill any loop of length $\gamma$ with a disk of area $O(\ell(\gamma)^{2})$ by choosing a basepoint $x_{0}$ in $S^{1}$ and mapping each geodesic segment in $D^{2}$ connecting $x_{0}$ to another point of the boundary $x$ to the geodesic segment from $\gamma(x_{0})$ to $\gamma(x)$.  By the CAT(0) property, this map is $O(\ell(\gamma)^{2})$ Lipschitz, and hence has quadratic area.

A lattice in a symmetric space $X=G/K$ is a discrete subgroup $\Gamma\subset G$ such that the quotient of $X$ by $\Gamma$ has finite volume (any such $\Gamma$ will act properly on $X$).  If $\Gamma$ is not commensurable with a direct product of subgroups, it is said to be irreducible. If the quotient $X/\Gamma$ is compact, then $\Gamma$ of course has quadratic Dehn function, but when $X/\Gamma$ is not compact, it is possible for $\Gamma$ to have larger Dehn function. For instance, $\SL(3;\ZZ)$, which is a lattice in the rank $2$ symmetric space $\SL(3;\RR)/\SO(p)$, has exponential Dehn function \cite[\S 10.4]{wordproc}.  Nonetheless, Gromov's study of the geometry of symmetric spaces lead to the following conjecture (a rough equivalent of this conjecture is implicit in \cite[5.d5c]{gromov}, as we explain below.)

\begin{conjecture}[Gromov]
Let $X$ be a symmetric space of noncompact type.  If $X$ has rank at least $3$, then any irreducible lattice $\Gamma\subset\Isom(X)$ has polynomial Dehn function.
\end{conjecture}

Since the groups $\SL(p;\ZZ)$ and $\Sp(2p;\ZZ)$ are lattices in the symmetric spaces $\SL(p;\RR)/SO(p)$ and $\Sp(2p;\RR)/U(p)$ respectively, Theorems \ref{theorem:MAINTHM} and \ref{theorem:RY} verify cases of this conjecture.

\paragraph{Horospheres and the thick part.}
Roughly speaking, the conjecture is believable because if $\Gamma$ is a lattice in a higher rank symmetric space $X$, then $\Gamma$ acts cocompactly on a subspace $X_{thick}\subset X$ (called the thick part of $X$), and given a loop $\gamma$ in $X_{thick}$, one ought to be able to ``push" a filling disk for $\gamma$ in $X$ back into the thick part.  We will now make this precise; references include \cite{bh} and \cite{borelji}.\\

First, we must recall the theory of horoballs in CAT(0) spaces.  If $X$ is a CAT(0) Riemannian manifold, one can compactify $X$ by adding a ``point at infinity" for each asymptotic equivalence class of geodesic rays in $X$ (two geodesic rays $c,\tilde{c}:[0,\infty)\rightarrow X$ are said to be asymptotic if $d(c(t),\tilde{c}(t))=O(1)$).  For instance, if $X$ is two dimensional Euclidean space, then two rays will be asymptotic if and only if they are parallel, so we add one point for each direction.  The added points form a boundary at infinity $\partial X$ for $X$.  Remarkably, Busemann discovered a simple notion of ``distance" between a point $\xi\in \partial X$ and a point $x\in X$.  Given a geodesic ray $c:[0,\infty)\rightarrow X$ representing $\xi \in \partial X$, define the Busemann function $b_{c}:X\rightarrow \RR$ by
$$b_{c}(x)=\lim_{t\rightarrow \infty} d(x,c(t))-t.$$
This function describes a sort of distance between $x$ and $\xi$.  Different choices of $c$ representing the same $\xi$ will yield Busemann functions which differ by some constant function.  Sublevel sets of Busemann functions are called horoballs and level sets are called horospheres.

{\it Example:} In $\RR^{2}$, the Busemann function associated to the geodesic ray $c:t\mapsto (at,bt)$ is given by $b_{c}:(x,y)\mapsto -ax-by$.  Hence, horoballs are half spaces, and horospheres associated to $c$ are just lines perpendicular to $c$.  This picture generalizes in the obvious way to higher dimensions.\\

We care about horoballs because if $\Gamma$ is a lattice in a symmetric space $X$, then $\Gamma$ acts cocompactly on a subspace $X_{thick}\subset X$ given by removing a union of horoballs from $X$.  The geometric observation which led Gromov to his conjecture was that the 5-dimensional Lie group
$$\Sol^{5}=\left\{ \begin{bmatrix}
e^{a} & 0 & 0 & x\\
0 & e^{b} & 0 & y\\
0 & 0 & e^{c} & z\\
0 & 0 & 0 & 1\\
\end{bmatrix}:a,b,c,x,y\in\RR; a+b+c=0\right\}$$
has quadratic Dehn function when equipped with a left invariant metric, while the three dimensional Lie group
$$\Sol^{3}=\left\{ \begin{bmatrix}
e^{a} & 0 & x\\
0 & e^{-a} & y\\
0 & 0 & 1\\
\end{bmatrix}:a,x,y\in\RR; a+b=0\right\}$$
has exponential Dehn function.  These facts are suggestive because $\Sol^{5}$ is quasi-isometric to a horosphere in the rank three symmetric space $(H^{2})^{3}$, while $\Sol^{3}$ is quasi-isometric to a horosphere in the rank two symmetric space $(H^{2})^{2}$.

\paragraph{Pushing filling disks into the thick part.}
If $\Gamma$ is a lattice in a symmetric space $X$, and $w$ is a word in the generators of $\Gamma$, then we can easily convert $w$ into a loop $\gamma:S^{1}\rightarrow X$, which can be taken to have image in $X_{thick}$.  If horospheres in $X$ have quadratic Dehn function, and $f:D^{2}\rightarrow X$ is a filling disk for $\gamma$, then we might hope to push $f$ into the thick part, replacing the region inside any horoball with a region on the boundary horosphere, as in the diagram.  In \cite[5d5c]{gromov}, Gromov suggests that, in light of an earlier claimed proof of Young's theorem, this program must work for $\SL(n;\ZZ)$, and hence that one should expect the Dehn functions of horospheres to control the Dehn functions of lattices. 

\begin{figure}[t]
\labellist
\small\hair 2pt
\pinlabel $\text{horosphere}$ at 36 154
\pinlabel $\text{filling disk}$ at 245 205
\pinlabel $\text{horoball}$ at 30 215
\pinlabel $\text{loop}$ at 152 0
\pinlabel $\triangleleft -$ at 92 0
\pinlabel $-\triangleright$ at 212 0
\endlabellist

\centering
\centerline{\psfig{file=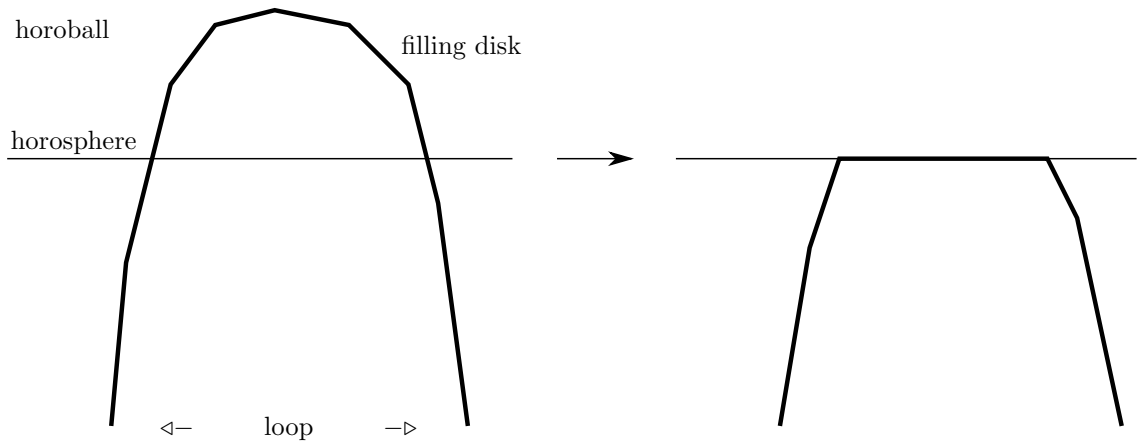,scale=70}}
\caption{Pushing a filling out of a horoball.}
\label{figure:horosphere}
\end{figure}

\paragraph{Known results}
In order to carry out this program, one would first like to know that horospheres in higher rank symmetric spaces have quadratic Dehn function.  This was proved (under some necessary assumptions) by Dru\c{t}u.
\begin{theorem}[Dru\c{t}u \cite{drutu}]
If $X$ is a symmetric space of rank at least $3$, and $c:[0,\infty)\rightarrow X$ is a ray not contained in any rank $1$ or rank $2$ factor of $X$, then any horosphere corresponding to $c$ has quadratic Dehn function.
\end{theorem}
Dru\c{t}u used this theorem to verify Gromov's conjecture whenever the lattice $\Gamma$ has $\QQ$-rank $1$.  For lattices of higher $\QQ$-rank, horoballs can overlap in interesting ways, which makes it unclear how to carry out the procedure outlined above. The most important progress in this direction is Young's Theorem \ref{theorem:RY}, which proves that $\SL(p;\ZZ)$ (a lattice in the rank $p-1$ symmetric space $\SL(p;\RR)/SO(p)$) has quadratic Dehn function for $p$ at least $5$.  We must also mention a powerful result of Bestvina, Eskin, and Wortman\cite[Corollary 5]{BEW} which shows a polynomial Dehn function for groups such as $\SL(n;\mathcal{O}_{K})$ or $\Sp(2n;\mathcal{O}_{K})$ where $\mathcal{O}_{K}$ is the ring of integers of a number field $K$ which has at least $3$ archimedean valuations (these groups are lattices in higher rank symmetric spaces.)

\subsection{Outline of proof.}
\label{subsection:RYtheorem}
The goal of this paper is to prove Theorem \ref{theorem:MAINTHM}, which states that $\Sp(2p;\ZZ)$ has quadratic Dehn function when $p$ is sufficiently large.  Although we follow the same basic strategy as \cite{RY}, which proves the analogous result for $\SL$, the proof is in many ways more difficult.  We have tried to summarize the differences near the end of this subsection.  We will now outline the strategy used to prove these theorems, explaining at each step how the proof is different in the symplectic case.

\tikzstyle{block} = [rectangle, draw, node distance=4cm, 
    text width=7em, text centered, rounded corners, minimum height=4em]
\tikzstyle{line} = [draw, -latex']
\tikzstyle{cloud} = [draw, rectangle, node distance=4cm,
    text width=5em, text centered, minimum height=2em]
 
\begin{tikzpicture}[node distance = 2cm, auto]
    \node [cloud] (gammaword) {word in $\Gamma$};
    \node [block, right of=gammaword] (precisereductiontheory) {precise reduction theory};
    \node [cloud, right of=precisereductiontheory] (parab) {product of words in parabolic subgroups};
    \node [block, below of=parab] (comb) {combinatorial techniques};
    \node [cloud, left of=comb] (diag) {product of words in diagonal blocks};
    \node [cloud, below of = comb] (min) {product of words representing elements of minimal diagonal blocks};
    \node [block, left of = min] (fill) {geometric filling};
    \path [line] (gammaword) -- (precisereductiontheory);
    \path [line] (precisereductiontheory) -- (parab);
    \path [line] (parab) -- (comb);
    \path [line] (comb) -- (diag);
    \path [line] (comb) -- (min);
    \path [line] (min) -- (fill);
    \path [line] (diag) -- (precisereductiontheory);
\end{tikzpicture}

\paragraph{The basic procedure.}
The flowchart illustrates the basic procedure used to find fillings of words in $\Gamma$ (where $\Gamma$ is $\SL(p;\ZZ)$ or $\Sp(2p;\ZZ)$).  Beginning with a word in $\Gamma$, one repeats the following procedure for a bounded number of steps.
\begin{itemize}
\item[1:] Using precise reduction theory, reduce a word in $\Gamma$, or a word in a non-minimal diagonal block in $\Gamma$, to a product of normal form words representing elements of parabolic subgroups of $\Gamma$.
\item[2:] Using combinatorial techniques, reduce a normal form word representing an element of a parabolic subgroup to a product of words in non-minimal diagonal blocks and normal form words representing elements of minimal diagonal blocks.
\end{itemize}
At the end of this procedure, one is left with a product of normal form words representing elements of minimal diagonal blocks, which can be filled by a simple geometric argument.

\paragraph{Precise reduction theory.}
Let $G=\SL(p;\RR)$ or $\Sp(2p;\RR)$, take $K$ to be a maximal compact subgroup of $G$, and let $X$ be the symmetric space $G/K$.  Any word $w$ in the generators of $\Gamma$ corresponds to a loop in $X_{thick}$, which has a Lipschitz filling disk in $X$.  If we triangulate this disk in an appropriate way, we can obtain an expression $w=\prod w_{i}$ where each $w_{i}$ is either a word representing an element of a parabolic subgroup or a word of bounded length, and $\sum \ell(w_{i})^{2}$ is at most quadratic in $\ell(w)$.  Recall that a parabolic subgroup is just the stabilizer of some point in $\partial X$.  The triangulation of the filling disk in $X$ is chosen so that the vertices of each triangle are simultaneously close to some point in $\partial X$.  This can be used to produce the desired expression for $w$.

\paragraph{Shortcuts.}
When $\Gamma=\SL(p;\ZZ)$, a parabolic subgroup of $\Gamma$ is (up to conjugacy) just a block upper triangular subgroup.  When $\Gamma=\Sp(2p;\ZZ)$, there is also a ``symplectic block", see \S\ref{section:preliminaries}.  In order to efficiently represent elements of parabolic subgroups, it is necessary to find efficient representations for large powers of unipotent matrices.  This problem was solved by Lubotzky, Mozes, and Raghunathan \cite{lmr}, and the words representing powers of unipotent matrices are called shortcuts.  One then defines a normal form $\Omega$ for elements of parabolics, so that $\Omega(M)$ is a product of a word representing each diagonal block of $M$ followed by a product of a bounded number of shortcuts representing the unipotent part of $M$.

\paragraph{Filling relations involving shortcuts.}  
In both the $\SL$ and $\Sp$ cases, one can fill all Steinberg relations involving shortcuts, and almost all relations between words in diagonal blocks and shortcuts in the unipotent radical.  In the $\SL$ case, shortcuts live in certain solvable subgroups of $\SL(p;\ZZ)$ of the form $\ZZ^{d_{1}}\ltimes\ZZ^{d_{2}}$, and to fill the necessary relations involving shortcuts, one must know that these solvable subgroups have quadratic Dehn function (this was already known, due to work of Dru\c{t}u\cite{drutu}, Leuzinger and Pittet\cite{lp}, and Cornulier and Tessera\cite{dct}).  In the $\Sp$ case, the analogous subgroups have the form $\ZZ^{d}\ltimes N$ where $N$ is two-step nilpotent, and it is harder to show that the required groups have quadratic Dehn function.  In the $\SL$ case, these techniques fail only in the case where a diagonal block has ``codimension one" \cite[\S 8.2]{RY}, but in the $\Sp$ case, they fail when the $\Sp$ block has codimension less than $3$ or more than $p-1$.  Handling these difficulties is the hardest part of this paper.\\

\paragraph{Lipschitz fillings.}
When $\Gamma=\SL(p;\ZZ)$, filling the relations alluded to above is usually enough to reduce a normal form triangle in a parabolic subgroup to a product of words in its diagonal blocks.  However, for a parabolic $P$ which has a diagonal block of size $p-1\times p-1$, these techniques fail, and one instead must shows that reduction theory techniques can be applied to the space $X=P/\SO(p-1)$, hence reducing to the case of smaller parabolic subgroups.  This requires showing that $X$ is Lipschitz 1-connected (see \cite{RYlip}), i.e., that any $\ell$-Lipschitz loop has a $O(\ell)$-Lipschitz filling, which strengthens Dru\c{t}u's result that $X$ has quadratic Dehn function (actually, \cite[Lemma 8.8]{RY} only shows that every $\ell$-Lipschitz loop has a $O(\ell+1)$-Lipschitz filling, but in a homogeneous space this property is manifestly equivalent to Lipschitz 1-connectedness).\\

When $\Gamma=\Sp(2p;\ZZ)$ we must handle analogous difficulties for any parabolic where the $\Sp$ block has codimension less than three, or when it is trivial.  In the case where the $\Sp$ block of a parabolic $P$ is trivial, we show that $P/\SO(p)$ is Lipschitz 1-connected (again strengthening Dru\c{t}u's theorem).  When $P$ is a parabolic with $\Sp$ block $\Sp(2q;\ZZ)\subset P$ of codimension less than three, we show that $\Sp(2q;\RR)\ltimes N/U(q)$ is Lipschitz 1-connected, where $N$ is the unipotent radical of $P$.  This is enough to run the precise reduction theory machinery as in the other case.  Our Lipschitz connectivity results are significantly harder to prove than the ones needed in the $\SL$ case.

\paragraph{Differences between $\SL$ and $\Sp$.} We give here a brief list of the difficulties that arise in generalizing from the $\SL$ case.
\begin{itemize}
\item In $\SL$, all elementary generators are conjugate.  In $\Sp$, this is no longer the case, see \S\ref{section:elementarysymplecticmatrices}.
\item In $\SL$, the unipotent radical of a maximal parabolic is abelian.  In $\Sp$, it is usually 2-step nilpotent.
\item In $\SL$, the maximal parabolic preserving a $k$-dimensional subspace is isomorphic to the maximal parabolic preserving a codimension $k$ subspace.  In $\Sp(2n)$, the maximal parabolic preserving a $k$-dimensional isotropic subspace is not isomorphic to the maximal parabolic preserving an $n-k$ dimensional isotropic subspace.
\item Similarly, the group $H_{S,T}$ considered in \cite{RY} has quadratic Dehn function whenever $S$ or $T$ has size at least three.  The analagous group $SpH_{S,T}$ considered here has quadratic Dehn function only when $S$ has size at least three, see Theorem \ref{theorem:solvable2}.
\end{itemize}

\paragraph{Outline.}
We now describe the organization of this paper.  The basic strategy employed to prove our main theorem is outlined in \S\ref{section:theorem}, which reduces the task of proving Theorem \ref{theorem:MAINTHM} to that of proving several other theorems.  In order to understand this machinery, the reader will need to read the preliminaries described in \S\ref{section:preliminaries}, which explains our notation along with some important facts about the structure of $\Sp(2p;\ZZ)$.  \S\ref{section:shortcuts} explains shortcuts, and proves all the important facts about them, modulo the results of \S\ref{section:solvable}.  \S\ref{section:diagtoparab} explains the results we need from precise reduction theory.  This allows us to prove that words in diagonal blocks (including the block consisting of $\Sp(2p;\ZZ)$ itself) can be broken into products of normal form triangles in parabolic subgroups.  \S\ref{section:parabtodiag} then proves that normal form triangles in parabolics can be reduced to words in diagonal sub-blocks, modulo results proved in \S\ref{section:lipfill}.

The remaining sections are devoted to proving the technical results which power our machinery, and hence contain most of the new ideas needed to go from $\SL$ to $\Sp$.  \S\ref{section:solvable} shows that certain solvable subgroups have quadratic Dehn function, justifying the lemmas of \S\ref{section:shortcuts}.  \S\ref{section:lipfill} shows that certain homogeneous spaces are Lipschitz 1-connected, and thus justifies the use of adaptive templates in \S\ref{section:parabtodiag}.

\paragraph{Acknowledgements}
We wish to thank Robert Young for explaining his results to us.  We also wish to thank Justin Martel for our discussions of the symplectic group, and to thank Yves Cornulier, Mark Sapir, and Kevin Wortman for their helpful comments on a draft version of this paper.  Finally, we wish to thank Andy Putman for his many expositional suggestions.
\section{Main Theorems}
\label{section:theorem}
We will presently reduce our main theorem to Theorems \ref{theorem:parabtodiag}, \ref{theorem:diagtoparab}, and \ref{theorem:sp2shortcut}, which will be proved in other sections. To state these theorems, we first need to recall the definition of the standard maximal parabolic subgroups of the symplectic group along with some notions related to filling.

\subsection{Symplectic linear algebra}
\label{section:symplecticvectorspaces}
Here we briefly recall the theory of symplectic linear algebra.  (An inspired reference is \cite{mcduff}).
\begin{definition}
The symplectic form $\omega$ is the bilinear form on $\RR^{2p}$ given by
$$\omega(v,w)=v^{T}J_{0}w$$
where we view $v$ and $w$ as column vectors and $J_{0}$ is the $2p\times 2p$ matrix
$$\begin{bmatrix} 0 & \Id_{p\times p}\\ -\Id_{p\times p} & 0\end{bmatrix}.$$
\end{definition}

\begin{definition}
The symplectic group $\Sp(2p;\RR)$ is the set of matrices $M\in\GL(2p;\RR)$ satisfying $M^{T}J_{0}M=J_{0}$.
\end{definition}

Note that $M\in\Sp(2p;\RR)$ preserves the symplectic form in the sense that $\omega(Mv,Mw)=\omega(v,w)$.
\begin{definition}
If $W\subset \RR^{2p}$ is a subspace, define $W^{\omega}$ to be
$$\{v\in\RR^{2p}|\omega(v,W)=0\}.$$
We call a subspace $W\subset \RR^{2p}$ isotropic if $W\subset W^{\omega}$, coisotropic if $W\supset W^{\omega}$, and Lagrangian if $W=W^{\omega}$.  We call $W$ symplectic if $W\cap W^{\omega}=0$.
\end{definition}

We label the standard basis of $\RR^{2p}$ by half roots (see \S\ref{section:roots}) as follows.  For $i=1,\ldots,p$, let $z_{[i]}$ be a vector with zeros except in the $i$-th position, where it has a $1$.  Similarly, let $z_{-[i]}$ have all zeros except in the $p+i$-th position, where it has a $1$.

Given $S\subset \mathcal{H}=\{\pm [i]:1\leq i\leq p\}$, let $\RR^{S}$ denote the span of $\{z_{s}\}_{s\in S}$. We say that $S$ is isotropic if $\RR^{S}$ is isotropic, and we say that $S$ is symplectic if $\RR^{S}$ is symplectic (see \S\ref{section:roots}).  Given disjoint $S,T\subset\mathcal{H}$ with $S$ isotropic and $T$ symplectic, we define $P_{S,T}(\ZZ)\subset\Sp(2p;\ZZ)$ to be the subgroup which preserves $\RR^{S}$ and fixes $z_{s}$ for $s\notin S\cup -S\cup T$. We further define $\GL(S;\ZZ)\subset P_{S,T}$ to be the subgroup which acts trivially on $\RR^{T}$, and $\Sp(T;\ZZ)\subset P_{S,T}$ to be the subgroup $\RR^{S}$. $\GL(S;\ZZ)$ and $\Sp(T;\ZZ)$ are called diagonal blocks, and $P_{S,T}$ is called a maximal parabolic subgroup of $\Sp(S\cup -S\cup T;\ZZ)$. All of these subgroups are explained in much greater detail in \S\ref{section:subgroups}.

\subsection{Quadratic breaking}
To find a filling for some relation $w$ in the generators of $\Sp(2p;\ZZ)$, we often proceed by finding a sequence of relations $w=w_{0},w_{1},\ldots,w_{n}$ with $w_{n}$ equal to the empty word and then finding fillings for all of the $w_{i}^{-1}w_{i}$. We now introduce language and notation based on this idea which will be used to state our theorems.

Suppose $\Gamma$ a group with a finite presentation $\langle \mathcal{S}|\mathcal{R}\rangle$ where $\mathcal{S}$ is symmetric (so that $s\in \mathcal{S}\Leftrightarrow s^{-1}\in\mathcal{S}$). Let $\mathcal{S}^{\ast}$ denote the set of words in $\mathcal{S}$. Given $w,w^{\prime}\in\mathcal{S}^{\ast}$, we write their concatenation as $ww^{\prime}$.  A relation is a word $w\in\mathcal{S}^{\ast}$ which represents the identity in $\Gamma$.
\label{section:quadraticbreaking}
\begin{definition}
A homotopy of cost $C$ between two words $w$ and $w^{\prime}$ in $\mathcal{S}^{\ast}$ is a sequence
$$w=w_{0},w_{1},\ldots,w_{n}=w^{\prime}$$ where each $w_{i}$ is related to the next by one of the following operations, and there are only $C$ uses of the first two operations (insertion or deletion of a relator).
\begin{itemize}
\item Insertion of a relator: $w_{i}=uv$ and $w_{i+1}=urv$ for some $u,v\in \mathcal{S}^{\ast}$ and $r\in \mathcal{R}$.
\item Deletion of a relator: $w_{i}=urv$ and $w_{i+1}=uv$ for some $u,v\in \mathcal{S}^{\ast}$ and $r\in \mathcal{R}$.
\item Free expansion: $w_{i}=uv$ and $w_{i+1}=ut^{-1}tv$ for some $u,v\in \mathcal{S}^{\ast}$ and $t\in \mathcal{S}$.
\item Free contraction: $w_{i}=ut^{-1}tv$ and $w_{i+1}=uv$ for some $u,v\in \mathcal{S}^{\ast}$ and $t\in \mathcal{S}$.
\end{itemize}
\end{definition}

Note that each $w_{i}$ necessarily represents the same element. We remark that a homotopy of cost $C$ between $w$ and $w^{\prime}$ can easily be converted to an area $C$ filling of the relation $w^{-1}w^{\prime}$ (i.e., an expression of $w^{-1}w^{\prime}$ as a product of $C$ conjugates of relators).

\begin{definition}
Given $w\in F(\mathcal{S})$ which represents the identity in $G$, we say that $w$ can be broken into relations $\{w_{i}\}_{i=1,\ldots,n}$ in $\mathcal{S}^{\ast}$ at cost $C$ if there is a homotopy of cost $C$ from $w$ to a product $\prod_{i=1}^{n}v_{i}w_{i}v_{i}^{-1}$ for some $v_{i}\in \mathcal{S}^{\ast}$.
\end{definition}

\begin{definition}
Let $\mathcal{C}\subset\mathcal{S}^{\ast}$ be some class of words representing the identity.  When we say that any word $w\in\mathcal{C}$ can be quadratically broken into relations $\{w_{i}\}$ with some property, the meaning is that $w$ can be broken at cost $O(\ell(w)^{2})$ into relations $\{w_{i}\}$ which have the prescribed property and, additionally,$$\sum\ell(w_{i})^{2}=O(\ell(w)^{2}).$$
\end{definition}

Similarly, suppose we are given some operation on words $f:\mathcal{C}\rightarrow \mathcal{S}^{\ast}$ defined on a class of words $\mathcal{C}$, and suppose further that $w$ and $f(w)$ always represent the same element of $\Gamma$. Then we write
$$w\leadsto f(w)$$
if $\ell(f(w))=O(\ell(w))$ and any word $w\in \mathcal{C}$ can be homotoped to $f(w)$ at cost $O(\ell(w)^{2})$.  More generally, we write
$$\delta(w,f(w))=g(w)$$
for some function $g:\mathcal{C}\rightarrow\RR$ if we can always find a homotopy of cost $g(w)$ from any $w\in \mathcal{C}$ to $f(w)$. If every element of $\mathcal{C}$ represents the identity, the expression
$$\delta(w)=g(w)$$
means that every $w\in\mathcal{C}$ has a filling of area $g(w)$.

\begin{definition}
A normal form is a map $\Omega:G\rightarrow \mathcal{S}^{\ast}$ such that $\Omega(g)$ represents $g$ for any $g\in G$. We always wish for a normal form to be efficient, in the sense that $\ell(w)=O(\ell(\Omega(g)))$ for any $w$ representing $g$, unless otherwise indicated.
\end{definition}

\subsection{Generating sets}
Just as the so-called elementary matrices generate $\SL(p;\ZZ)$, there is a natural generating for $\Sp(2p;\ZZ)$ consisting of matrices which we call elementary symplectic generators (denoted $e_{\alpha}$, where $\alpha$ is a root as described in \S\ref{section:roots}). These generators and their properties are described in \S\ref{section:elementarysymplecticmatrices}, but for convenience we will use a larger generating set. Throughout this section, and the rest of the paper, we fix a finite generating set $\SSp$ for $\Sp(2p;\ZZ)$ which contains every elementary symplectic generator and also contains finite generating sets for the subgroups discussed later in \S\ref{section:subgroups}, i.e., $\SSp$ contains a generating set for each subgroup of the form $\GL(S;\ZZ)$, $\Sp(T;\ZZ)$, $N_{S,T}(\ZZ)$, 
 or $\SpH_{S,T}(\ZZ)$ contained in $\Sp(2p;\ZZ)$. This assumption is harmless because there are only finitely many such subgroups.

\subsection{Reduction of the main theorem}
Recall that the following is the main theorem of this paper (Theorem \ref{theorem:MAINTHM}).

\begin{nolabeltheorem}
For $p\geq 5$, the group $\Sp(2p;\ZZ)$ has quadratic Dehn function.
\end{nolabeltheorem}

We prove this by developing a normal form $\Omega=\Omega_{P_{S,T}}$ (in \S\ref{section:shortcuts}) for each parabolic $P_{S,T}(\ZZ)$, and then (in subsequent sections) showing the theorems stated after the following definition.
\begin{definition}
\label{definition:shortcutword}
\begin{itemize}
\item Suppose $T\subset \mH$ is symplectic.  If we say that $w$ is a word in $\Sp(T)$, we mean that it is a word in the finite alphabet $\{e_{\alpha}\in\Sp(T;\ZZ)\}$.  (Similarly for words in $\SL(S)$ where $S\subset\mH$ isotropic.) Here $e_{\alpha}$ is an elementary symplectic generator as defined in \S\ref{section:elementarysymplecticmatrices}.
\item A shortcut word in a subgroup $H\subset \Sp(2p;\ZZ)$ is a product of words $\he_{\alpha_{1}}(x_{1})\ldots\he_{\alpha_{n}}(x_{n})$ where each $\he_{\alpha_{i}}(x_{i})$ is a ``shortcut" for some $e_{\alpha_{i}}(x_{i})\in H$. A shortcut $\he_{\alpha}(x)$ is a special type of word in $\SSp$ representing the elementary symplectic matrix $e_{\alpha}(x)$. Shortcuts are defined in \ref{section:shortcuts}, and elementary symplectic matrices are defined in \S\ref{section:preliminaries}. The reader is cautioned that a shortcut word in $\Sp(T)$ is not generally a word in $\Sp(T)$ (the shortcuts involved represent elements of $\Sp(T)$, but are not words in the generators of $\Sp(T;\ZZ)$).
\item An $\Omega$ triangle in a parabolic $P_{S,T}$ is a product of words
 $\Omega(p_{1})\Omega(p_{2})\Omega(p_{3})$ where $p_{1},p_{2},p_{3}\in P_{S,T}(\ZZ)$ and $p_{1}p_{2}p_{3}=1$ (here $\Omega$ is understood to be $\Omega_{P_{S,T}}$).
\end{itemize}
\end{definition}

\begin{theorem}
\label{theorem:diagtoparab}
Suppose $T\subset \mH$ is symplectic with $4\leq \# T \leq 2p$.  If $w$ a relation in $Sp(T)$, then $w$ can be broken into a collection of relations $w_{1},\ldots,w_{n}$ such that each $w_{i}$ is an $\Omega$ triangle in some maximal parabolic $P_{S^{\prime},T^{\prime}}$ of $\Sp(T)$.  
\end{theorem}

Theorem \ref{theorem:diagtoparab} is proved in \S\ref{section:diagtoparab}.

\begin{theorem}
\label{theorem:parabtodiag}
Suppose $\emptyset\neq S\subset \mH$ is isotropic, and $T\subset \mH$ is symplectic with $S,T$ disjoint. Let $\Delta$ be an $\Omega$-triangle in $P_{S,T}$.  Then we can homotope $\Delta$ as follows (note that all these homotopies have quadratic cost).
\begin{itemize}
\item[(a)] If $\# T > 2(p-3)$, then $\Delta$ can be quadratically broken into relations $w_{i}$ with each $w_{i}$ an $\Omega$ triangle in some $P_{S_{i},T_{i}}$ with $\# T_{i}$ strictly smaller than $\# T$.
\item[(b)] If $4\leq \# T \leq 2(p-3)$, then $\Delta$ can be homotoped at cost $O(\ell(\Delta)^{2})$ to a relation of length $O(\ell(\Delta))$ in $\Sp(T)$.
\item[(c)] If $\# T = 2$, then $\Delta$ can be homotoped at cost $O(\ell(\Delta)^{2})$ to an identity-representing shortcut word in $\Sp(T)$ of length $O(\ell(\Delta))$.
\item[(d)] If $\# T = 0$, then $\Delta$ can be filled at cost $O(\ell(\Delta)^{2})$.
\end{itemize}
\end{theorem}

Theorem \ref{theorem:parabtodiag} is proved in \S\ref{section:parabtodiag}.

\begin{theorem}
\label{theorem:sp2shortcut}
If $w$ is an identity-representing shortcut word in $\Sp(T)$ where $\#T=2$, then $w$ can be filled at cost $O(\ell(w)^{2})$.
\end{theorem}

\begin{proof}
The proof of theorem \ref{theorem:sp2shortcut} is omitted because \cite[Lemma 3.5]{RY} proves the same result for a diagonal block $\SL(2;\ZZ)\subset\SL(p;\ZZ)$, and the same proof applies almost verbatim (one uses our Lemma \ref{lemma:steinberg} in place of \cite[Lemma 7.6]{RY}).
\end{proof}

We now explain how these theorems combine to prove the main theorem (see the flow chart in \S\ref{subsection:RYtheorem}).  Suppose we begin with a relation $w$ in $\Sp(2p;\ZZ)$. We will apply the procedure described below for a finite number of iterations. Each cycle consumes a product of relations $\Pi=w_{1}\ldots w_{N}$ and produces a homotopy of cost $O(\ell(\Pi)^{2})$ to a product of relations $\Pi^{\prime}=w_{1}^{\prime}\ldots w_{N^{\prime}}^{\prime}$ with $\sum \ell(w_{i}^{\prime})^{2}=O(\sum \ell(w_{i})^{2})$ where the relations involved in $\Pi^{\prime}$ have, in some sense, lower complexity. Whenever a relation has the minimal possible complexity, it can be filled. Because the complexity can only assume a finite number of possible values, the procedure can repeat only a finite number of cycles before we are left with the empty word.

\begin{itemize}
\item[1] Consider a relation of length $\ell$ in $\Sp(2p;\ZZ)$. Apply theorem \ref{theorem:diagtoparab} to quadratically break this relation into a product of $\Omega$-triangles in parabolics $P_{S,T}$ with each $T$ having strictly less than $2p$ elements. Note that this homotopy has cost $O(\ell^{2})$ and the sum of $\ell(\Delta)^{2}$ as $\Delta$ ranges over these $\Omega$-triangles is $O(\ell^{2})$.

\item[2] Suppose we have a product of $\Omega$ triangles in various parabolics $P_{S,T}$ with total squared length $O(\ell^{2})$. We break each triangle $\Delta$ in this product into a product of relations in diagonal blocks or identity-representing shortcut words in diagonal blocks as follows.

Let $P_{S,T}$ be the parabolic associated to $\Delta$.  If $\#T=0$, then $\Delta$ can be filled at cost $O(\ell(\Delta)^{2})$ by part (d) of Theorem \ref{theorem:parabtodiag}.  If $\#T=2$, $\Delta$ can be reduced to a product of identity-representing shortcut words in $\Sp(T)$ by Theorem \ref{theorem:parabtodiag}.  If $4\leq \#T\leq 2(p-3)$, $\Delta$ can be reduced to a product of relations in $\Sp(T)$ by part (c) of theorem \ref{theorem:parabtodiag}.  If $\#T>2(p-3)$, then at most three applications of Theorem \ref{theorem:parabtodiag} part (a) produces a homotopy of cost $O(\ell(\Delta)^{2})$ to a product of $\Omega$ triangles in parabolics $P_{S,T^{\prime}}$ with each $T^{\prime}$ of size at most $2(p-3)$ and total squared length $O(\ell(\Delta))^{2}$. These $\Omega$ triangles can then be handled as in the previous cases.

In each case, the given theorem visibly provides either a filling of cost $O(\ell(\Delta)^{2})$ or homotopy of cost $O(\ell(\Delta)^{2})$ to a product
$w_{1}^{\Delta}\ldots w_{N}^{\Delta}$
where each $w_{i}^{\Delta}$ is either a relation in some $\Sp(T)$ with $\#T\geq 4$ or an identity-representing shortcut word in some $\Sp(T)$ with $\#T=2$, and
$$\sum \ell(w_{i})^{2}=O(\ell(\Delta)^{2}).$$
It follows that the total cost of applying these manipulations to all of the $\Delta$ in our original product is $O(\ell^{2})$ and the resulting product of relations and identity-representing shortcut words has total squared length $O(\ell^{2})$. Furthermore, if $T_{0}$ is the largest $T$ of any parabolic $P_{S,T}$ involved in our original product, then all of the symplectic blocks $\Sp(T)$ involved in our final product have size at most $\# T_{0}$.

\item[3] Suppose we are given a product like that produced by Step 2, i.e., a product of relations in symplectic diagonal blocks $\Sp(T)$ (with $\#T\geq 4$) and identity-representing shortcut words in symplectic diagonal blocks $\Sp(T)$ (with $\#T=2$) such that the sum of the squares of their lengths is $O(\ell^{2})$. We will now produce a homotopy of cost $O(\ell^{2})$ to a product of $\Omega$-triangles in parabolics $P_{S,T}$ such that the sum $\sum \ell(\Delta)^{2}$ as $\Delta$ ranges over these $\Omega$-triangles is $O(\ell^{2})$.

For any relation $w$ in a block $\Sp(T)$ such that $\#T>2$, we can apply Theorem \ref{theorem:diagtoparab} to homotope $w$ to a product of $\Omega$ triangles in parabolics $P_{S,T^{\prime}}$ such that, crucially, $\#T^{\prime}$ is strictly smaller than $\#T$. Furthermore, the sum of the squares of the lengths of these $\Omega$ triangles is $O(\ell(w)^{2})$. If $w$ is an identity-representing shortcut word in $\Sp(T)$ with $\#T=2$ it can be filled at cost $O(\ell(w)^{2})$ by Theorem \ref{theorem:sp2shortcut}.

By applying this to every $w$, we obtain a cost $O(\ell^{2})$ homotopy to a product of $\Omega$ triangles in parabolic subgroups $P_{S,T}$, with the sum of the squares of the lengths of these $\Omega$ triangles at most $O(\ell^{2})$, so that we can go back to step 2. Furthermore, if $T_{0}$ is the largest symplectic block involved in our original product, then every $P_{S,T}$ involved in our final product has $\#T$ strictly less than $T_{0}$. .
\end{itemize}

After each repetition of steps 2 and 3, the maximum size of the symplectic blocks involved goes down, hence we obtain a filling for our original relation after a finite number of steps, since symplectic blocks can only have sizes $2,4,\ldots,2p$.
\section{Preliminaries}
\label{section:preliminaries}
In this section we introduce our system of notation, explain the basic structure of $\Sp(2p;\ZZ)$ and its subgroups, and review the basic notions of filling in groups.

\subsection{Roots}
\label{section:roots}
We first discuss roots, a basic notion of Lie theory, as they pertain to the symplectic group.  We first learned the subject from \cite{fh}.  Mostly we will just use roots as labels for various elements and subgroups of $\Sp$, but \S\ref{section:solvable} requires a slightly more sophisticated understanding of weights, roots, and Lie algebras.

\begin{definition}
The Lie algebra $\fsp(2p;\RR)$ is the set of $2p\times 2p$ real matrices $M$ such that
$$M^{T}J_{0}+J_{0}M=0.$$
\end{definition}

For more general background on Lie algebras, we again recommend \cite{fh}.

\begin{definition}
\label{def:diag}
\begin{itemize}
\item For real numbers $a_{1},\ldots,a_{p}$, let $\diag_{\fsp}(a_{1},\ldots,a_{p})$ be the diagonal matrix
$$\begin{bmatrix} a_{1} & \ldots & 0 & 0 & \ldots & 0\\
\ldots & \ldots & \ldots & \ldots & \ldots & \ldots\\
0 & \ldots & a_{p} & 0 & \ldots & 0\\
0 & \ldots & 0 & -a_{1} & \ldots & 0\\
\ldots & \ldots & \ldots & \ldots & \ldots & \ldots\\
0 & \ldots & 0 & 0 & \ldots & -a_{p}\end{bmatrix}\in\fsp(2p;\RR).$$
\item Let $\Diag_{\fsp}=\{\diag_{\fsp}(a_{1},\ldots,a_{p}):a_{1},\ldots,a_{p}\in\RR\}$.  One refers to $\Diag_{\fsp}$ as a Cartan subalgebra of $\fsp(2p;\RR)$.
\item For real numbers $a_{1},\ldots,a_{p}>0$, let $\diag_{\Sp}(a_{1},\ldots,a_{p})$ be the diagonal matrix
$$\begin{bmatrix} a_{1} & \ldots & 0 & 0 & \ldots & 0\\
\ldots & \ldots & \ldots & \ldots & \ldots & \ldots\\
0 & \ldots & a_{p} & 0 & \ldots & 0\\
0 & \ldots & 0 & a_{1}^{-1} & \ldots & 0\\
\ldots & \ldots & \ldots & \ldots & \ldots & \ldots\\
0 & \ldots & 0 & 0 & \ldots & a_{p}^{-1}\end{bmatrix}\in\Sp(2p;\RR).$$
\item Let $\Diag_{\Sp}^{+}=\{\diag_{\Sp}(a_{1},\ldots,a_{p}):a_{1},\ldots,a_{p}\in(0,\infty)\}\subset \Sp(2p;\RR)$.
\item Observe that $\Diag_{\fsp}$ is a vector space.  Let $\{[1],\ldots,[p]\}$ denote a basis for its dual with $[i](\diag_{\fsp}(a_{1},\ldots,a_{p}))=a_{i}$.  Let $\mathcal{H}=\{\pm [i]:i=1,\ldots,p\}$.  We sometimes call $\mathcal{H}$ the set of half roots.  The literature appears not to contain any standard notation for $\mathcal{H}$ or its elements.  Typically we will denote half roots by $s$ or $t$ and subsets of $\mathcal{H}$ by $S$ or $T$.
\item Let $\Phi=\{s-t\neq 0:s,t\in H\}$.  An element of $\Phi$ is called a root.  We will generally denote elements of $\Phi$ by lower case Greek letters, or by expressions like $s-t$.  A root of the form $s-t$ where $t\neq \pm s\in \mathcal{H}$ is called short.  One of the form $2s$ is called long.
\end{itemize}
\end{definition}

Now we will use half roots and roots to label
\begin{itemize}
\item the standard basis of $\RR^{2p}$ (\S\ref{section:standardbasis}),
\item the elementary generators of $\Sp(2p;\ZZ)$(\S\ref{section:elementarysymplecticmatrices}),
\item certain copies of $\GL(q)$ and $\Sp(2q)$ contained in $\Sp(2p)$(\S\ref{section:subgroups}),
\item and maximal parabolic subgroups of these subgroups(\S\ref{section:subgroups}).
\end{itemize}

\subsection{The standard basis}
\label{section:standardbasis}
Recall that we have labeled the standard basis of $\RR^{2p}$ as
$$\{z_{[1]},\ldots,z_{[p]},z_{-[1]},\ldots,z_{-[p]}\},$$
so that for $D\in \Diag_{\fsp}$ we have $Dz_{[i]}=([i]D)z_{[i]}$ and $Dz_{-[i]}=(-[i]D)z_{-[i]}$.

Given $S\subset \mathcal{H}$, recall that $\RR^{S}$ denotes the span of $\{z_{s}\}_{s\in S}$, and we call $S$ isotropic (respectively, symplectic) when $\RR^{S}$ is isotropic (respectively, symplectic).  Note that these properties have easy combinatorial descriptions: $S$ is isotropic if it contains no pair $\{\pm s\}$ and symplectic if $s\in S\Rightarrow -s\in S$. Similarly, if $S,T\subset \mH$ with $T$ symplectic, then $\omega(\RR^{T},\RR^{S})=0$ if and only if $S$ and $T$ are disjoint.

If $T\subset \mathcal{H}$ is symplectic, let $T^{+}$ be the isotropic set $T\cap \{[1],\ldots,[p]\}$, and $T^{-}$ the isotropic set $T\cap\{-[1],\ldots,-[p]\}$.  If $S\subset \mathcal{H}$ is isotropic, let $\pm S$ denote the symplectic set $\{\pm s:s\in S\}$.

\subsection{Elementary symplectic matrices}
\label{section:elementarysymplecticmatrices}
There is a well known generating set of $\Sp(2p;\ZZ)$ consisting of matrices called elementary symplectic matrices.  Typically, these are divided into several ``types" and labeled by indices running from $1$ to $p$.  To avoid worrying about these different types, we shall instead label elementary symplectic matrices by roots.

Suppose $s,t\in \mathcal{H}$ with $s\neq \pm t$.  Then we define a matrix $e_{s-t}\in\Sp(2p;\RR)$ by mandating that for $v\in \RR^{2p}$, we have
$$e_{s-t}v=v+(z_{t}^{\text{T}}v)z_{s}+(z_{s}^{\text{T}}J_{0}v)J_{0}z_{t}.$$
For example, in $\Sp(6)$ we have that
$$e_{[1]-[2]}=
\begin{bmatrix}
1 & 1 & 0 & 0 & 0 & 0\\
0 & 1 & 0 & 0 & 0 & 0\\
0 & 0 & 1 & 0 & 0 & 0\\
0 & 0 & 0 & 1 & 0 & 0\\
0 & 0 & 0 & -1 & 1 & 0\\
0 & 0 & 0 & 0 & 0 & 1
\end{bmatrix}$$
and
$$e_{[1]+[2]}=
\begin{bmatrix}
1 & 0 & 0 & 0 & 1 & 0\\
0 & 1 & 0 & 1 & 0 & 0\\
0 & 0 & 1 & 0 & 0 & 0\\
0 & 0 & 0 & 1 & 0 & 0\\
0 & 0 & 0 & 0 & 1 & 0\\
0 & 0 & 0 & 0 & 0 & 1
\end{bmatrix}.$$

For a long root $\alpha=2s$, define $e_{\alpha}\in\Sp(2p;\RR)$ to be the matrix such that for $v\in \RR^{2p}$, we have $e_{\alpha}v=v+(z_{-s}^{\text{T}}v)z_{s}$.  For example, in $\Sp(6)$ we have
$$e_{2\cdot[1]}=
\begin{bmatrix}
1 & 0 & 0 & 1 & 0 & 0\\
0 & 1 & 0 & 0 & 0 & 0\\
0 & 0 & 1 & 0 & 0 & 0\\
0 & 0 & 0 & 1 & 0 & 0\\
0 & 0 & 0 & 0 & 1 & 0\\
0 & 0 & 0 & 0 & 0 & 1
\end{bmatrix}$$

We often write $e_{\alpha}(x)$ for the element $e_{\alpha}^{x}$.  We call $e_{\alpha}(x)$ an elementary symplectic matrix or just an elementary matrix.  We call $e_{\alpha}$ an elementary generator.  For example in $\Sp(6)$ we have the following.
$$e_{[1]-[2]}(-3)=
\begin{bmatrix}
1 & -3 & 0 & 0 & 0 & 0\\
0 & 1 & 0 & 0 & 0 & 0\\
0 & 0 & 1 & 0 & 0 & 0\\
0 & 0 & 0 & 1 & 0 & 0\\
0 & 0 & 0 & 3 & 1 & 0\\
0 & 0 & 0 & 0 & 0 & 1
\end{bmatrix}$$

We now record some key properties of elementary symplectic matrices, which the skeptical reader can verify either by multiplying matrices or meditating on Lie theory.  Hopefully these properties help to motivate our choice of notation.
\begin{itemize}
\item Elementary symplectic matrices are symplectic.
\item If $D=\diag_{Sp}(e^{a_{1}},\ldots,e^{a_{p}})$, then
$$De_{\alpha}(x)D^{-1}=e_{\alpha}(x^{\prime})$$
where $x^{\prime}=\exp(\alpha(\diag_{\fsp}(a_{1},\ldots,a_{p})))$.
\item $e_{\alpha}(x)e_{\alpha}(y)=e_{\alpha}(x+y)$.
\item If $\alpha+\beta$ is a root, then $[e_{\alpha}(x),e_{\beta}(y)]=e_{\alpha+\beta}(\kappa xy)$ where $\kappa\in \{-2,-1,1,2\}$.  The absolute value of $\kappa$ is $1$ if $\alpha+\beta$ is short and $2$ if $\alpha+\beta$ is long.  Its sign is reversed by switching $\alpha$ and $\beta$.
\item If $\alpha+\beta$ is not a root, and is not equal to zero, then $e_{\alpha}(x)$ commutes with $e_{\beta}(y)$.  The relationship between $e_{\alpha}$ and $e_{-\alpha}$ is more complicated (they generate a group isomorphic to $SL(2;\ZZ)$).
\item The set $\{e_{\alpha}:\alpha\in\Phi\}$ generates $\Sp(2p;\ZZ)$ (see \cite{mennicke},\cite{hr}).
\end{itemize}

\subsection{Labeling subgroups}
\label{section:subgroups}
We now describe and label several important subgroups of $\Sp(2p;R)$ (where $R=\RR$ or $\ZZ$) using roots and half roots.  Here is a brief list of these groups and their significance (the definitions will follow afterward).  Throughout this list, $S$ is an isotropic subset of $\mH$ and $T$ is a symplectic subset of $\mH$.  Each group will come in real and integral flavors, but we typically mean the integer version unless otherwise specified.
\begin{itemize}
\item The group $P_{S,T}(R)$ is a maximal parabolic subgroup of $\Sp(\pm S \cup T;R)$ (defined below, in a way which agrees with the discussion in \S\ref{section:theorem}).  By comparison, a maximal parabolic of $\SL(q;R)\subset\SL(p;R)$ consists of all matrices preserving some proper subspace of $R^{q}\subset R^{p}$.  Section \ref{section:parabtodiag} will describe how to reduce normal form triangles in $P_{S,T}(\mathbb{Z})$ to products of relations in the diagonal blocks of $P_{S,T}(\mathbb{Z})$.
\item The groups $\GL(S;R)$ and $\Sp(T;R)$ are diagonal blocks of the parabolic $P_{S,T}(R)$.  When $\#S\geq 5$, relations in $\GL(S;\ZZ)$ are known to have quadratic fillings by \cite{RY}.  One of our main theorems (theorem \ref{theorem:diagtoparab}) shows that relations in $\Sp(T;\ZZ)$ can be broken into normal form triangles in its maximal parabolic subgroups.
\item The group $N_{S,T}(R)$ is the unipotent radical of $P_{S,T}(R)$.  We denote its center by $Z_{S,T}(R)$ and its abelianization by $A_{S,T}(R)$.
\item The group $\SpH_{S,T}(R)$ is generated by $N_{S,T}(R)$ together with two abelian subgroups $\TT_{S}(R)\subset\GL(S;R)$ and $\TT_{T}(R)\subset\Sp(T;R)$.  This group has several important properties.  First the subgroup $N_{S,T}(R)$ is exponentially distorted.  Second, the subgroup consisting of integer matrices is cocompact; i.e., $\SpH_{S,T}(\ZZ)$ is cocompact in $\SpH_{S,T}(\RR)$.  Third, for sufficiently large $S$ and $T$, this group has quadratic Dehn function.  (We show this in \S\ref{section:solvable}).
\item $\TT_{S}(\ZZ)\subset \GL(S;\ZZ)$ and $\TT_{T}(\ZZ)\subset \Sp(T;\ZZ)$ are abelian groups of maximal rank which are generated by integer matrices.
\end{itemize}

\begin{definition}
Suppose $S\subset \mathcal{H}$ is isotropic and $T\subset\mathcal{H}$ is symplectic, with $S\cap T=\emptyset$.  Then we define $P_{S,T}(R)$ be the set of all $M\in \Sp(2p;R)$ such that $M$ preserves $\RR^{S}$ and acts as the identity on $\RR^{\mathcal{H}\setminus (S\cup -S \cup T)}$.
\end{definition}

For example, if $p=3$ and $S=\{[1]\}$ and $T=\{\pm [2],\pm [3]\}$, we have that 
$P_{S,T}$ consists of all symplectic matrices of the form

$$\begin{bmatrix}
\ast & \ast & \ast & \ast & \ast & \ast\\
0 & \ast & \ast & \ast & \ast & \ast\\
0 & \ast & \ast & \ast & \ast & \ast\\
0 & 0 & 0 & \ast & 0 & 0\\
0 & \ast & \ast & \ast & \ast & \ast\\
0 & \ast & \ast & \ast & \ast & \ast
\end{bmatrix}.$$

We refer to $P$ as a maximal parabolic of the symplectic diagonal block $\Sp(S\cup -S\cup T)$ defined below.  The reader can verify that the set of $\alpha$ such that $e_{\alpha}\in P_{S,T}$ is
$$\{s-t:s\in S,t\in T\}\cup
\{s\pm s^{\prime}\neq 0:s,s^{\prime}\in S\}
\cup \{t-t^{\prime}\neq 0:t,t^{\prime}\in T\}$$
We refer to this set of roots as $\Phi P_{S,T}$.  We note that $P_{S,T}(\ZZ)$ is not quite generated by $\{e_{\alpha}(x):x\in\RR,\alpha\in\Phi P_{S,T}\}$.  It will momentarily become apparent why this is the case.

\begin{definition}
Suppose $S\subset \mH$ is isotropic.  The subgroup $\GL(S;R)\subset \Sp(2p;R)$ is defined to be all $M\in \Sp(2p;R)$ such that $M$ preserves $\RR^{S}$ and $\RR^{-S}$ and acts as the identity on $\RR^{\mH\setminus (S\cup -S)}$.
\end{definition}

For example, if $p=3$ and $S=\{[1]\}$, we get that $\GL(S;\RR)$ consists of matrices of the form

$$\begin{bmatrix}
x & 0 & 0 & 0 & 0 & 0\\
0 & 1 & 0 & 0 & 0 & 0\\
0 & 0 & 1 & 0 & 0 & 0\\
0 & 0 & 0 & x^{-1} & 0 & 0\\
0 & 0 & 0 & 0 & 1 & 0\\
0 & 0 & 0 & 0 & 0 & 1
\end{bmatrix}$$
where $x\neq 0$.  We refer to $\GL(S)$ as a diagonal block of $P_{S,T}$, of which it is clearly a subgroup.  The set of $\alpha$ such that $e_{\alpha}\in \GL(S)$ is
$$\Phi\GL(S):=\{s-s^{\prime}\neq 0:s,s^{\prime}\in S\}.$$
The group generated by
$$\{e_{\alpha}(x):x\in\RR,\alpha\in\Phi \GL(S)\}$$
is $\SL(S;\ZZ)$, i.e, all the matrices in $\GL(S;\ZZ)$ which act by determinant $1$ matrices on $\RR^{S}$.  Of course, $\GL(S;R)$ is isomorphic to $\GL(\#S;R)$ via the obvious map.

\begin{definition}
Suppose $T\subset \mH$ is symplectic.  The subgroup $\Sp(T;R)$ of $\Sp(2p;R)$ is defined to be all $M\in \Sp(2p;R)$ such that $M$ preserves $\RR^{T}$ and acts as the identity on $\RR^{\mH\setminus T}$.
\end{definition}

For example, if $p=3$ and $T=\{\pm[2],\pm[3]\}$, we get that $\Sp(T)$ consists of all symplectic matrices of the form

$$\begin{bmatrix}
1 & 0 & 0 & 0 & 0 & 0\\
0 & \ast & \ast & 0 & \ast & \ast\\
0 & \ast & \ast & 0 & \ast & \ast\\
0 & 0 & 0 & 1 & 0 & 0\\
0 & \ast & \ast & 0 & \ast & \ast\\
0 & \ast & \ast & 0 & \ast & \ast
\end{bmatrix}.$$
We also call $\Sp(T)$ a diagonal block of $P_{S,T}$. The set of $\alpha$ such that $e_{\alpha}\in \Sp(T)$ is
$\{t-t^{\prime}\neq 0:t,t^{\prime}\in T\}$.
We refer to this set as $\Phi \Sp(T)$.  As noted before, the group generated by $\{e_{\alpha}(x):x\in\RR,\alpha\in\Phi \Sp(T)\}$ is all of $\Sp(T;\ZZ)$.  Furthermore, the group generated by
$$\{e_{\alpha}(x):x\in\RR,\alpha\in\Phi \Sp(T)\}$$
is $\Sp(T;\RR)$ which is indeed isomorphic to $\Sp(\RR^{T})$.

\begin{definition}
Suppose $S\subset \mathcal{H}$ is isotropic and $T\subset\mathcal{H}$ is symplectic, with $S\cap T=\emptyset$.
  Then we define $N_{S,T}(R)$ to be the set of all $M\in \Sp(2p;R)$ such that $M$ acts as the identity on $\RR^{S}$, on $\RR^{\mathcal{H}\setminus (S\cup -S \cup T)}$, and on $\RR^{S\cup -S \cup T}/\RR^{S\cup -S}$.
\end{definition}

For example, if $p=3$ and $S=\{[1]\}$ and $T=\{\pm [2],\pm [3]\}$, we have that 
$N_{S,T}$ consists of all symplectic matrices of the form

$$\begin{bmatrix}
1 & \ast & \ast & \ast & \ast & \ast\\
0 & 1 & 0 & \ast & 0 & 0\\
0 & 0 & 1 & \ast & 0 & 0\\
0 & 0 & 0 & 1 & 0 & 0\\
0 & 0 & 0 & \ast & 1 & 0\\
0 & 0 & 0 & \ast & 0 & 1
\end{bmatrix}.$$

Sometimes $N_{S,T}$ is called the unipotent radical of $P_{S,T}$.  We denote the center of $N_{S,T}$ by $Z_{S,T}$. For example, if $S=\{[1]\}$ and $T=\{\pm [2],\pm [3]\}$, then $Z_{S,T}$ consists of all matrices of the form

$$\begin{bmatrix}
1 & 0 & 0 & \ast & 0 & 0\\
0 & 1 & 0 & 0 & 0 & 0\\
0 & 0 & 1 & 0 & 0 & 0\\
0 & 0 & 0 & 1 & 0 & 0\\
0 & 0 & 0 & 0 & 1 & 0\\
0 & 0 & 0 & 0 & 0 & 1
\end{bmatrix}.$$

$Z_{S,T}$ is entirely determined by $S$ and does not depend on $T$, so we will also write it as $Z_{S}$.

The set of $\alpha$ such that $e_{\alpha}\in Z_{S}$ is
$\{s+s^{\prime}:s,s^{\prime}\in S\}$.
We refer to this set as $\Phi Z_{S}$.  The group generated by $\{e_{\alpha}(x):x\in\RR,\alpha\in\Phi Z_{S}\}$ is $Z_{S}(\ZZ)$.  The set of $\alpha$ such that $e_{\alpha}\in N_{S,T}$ is $\Phi Z_{S} \cup \{s-t:s\in S, t\in T\}$, denoted $\Phi N_{S,T}$, and this does generate $N_{S,T}$.

\begin{definition}
Let $S\subset \mH$ be isotropic.  Define $\TT_{S}(\ZZ)\subset \SL(S;\ZZ)$ be a free abelian group of rank $\#S -1$ such that every element is diagonalizable with positive eigenvalues.  (Surprisingly, such subgroups exist\cite[p.\ 236-237]{wordproc}).  Define $\TT_{S}(\RR)$ to be the group generated by all $g^{x}$ such that $g\in \TT_{S}(\ZZ)$ and $x\in \RR$.  (The expression $g^{x}$ makes sense because all eigenvalues of $g$ are positive).
\end{definition}

\begin{definition}
Let $T\subset \mH$ be symplectic.  Define $\TT_{T}(\ZZ)\subset \Sp(T;\ZZ)$ be a free abelian group of rank $\frac{1}{2}\#T$ such that every element is diagonalizable with positive eigenvalues.
\end{definition}

If we wish, we can take $\TT_{T}$ to be generated by matrices which act as
$\begin{bmatrix} 2 & 1 \\ 1 & 1\end{bmatrix}$ on some symplectic $\RR^{\{\pm t\}}\subset \RR^{T}$ and trivially on $\RR^{\mH\setminus\{\pm t\}}$.

\begin{definition}
Let $S,T\subset \mH$ be disjoint with $S$ isotropic and $T$ symplectic.  Define $\SpH_{S,T}$ be be the set of elements of $P_{S,T}$ which map to elements of $\TT_{S}\times \TT_{T}$ when modded out by $N_{S,T}$.
\end{definition}

There are three facts which will make $\SpH_{S,T}$ important to us.  First, $\SpH_{S,T}(\ZZ)$ is a cocompact subgroup of $\SpH_{S,T}(\RR)$ (Lemma \ref{lemma:hstcocompact}).  Second, $N_{S,T}$ is exponentially distorted in $\SpH_{S,T}$.  Finally, $\SpH_{S,T}$ has quadratic Dehn function when $\#S\geq 3$ (Theorem \ref{theorem:solvable2}).  All this implies that an integer elementary matrix $e_{s-t}(x)$ in can be expressed as a word of length $O(\log(x))$ in the generators of $\SpH_{S,T}(\ZZ)$, and any relation involving such words has a quadratic area filling in $\Sp(2p;\ZZ)$, as we exploit in \S\ref{section:shortcuts}.\\

The following definition generalizes the group $H_{S,T}$ defined in \cite{RY}.
\begin{definition}
Let $S,S^{\prime}\subset\mH$ be disjoint with $S\cup S^{\prime}$ isotropic.  Then $H_{S,S^{\prime}}(\ZZ)$ is the group generated by $\TT_{S}(\ZZ)$, $\TT_{S^{\prime}}(\ZZ)$, and $\{e_{s-s^{\prime}}:s\in S, s^{\prime}\in S^{\prime}\}$.  Similarly, $H_{S,S^{\prime}}(\RR)$ is the group generated by
$$\TT_{S}(\RR)\cup\TT_{S^{\prime}}(\RR)\cup\{e_{s-s^{\prime}}(r):s\in S, s^{\prime}\in S^{\prime}, r\in \RR\}.$$
\end{definition}

Observe that the group generated by $\{e_{s-s^{\prime}}(r):s\in S, s^{\prime}\in S^{\prime}, r\in \RR\}$ is abelian.  For instance, if $S=\{[1]\}$ and $S^{\prime}=\{[2],[3]\}$, then this group consists of matrices of the form
$$\begin{bmatrix}
1 & x & y & 0 & 0 & 0\\
0 & 1 & 0 & 0 & 0 & 0\\
0 & 0 & 1 & 0 & 0 & 0\\
0 & 0 & 0 & 1 & 0 & 0\\
0 & 0 & 0 & -x & 1 & 0\\
0 & 0 & 0 & -y & 0 & 1
\end{bmatrix}.$$

\subsection{The structure of $P_{S,T}(\RR)$}
\label{section:pst}
The group $P_{S,T}(\RR)$ is a semidirect product $(\GL(S;\RR)\times\Sp(T;\RR))\ltimes N_{S,T}(\RR)$, and the group $N_{S,T}(\RR)$ is a central extension of $\RR^{S}\otimes\RR^{T}$ by $\Sym^{2}\RR^{S}\cong Z_{S}(\RR)$.  It is crucial for our purposes to have an exact description of the action of $\GL(S;\RR)\times\Sp(T;\RR)$ on $N_{S,T}(\RR)$, and a means of labeling elements of $N_{S,T}(\RR)$ in terms of $\RR^{S}\otimes\RR^{T}$ and $\Sym^{2}\RR^{S}$.  There is a map
$$\Ab:N_{S,T}(\RR)\rightarrow\RR^{S}\otimes\RR^{T}$$
with kernel $Z_{S,T}(\RR)$, given by
$$\Ab(M)=\sum_{t\in T}(M-1)z_{t}\otimes z_{t}.$$
In terms of matrices, $\Ab$ reads off the $\RR^{T}$ entries of the $\RR^{S}$ rows.  For example, if $S=\{[1]\}$ and $T=\{\pm [2],\pm [3]\}$, we have

$$\Ab\left(\begin{bmatrix}
1 & 3 & 7 & 4 & 2 & 1\\
0 & 1 & 0 & 2 & 0 & 0\\
0 & 0 & 1 & 1 & 0 & 0\\
0 & 0 & 0 & 1 & 0 & 0\\
0 & 0 & 0 & -3 & 1 & 0\\
0 & 0 & 0 & -7 & 0 & 1
\end{bmatrix}\right)=z_{[1]}\otimes(3z_{[2]}+7z_{[3]}+2z_{-[2]}+z_{-[3]})$$

This map has a set theoretic section
$u:\RR^{S}\otimes \RR^{T}\rightarrow N_{S,T}$ 
given by
$$u(v\otimes w)z=z+(w^{\text{T}}z)v+\omega(v,z)J_{0}w.$$
In terms of matrices, this is the section given by putting zeroes in the $\RR^{-S}$ columns of the $\RR^{S}$ rows. For example, if $S=\{[1]\}$ and $T=\{\pm [2],\pm [3]\}$, we have

$$u(z_{[1]}\otimes(2z_{[2]}-5z_{-[3]}))=
\begin{bmatrix}
1 & 2 & 0 & 0 & 0 & -5\\
0 & 1 & 0 & 0 & 0 & 0\\
0 & 0 & 1 & -5 & 0 & 0\\
0 & 0 & 0 & 1 & 0 & 0\\
0 & 0 & 0 & -2 & 1 & 0\\
0 & 0 & 0 & 0 & 0 & 1
\end{bmatrix}.$$

(In this notation $u(z_{s}\otimes z_{t})=e_{s-t}$).

We now describe $Z_{S}(\RR)$.  There is a isomorphism $u_{Z}:\Sym^{2}\RR^{S}\rightarrow Z_{S}(\RR)$ given by
$$u_{Z}(vw)z=z+\omega(v,z)w+\omega(w,z)v$$
with inverse given by
$$M\mapsto \frac{-1}{2}\sum_{s\in S}z_{s}\odot (M-1)J_{0}z_{s}$$
In terms of matrices, $u_{Z}$ puts the bilinear form associated to $vw$ in the $\RR^{S}$ by $\RR^{-S}$ block.  For instance, with $S=\{1,2\}$ and $T=\{\pm 3\}$ we have

$$u_{Z}(z_{[1]}^{2}+z_{1}z_{2})=
\begin{bmatrix}
1 & 0 & 0 & 2 & 1 & 0\\
0 & 1 & 0 & 1 & 0 & 0\\
0 & 0 & 1 & 0 & 0 & 0\\
0 & 0 & 0 & 1 & 0 & 0\\
0 & 0 & 0 & 0 & 1 & 0\\
0 & 0 & 0 & 0 & 0 & 1
\end{bmatrix}.$$

The maps $u$ and $u_{Z}$ allow a convenient labeling of elements of $N_{S,T}(\RR)$.  The following proposition explains the basic relations of $P_{S,T}(\RR)$ in terms of this labeling.  In particular, the proposition describes the action of 
$\GL(S;\RR)\times \Sp(T;\RR)$ on $N_{S,T}(\RR)$ and the cocycle in $H_{2}(\RR^{S}\otimes \RR^{T})$ corresponding to the central extension 
$$0\rightarrow \Sym^{2}\RR^{S}\rightarrow N_{S,T}(\RR)\rightarrow \RR^{S}\otimes\RR^{T}\rightarrow 0.$$

\begin{proposition}
\label{proposition:umanip}
Fix disjoint $S,T\subset \mH$ with $S$ isotropic and $T$ symplectic.  Suppose $v,v^{\prime}\in \RR^{S}$ and $w,w^{\prime}\in\RR^{T}$.  Then we have the following equalities.
\begin{itemize}
\item[(a)] $u(v\otimes w)^{-1}=u(-v\otimes w)$ and $u_{Z}(vv^{\prime})^{-1}=u_{Z}(-vv^{\prime})$.
\item[(b)] For $d\in \GL(S)$, we have
$$du(v\otimes w)d^{-1}=u(dv\otimes w)$$
and
$$du_{Z}(vv^{\prime})d^{-1}=u(dv\odot dv^{\prime}).$$
\item[(c)] For $d\in \Sp(T)$, we have
$$du(v\otimes w)d^{-1}=u(v\otimes (d^{\text{T}})^{-1}w),$$
where $d^{\text{T}}$ denotes the transpose of $d$.
\item[(d)] If $v\otimes w,v^{\prime}\otimes w^{\prime}\in\RR^{S}\otimes\RR^{T}$, we have
$$[u(v\otimes w),u(v^{\prime}\otimes w^{\prime})]=u_{Z}(\omega(w,w^{\prime})vv^{\prime})$$
\end{itemize}
\end{proposition}

The proof of these formulas is routine, so we omit it.

\section{Shortcuts}
\label{section:shortcuts}
It is well known that the copy of $\ZZ$ generated by some unipotent $e_{\alpha}$ is exponentially distorted in $\Sp(2p,\ZZ)$ (see \cite{lmr} or \cite{riley}).  In order to find words in $\Sp(2p,\ZZ)$ which efficiently represent elements of parabolic subgroups, we will first need to find, for each root $\alpha$ and any $x\in \ZZ$, a word $\he_{\alpha}(x)$ of length $O(\log\vert x\vert)$ representing $e_{\alpha}(x)$.  (Following \cite{RY}, we will call $\he_{\alpha}(x)$ a shortcut).  If we further wish to break a normal form triangle in a maximal parabolic into a product of relations in the symplectic block, we will then need an efficient way to fill relations between different shortcuts, and to fill relations between shortcuts and words in diagonal blocks.\\

Subsections \S\ref{subsection:shortrootshortcuts} through \S\ref{subsection:longrootshortcuts} are devoted to defining shortcuts.  In fact, for each root $\alpha$, we will describe several different shortcuts for $e_{\alpha}(x)$, and show that one can homotope between different shortcuts for $e_{\alpha}(x)$ at quadratic cost. (Having several words representing $e_{\alpha}(x)$ makes it easier to find fillings for words built out of shortcuts. For instance, if a shortcut is a word in the generators of $\GL(S;\ZZ)$ then it commutes at quadratic cost with any word in the generators of $\Sp(\mH\setminus\pm S;\ZZ)$).  In particular, \S\ref{subsection:shortrootshortcuts} handles the case where $\alpha$ is short.  In this case, a shortcut for $e_{\alpha}(x)$ may be defined as a word in the generators of any reasonable $H_{S,S^{\prime}}$ containing $e_{\alpha}$, and \ref{lemma:shortrootswitching} shows that we can switch between different choices of $S$ and $S^{\prime}$ at quadratic cost.  \S\ref{subsection:shortrootspecialshortcuts} handles the same case, but produces shortcuts contained in conjugates of $\Sp(4;\ZZ)$. \S\ref{subsection:longrootshortcuts} produces shortcuts for $e_{\alpha}(x)$ when $\alpha$ is a long root (these are words in the generators of $\SpH_{S,T}(\ZZ)$).

Next, \S\ref{subsection:blockshortcut} and \S\ref{subsection:steinberg} will show that one can fill various relations involving shortcuts.  Since our shortcuts live in groups like $\SpH_{S,T}(\ZZ)$ or $\GL(q;\ZZ)$, the main tool used to fill these relations is the fact that $\GL(q;\ZZ)$ and $\SpH_{S,T}(\ZZ)$ have quadratic Dehn function (for appropriately large $q,S,T$) by Theorem \ref{theorem:RY} and Theorem \ref{theorem:solvable2} respectively.  The remaining subsections explain some consequences of the existence of shortcuts.  \S\ref{subsection:lmr} explains that matrices in $\Sp(T;\ZZ)$ may be efficiently represented by products of shortcuts. Finally \S\ref{subsection:omegapst} describes a normal form for elements of $P_{S,T}(\ZZ)$.

\subsection{Defining $\he_{\alpha}$ for $\alpha$ a short root}
\label{subsection:shortrootshortcuts}
When $\alpha$ is a short root, $e_{\alpha}$ lives inside some embedding of $\SL(p;\ZZ)$ (since $\alpha\in \Phi\GL(S)$ for some maximal isotropic $S\subset \mH$), so we will find shortcuts for $e_{\alpha}(x)$ exactly as in \cite[\S 6.1]{RY}.  Suppose $\alpha=s-s^{\prime}$, and we have $S,S^{\prime}\subset \mathcal{H}$.
\begin{definition}
We say that the pair $(S,S^{\prime})$ is compatible with $\alpha$ if the following hold.
\begin{itemize}
\item $s\in S$ and $s^{\prime}\in S^{\prime}$
\item $S\cap S^{\prime}=\emptyset$ and $S\cup S^{\prime}$ isotropic
\item $\# S \geq 2$ and $\# S^{\prime} \geq 1$
\end{itemize}
\end{definition}
For $S,S^{\prime}\subset \mH$ compatible with $\alpha\in\Phi$, we will construct a shortcut $\he_{\alpha;S,S^{\prime}}$ living in $H_{S,S^{\prime}}$.\\

Let $u:\RR^{S}\otimes \RR^{S^{\prime}}\rightarrow H_{S,S^{\prime}}(\RR)$ be the homomorphism given by setting $u(z_{s}\otimes z_{s^{\prime}})=e_{s-s^{\prime}}$ for any $s\in S$ and $s^{\prime}\in S^{\prime}$.  For $V\in \RR^{S}\otimes \RR^{S^{\prime}}$, we will construct a word $\hu(V)\in\Sigma_{H_{S,S^{\prime}}}^{\ast}$ representing $u(V)$ (the generating set $\Sigma_{H_{S,S^{\prime}}}$ will be explained below).  Our shortcut $\he_{s-s^{\prime};S,S^{\prime}}(x)$ will then be an approximation of $\hu(x z_{s}\otimes z_{s^{\prime}})$ by integer matrices.\\

Let $v_{1},\ldots,v_{m}$ be an eigenbasis for the action of $\TT_{S}$ on $\RR^{S}$ and $w_{1},\ldots,w_{n}$ an eigenbasis for the action of $\TT_{S^{\prime}}$ on $\RR^{S^{\prime}}$.  Choose $d_{i}\in \TT_{S}$ such that $d_{i}v_{i}=ev_{i}$ (where $e=\exp(1)$), and let $\Sigma_{H_{S,S^{\prime}}}$ consist of all $d_{i}^{x}$ with $\| x \|\leq 1$ together with all $u(xv_{i}\otimes w_{j})$ with $\| x \|\leq 1$.  We first define $\hu(V)$ for $V$ of the form $xv_{i}\otimes w_{j}$ as follows.  If $\| x \| \leq 1$, just take $\hu(V)=u(V)$, otherwise let $\hu(V)=d_{i}^{\log(x)}u(v_{i}\otimes w_{j})d_{i}^{-\log(x)}$ if $x>0$ and $d_{i}^{\log(-x)}u(-v_{i}\otimes w_{j})d_{i}^{-\log(-x)}$ if $x<0$.  Extend this to arbitrary $V$ by taking $\hu(\sum x_{ij}v_{i}\otimes w_{j})=\prod \hu(x_{ij}v_{i}\otimes w_{j})$.

\begin{lemma}
\label{lemma:efficientshortcuts}
In the following sense, $\hu(V)$ is an efficient representative of $u(V)$.
\begin{itemize}
\item[(a)] $\ell(\hu(V))=O(\log\Vert V\Vert_{\infty})$
\item[(b)] There exists a constant $\epsilon$ such that, if $w\in \Sigma_{H_{S,S^{\prime}}}^{\ast}$ is some word representing $u(V)$, then $\ell(w)>\epsilon \log\Vert V\Vert_{\infty}$.\\
\end{itemize}
\end{lemma}
\begin{proof}
\begin{itemize}
\item[(a)] This follows directly from the definition.
\item[(b)] Observe that if $g,h$ are $p\times p$ real matrices, $\Vert gh \Vert_{\infty}\leq p\Vert g\Vert_{\infty} \Vert h\Vert_{\infty}$.  Hence, for any word $w\in \Sigma_{H_{S,S^{\prime}}}^{\ast}$, we have that
$$\log\Vert w \Vert_{\infty}\leq\ell(w)(\log(p)+
\sup\{\log\Vert M\Vert:M\in \Sigma_{H_{S,S^{\prime}}}\}).$$
Taking $\epsilon=1/(\log(p)+
\sup\{\log\Vert M\Vert:M\in \Sigma_{H_{S,S^{\prime}}}\})$ the result follows.
\end{itemize}\end{proof}

We can now represent a matrix of the form $e_{\alpha}(x)$ by a word in $\Shssr$, but we really want to represent it by a word in the generators of $\Sp(2p;\ZZ)$.  Because $H_{S,S^{\prime}}(\ZZ)$ is cocompact in $H_{S,S^{\prime}}(\RR)$, there exists a constant $C$ such that for any $M\in H_{S,S^{\prime}}(\RR)$, there is a $N\in H_{S,S^{\prime}}(\ZZ)$ such that $MN^{-1}$ can be written as a product of at most $C$ elements of $\Shssr$.  Hence, given a word $w=w_{1}\ldots w_{n}$ in $\Shssr$, we can find a sequence $\{\gamma_{i}\}_{i=0,\ldots,n}\subset H_{S,S^{\prime}}(\ZZ)$ such that $\gamma_{i}^{-1}w_{1}\ldots w_{i}$ is a product of at most $C$ elements of $\Shssr$, $\gamma_{0}=1$ and $\gamma_{n}=w_{1}\ldots w_{n}$.  Thus, each $\gamma_{i-1}^{-1}\gamma_{i}$ is a product of at most $2C+1$ elements of $\Shssr$.  If we fix a sufficiently large finite generating set $\Shssz$ for $H_{S,S^{\prime}}(\ZZ)$, it follows that for all $i$, $\gamma_{i-1}^{-1}\gamma_{i}\in\Shssz$, thus (setting $w_{i}=\gamma_{i-1}^{-1}\gamma_{i}$), we get a word $w^{\prime}=w_{1}^{\prime}\ldots w_{n}^{\prime}$ in $\Shssz$.  We assumed that our generating set for $\Sp(2p;\ZZ)$ contained $\Shssz$, so we get a word $w^{\prime\prime}$ in the generators of $\Sp(2p;\ZZ)$.  We refer to $w^{\prime\prime}$ as an approximation of $w$.

\begin{definition}
For $S,S^{\prime}\subset \mathcal{H}$ with $(S,S^{\prime})$ compatible with some root $\alpha=s-s^{\prime}\in\Phi$ and $x\in \ZZ$, define $\he_{\alpha;S,S^{\prime}}(x)$ to be a word in $H_{S,S^{\prime}}(\ZZ)$ approximating the word $\hu(x z_{s}\otimes z_{s^{\prime}})\in \Sigma_{H_{S,S^{\prime}}}^{\ast}$.
\end{definition}

For each short root $\alpha$, we have produced several shortcuts $\he_{\alpha;S,S^{\prime}}(x)$ (where $(S,S^{\prime})$ ranges over compatible pairs) representing $e_{\alpha}(x)$.  By Lemma \ref{lemma:efficientshortcuts} and the definition of an approximation, we have that $\ell(\he_{\alpha;S,S^{\prime}}(x))=O(\log(x))$.  The following lemma shows that we can homotope between these shortcuts at quadratic cost.
\begin{lemma}
\label{lemma:shortrootswitching}
Suppose $(S_{0},S_{0}^{\prime})$ and $(S_{1},S_{1}^{\prime})$ are compatible with $\alpha=s-s^{\prime}$, then 
$\delta(\he_{\alpha;S_{0},S_{0}^{\prime}}(x),\he_{\alpha;S_{1},S_{1}^{\prime}}(x))=
O(\log(x)^{2})$.
\end{lemma}
\begin{proof}
We will first show that we can perform the desired homotopy when $S_{0}\cup S_{0}^{\prime}\cup S_{1}\cup S_{1}^{\prime}$ is isotropic.  By combining six homotopies of this type, the general case will follow.\\
If $S_{0}\cup S_{0}^{\prime}\cup S_{1}\cup S_{1}^{\prime}$ is isotropic, then $\he_{s-s^{\prime};S_{0},S_{0}^{\prime}}(x)$ and $\he_{s-s^{\prime};S_{1},S_{1}^{\prime}}(x)$ live in $\GL(L)$ for some isotropic set $L\supset S_{0}\cup S_{0}^{\prime}\cup S_{1}\cup S_{1}^{\prime}$ with $\# L\geq 5$, so we can switch between them at quadratic cost, i.e.,
$$\delta(  \he_{s-s^{\prime};S_{0},S_{0}^{\prime}}(x) ,  \he_{s-s^{\prime};S_{1},S_{1}^{\prime}}(x) )=
O(\log(x)^{2})$$ (because $\GL(q)$ is a finite index supergroup of $\SL(q)$ which has quadratic Dehn function for $q\geq 5$ by Theorem \ref{theorem:RY}).\\
Now we show that we can switch, at quadratic cost, between shortcuts $\he_{s-s^{\prime};S_{0},S_{0}^{\prime}}(x)$ and $\he_{s-s^{\prime};S_{1},S_{1}^{\prime}}(x)$ for any choice of $(S_{0},S_{0}^{\prime})$ and $(S_{1},S_{1}^{\prime})$ compatible with $\alpha$.  Choose $s_{0}\in S_{0}$ and $s_{1}\in S_{1}$ distinct from $s$, and choose
$s_{2}\in \mathcal{H}\setminus
\{\pm s, \pm s^{\prime}, \pm s_{0}, \pm s_{1}\}$
Then we can homotope through shortcuts with respect to the following sequence of compatible pairs, because the union of the four sets involved in any two consecutive pairs is isotropic:
$$S_{0},S_{0}^{\prime}\rightarrow \{s,s_{0}\},\{s^{\prime}\} \rightarrow
\{s,s_{0},s_{2}\},\{s^{\prime}\} \rightarrow
\{s,s_{2}\},\{s^{\prime}\} \rightarrow$$
$$\{s,s_{2},s_{1}\},\{s^{\prime}\} \rightarrow
\{s,s_{1}\},\{s^{\prime}\} \rightarrow
S_{1},S_{1}^{\prime}$$
Since each step has quadratic cost, the entire homotopy, going from\\
$\he_{s-s^{\prime};S_{0},S_{0}^{\prime}}(x)$ to $\he_{s-s^{\prime};S_{1},S_{1}^{\prime}}(x)$ has quadratic cost.\end{proof}

\begin{definition}
For each $\alpha=s-s^{\prime}$ (where $s^{\prime}\neq \pm s$), fix a compatible pair $S(\alpha),S^{\prime}(\alpha)\subset \mathcal{H}$.  We define $\he_{\alpha}(x)$ to be $\he_{\alpha;S(\alpha),S^{\prime}(\alpha)}(x)$.\\
\end{definition}

\subsection{Defining special shortcuts for $\alpha$ a short root}
\label{subsection:shortrootspecialshortcuts}
At some point we will need to find a shortcut for $e_{[1]-[2]}(x)$ which commutes with words in $\Sp(T;\ZZ)$ when $T=\{\pm [3],\ldots,\pm [p]\}$.  The easiest place to find such a shortcut is inside $\Sp(\{\pm [1],\pm [2]\};\ZZ)$.

\begin{definition}
Let $\alpha=s-t$ be a short root, and $x\in \ZZ$.  We define the special shortcut $\he_{\alpha;\{s\},\{\pm t\}}$ to be an approximation of a word in $\SpH_{\{s\},\{\pm t\}}(\RR)$ of length $O(\log(x))$ representing $e_{\alpha}(x)$.
\end{definition}

This word exists because (writing $S$ for $\{s\}$ and $T$ for $\{\pm t\}$), $\SpH_{S,T}(\ZZ)$ is cocompact in $\SpH_{S,T}(\RR)$ (Lemma \ref{lemma:hstcocompact}) and $N_{S,T}(\RR)$ is exponentially distorted in $\SpH_{S,T}(\RR)$, as we can see by the following argument.  Let $v_{1},v_{2}\in\RR^{T}$ be the eigenbasis of $\TT_{T}$, then for $i=1,2$, we know that $u(z_{s}\otimes v_{i})$ can be represented by a word $\hu(z_{s}\otimes v_{i})$ of length $O(\log\Vert v_{i}\Vert)$.  From this, it easily follows that all of $N_{S,T}$ is exponentially distorted. The following lemma asserts that we can switch to a special shortcut at quadratic cost.

\begin{lemma}
\label{lemma:specialswitching}
$$\delta(\he_{\alpha}(x),\he_{\alpha;\{s\},\{\pm t\}}(x))=O(\log(|x|)^{2})$$
\end{lemma}

\begin{proof}
Because $\# \{ s\}=1$, if $S\subset \mH$ is an isotropic set of size at least $3$ containing $s$ but not $\pm t$, we have that $\SpH_{S,\{\pm t\}}\supset \SpH_{\{s\},\{\pm t\}}$.  Theorem \ref{theorem:solvable2} asserts that $\SpH_{S,\{\pm t\}}(\ZZ)$ has quadratic Dehn function, thus
$$\delta(\he_{\alpha;S,\{t\}}(x),\he_{\alpha;\{s\},\{\pm t\}}(x))=O(\log(|x|)^{2})$$
since both shortcuts are approximations of curves in $\SpH_{S,\{\pm t\}}(\RR)$.  But we already know that
$$\delta(\he_{\alpha;S,\{t\}}(x),\he_{\alpha}(x))=O(\log(|x|)^{2})$$
by Lemma \ref{lemma:shortrootswitching}.\end{proof}

\subsection{Defining $\he_{\alpha}$ for $\alpha$ a long root}
\label{subsection:longrootshortcuts}
Let $\alpha=2s$ be a long root.  Since $e_{\alpha}$ is not contained in any $GL(S)$, we will need a different strategy.  We would like to define a shortcut for $e_{\alpha}(x)$ as a commutator of shortcuts for short roots, i.e. as $[\he_{s-t}(x_{1}),\he_{s+t}(x_{2})]$.  This is possible only when $x$ is even, otherwise we will need to append $e_{2s}$ itself.
\begin{definition} We say that a triple $(S,T,t)$ with $S,T\subset \mathcal{H}$ is compatible with $\alpha=2s$ if the following hold. 
\begin{itemize}
\item $s\in S$ and $t\in T$
\item $S\cap T= \emptyset$, $S$ isotropic, and $T$ symplectic
\item $\# S\geq 2$ and $\# T\geq 2$
\end{itemize}

We say that a triple $(S,T,t)$ is special if all of these conditions hold except that $\# S=1$, and furthermore $\# T =2$.\\
\end{definition}

For a long root $\alpha=2s$ and a compatible triple $(S,T,t)$, we now define a shortcut for $e_{\alpha}(x)$ which approximates a word in $\SpH_{S,T}(\RR)$.

\begin{definition}
Let $(S,T,t)$ be compatible with $\alpha=2s\in \Phi$.  Then $\he_{\alpha;S,T,t}(x)$ is definied to be $[\he_{s-t;S,T^{+}}(\lfloor\frac{x}{2}\rfloor),\he_{s+t;S,T^{-}}(1)]e_{2s}(\{\frac{x}{2}\})$ where $\lfloor\cdot\rfloor$ denotes integer part, $\{\cdot\}$ denotes fractional part, and $T^{+}$ and $T^{-}$ are as defined in \S\ref{section:preliminaries}.
\end{definition}

We will also need to define shortcuts relative to special triples in order to represent elements of a symplectic diagonal block $\Sp(4;\ZZ)\subset\Sp(2p;\ZZ)$ as words in $\Sp(4;\ZZ)$ when $\#T=4$.  Given a triple $(S,T,t)$ which is special relative to $\alpha=2s$, observe that $e_{2s}$ is exponentially distorted in $\SpH_{S,T}(\RR)$ because it is a commutator of exponentially distorted elements, and recall that integer matrices are cocompact in $\SpH_{S,T}(\RR)$.  Hence there exists a word of length $O(\log|x|)$ in the generators of $\SpH_{S,T}(\RR)$ representing $e_{\alpha}(x)$, so let $\he_{\alpha;S,T}(x)$ be an approximation of this word.\\

Now we wish to prove that we can switch between different shortcuts for $e_{\alpha}(x)$ at quadratic cost.
\begin{lemma}
\label{lemma:longrootswitching}
Suppose $(S,T,t)$ and $(S^{\prime},T^{\prime},t^{\prime})$ are compatible with a long root $\alpha=2s$.  Then we have
$$\delta(\he_{\alpha;S,T,t}(x),\he_{\alpha;S^{\prime},T^{\prime},t^{\prime}}(x))=O(\log\| x \|^{2})$$
\end{lemma}
\begin{proof}
{\it Case 1:} If $t=t^{\prime}$, then the result follows from Lemma \ref{lemma:shortrootswitching}.\\
{\it Case 2:} If $S=S^{\prime}$ is of size at least $3$, and $T=T^{\prime}$, then both $\he_{\alpha;S,T,t}(x)$ and $\he_{\alpha;S^{\prime},T^{\prime},t^{\prime}}$ live inside $\SpH_{S,T}$ which has quadratic Dehn function by Theorem \ref{theorem:solvable2}, so we can switch between them at cost $O(\log\|x\|^{2})$.\\
{\it General case:} Without loss of generality, assume $t\neq t^{\prime}$.  Let $T^{\prime\prime}=\{\pm t, \pm t^{\prime}\}$ and $S^{\prime\prime}=(\mathcal{H}\setminus{T})^{+}$.  We can homotope at quadratic cost from $\he_{\alpha;S,T,t}(x)$to $\he_{\alpha;S^{\prime},T^{\prime},t^{\prime}}(x)$ through shortcuts relative to the following sequence of triples.
\[S,T,t{\buildrel \text{Case 1}\over\rightarrow} S^{\prime\prime},T^{\prime\prime},t{\buildrel \text{Case 2}\over\rightarrow}\\
S^{\prime\prime},T^{\prime\prime},t^{\prime}{\buildrel \text{Case 1}\over\rightarrow} S^{\prime},T^{\prime},t^{\prime}\qedhere\]
\end{proof}

We also wish to show that we can switch to a special shortcut at quadratic cost.
\begin{lemma}
\label{lemma:specialshortcutswitching}
Suppose $(S,T,t)$ is compatible with $\alpha=2s$ and $(S^{\prime},T^{\prime},t^{\prime})$ is special relative to $2s$.  Then
 $$\delta(\he_{\alpha;S,T,t}(x),\he_{\alpha;S^{\prime},T^{\prime},t^{\prime}}(x))=O(\log\| x \|^{2})$$
\end{lemma}
\begin{proof}
Since $S^{\prime}$ is special, we know $S^{\prime}=\{s\}$ and $T^{\prime}=\{\pm t\}$.  Let $S^{\prime\prime}$ be $(\mathcal{H}\setminus\{\pm s,\pm t\})^{+}\cup \{s\}$.  Then we can homotope from $\he_{\alpha;S^{\prime},T^{\prime},t^{\prime}}(x)$ to $\he_{\alpha;S^{\prime\prime},T^{\prime},t^{\prime}}(x)$ at quadratic cost since both live in $\SpH_{S^{\prime\prime},T^{\prime}}$.  But by the previous lemma, we can homotope from $\he_{\alpha;S^{\prime\prime},T^{\prime},t^{\prime}}(x)$ to $\he_{\alpha;S,T,t}(x)$ at quadratic cost.\end{proof}

In light of these facts, we make the following definition.
\begin{definition}
For each $s\in \mH$, choose a triple $(S(s),T(s),t(s))$ compatible with $2s$ and define $\he_{2s}(x)$ to be $\he_{2s;S(s),T(s),t(s)}(x)$.
\end{definition}

\subsection{The block-shortcut lemma}
\label{subsection:blockshortcut}
Let $S\subset \mH$ be isotropic and $T\subset \mH$ symplectic and suppose $\gamma$ is a word in $\SL(S;\ZZ)$ or $\Sp(T;\ZZ)$ representing some matrix $M$, and $\alpha\in\Phi N_{S,T}$.  Then $M e_{\alpha}(x)M^{-1}=_{G} \prod_{\beta}e_{\beta}(xM_{\alpha,\beta})$, where $\beta$ ranges over all roots in $\alpha+\Phi\SL(S)$ if $M\in\SL(S;\ZZ)$ and over all roots in $\alpha+\Phi\Sp(T)\cup\Phi Z_{S}$ if $M\in\Sp(T;\ZZ)$, and the $M_{\alpha,\beta}$ are integers depending on $M$, $\alpha$, $\beta$, and the order of multiplication in the product.  The following ``block-shortcut" lemma will allow us to move words in $\SL(S;\ZZ)$ or $\Sp(T;\ZZ)$ past products of shortcuts for elements of $N_{S,T}(\ZZ)$.

\begin{lemma}
\label{lemma:blockshortcut}
We can move shortcuts past words in diagonal blocks at quadratic cost as follows.
\begin{itemize}
\item[(a)] Suppose $S\subset \mathcal{H}$ isotropic with $\# S\leq p-1$ and $\alpha\in\Phi$ is a root with $\alpha=s-t$ for some $s\in S$ and $t\notin S$.  If $\gamma$ is a word in $\Sigma_{SL(S;\ZZ)}$ representing some matrix $M\in SL(S;\ZZ)$, and $M e_{\alpha}(x)M^{-1}=\prod_{\beta\in\alpha+\Phi\SL(S)}e_{\beta}(xM_{\alpha\beta})$ then we have
$$\delta(\gamma \he_{\alpha}(x)\gamma^{-1},\prod_{\beta\in\alpha+\Phi\SL(S)}\he_{\beta}(xM_{\alpha\beta}))=O(\ell(\gamma)^{2}+(\log|x|)^{2}).$$
\item[(b)] Suppose $T\subset \mathcal{H}$ symplectic with $\# T\leq 2(p-3)$ and $\alpha$ is a root of the form $s-t$ for some $s\notin T$ and $t\in T$.  If $\gamma$ is a word in $\Sigma_{\Sp(T;\ZZ)}$ representing some matrix $M\in \Sp(T;\ZZ)$, and $M e_{\alpha}(x)M^{-1}=\prod_{\beta\in\alpha+\Phi\Sp(T)\cup\Phi Z_{S}}e_{\beta}(xM_{\alpha\beta})$ then we have
$$\delta(\gamma e_{\alpha}(x)\gamma^{-1},\prod_{\beta\in\alpha+\Phi\Sp(T)\cup\Phi Z_{S}}e_{\beta}(xM_{\alpha\beta}))=O(\ell(\gamma)^{2}+(\log|x|)^{2})$$
\end{itemize}
\end{lemma}

\begin{proof}
\begin{itemize}
\item[(a)] Without loss of generality, $\#S\geq 4$.  (Otherwise, we can add elements to $S$ without changing the setup).\\
{\it Case 1:} If $-t\notin S$, let $T=\{t\}$ and homotope as follows (each step being quadratic cost by the given reason).\\

\begin{tabular}{l l}
$\gamma \he_{\alpha}(x)\gamma^{-1}$ & \\
$\leadsto\gamma \he_{\alpha;S,T}(x)\gamma^{-1}$ & Lemma \ref{lemma:shortrootswitching} (switching)\\
$\leadsto\prod_{\beta\in\alpha+\Phi\SL(S)}\he_{\beta;S,T}(xM_{\alpha\beta})$ & since $\SL(S\cup T;\ZZ)$ has quadratic Dehn function\\
$\leadsto\prod_{\beta\in\alpha+\Phi\SL(S)}\he_{\beta}(xM_{\alpha\beta})$ & by switching\\
\end{tabular}\\

{\it Case 2:} For $-t\in S\setminus \{s\}$, let $s^{\prime}=-t$ and choose $t^{\prime}\in \mathcal{H}$ such that $\pm t^{\prime}\notin S$.  Then $e_{\alpha}(x)=[e_{s-t^{\prime}}(x),e_{s^{\prime}+t^{\prime}}]$, and all these unipotents have shortcuts living in $\SpH_{S,T}(\ZZ)$ (where $T=\{\pm t^{\prime}\}$), so we can homotope\\

\begin{tabular}{l l}
$\gamma \he_{\alpha}(x)\gamma^{-1}$ & \\
$\leadsto\gamma[\he_{s-t^{\prime}}(x), \he_{s^{\prime}+t^{\prime}}]\gamma^{-1}$ &
Theorem \ref{theorem:solvable2}.\\
$\leadsto[\gamma \he_{s-t^{\prime}}(x) \gamma^{-1},\gamma e_{s^{\prime}+t^{\prime}}\gamma^{-1}]$ & free insertion\\
$\leadsto[ \prod_{\beta\in(s-t^{\prime})+\Phi_{S}}\he_{\beta}(xM_{\alpha\beta}),\prod_{\beta\in(s+t^{\prime})+\Phi_{S}}\he_{\beta}(xM_{\alpha\beta})]$ & by case 1\\
$\leadsto\prod_{\beta\in\alpha+\Phi_{S}}\he_{\beta}(xM_{\alpha\beta})$ & Theorem \ref{theorem:solvable2}\\
\end{tabular}\\

{\it Case 3:} The case where $-t=s$ (i.e. $\alpha$ is long) is the hardest, and here we will need to directly emulate the proof of \cite[Lemma 7.3]{RY}.  Observe that $\he_{\alpha;S,T,t^{\prime}}(x)$ (where $t^{\prime}\in T=(\mathcal{H}\setminus S)^{+}$) has the form $[\he_{s-t^{\prime};S,T}(\lfloor\frac{x}{2}\rfloor),\he_{s-t^{\prime}}]e_{2s}(\{\frac{x}{2}\})$.  By previous cases, we can transform the $\gamma \he_{s\pm t^{\prime}}\gamma^{-1}$ words into products of shortcuts living in $\SpH_{S,T}$, so it suffices to do the same for $\gamma e_{2s}\gamma^{-1}$.\\

In the following finite sequence of homotopies, each stage $w_{i}$ is a word in some subgroup $H_{i}$ of $\Sp(2p;\RR)$ in which integer matrices are cocompact, and moreover, each pair of consecutive stages $w_{i},w_{i+1}$ is contained in some common group $H_{i,j}\subset \Sp(2p;\RR)$ in which integer matrices are cocompact.  By approximating each stage $w_{i}$ by a word in $\SSp$, we obtain an actual homotopy through words in $\Sp(2p;\ZZ)$.  The area of this approximation will be bounded by some multiple of the area of the original homotopy.\\

Let $T={\pm t}$, let $q=\#S$ and let $v_{1},\ldots,v_{q}$ be an eigenbasis for the action of $\TT_{S}$ on $\RR^{S}$.  Choose some $w_{1},w_{2}$ part of an eigenbasis for $\TT_{T}$ with $\omega(w_{1},w_{2})=1$ (see, e.g., \cite[Lemma 2.18]{mcduff}).  At bounded cost, we can homotope $e_{2s}$ to a product $\prod_{i,j=1}^{q} u_{Z}(x_{ij}v_{i}v_{j})$.  (We think of each $u(x_{ij}v_{i}v_{j})$ as being in some compact generating set for $Z_{S}(\RR)$).  At bounded cost, we can homotope this word to
$$\prod_{i,j=1}^{q}[u(x_{ij}v_{i}\otimes w_{1}),u(v_{j}\otimes w_{2})]$$
Thus, we can homotope at bounded cost from $\gamma e_{2s} \gamma^{-1}$ to
$$\prod_{i,j=1}^{q}
[\gamma u(x_{ij}v_{i}\otimes w_{1})\gamma^{-1},\gamma u(v_{j}\otimes w_{2})\gamma^{-1}].$$
Now, we can homotope each $\gamma u(xv\otimes w)\gamma^{-1}$ to $\hu((Mxv)\otimes w)$ as in figure \ref{figure:rectanglehomotopy1}, where $D$ is a word of length $O(\ell)$ in $\TT_{T}$ such that
$$\| Dw \|\leq \sup\{\|N\|_{\infty}:N\in\SSp\}^{-\ell(\gamma)}$$
(compare \cite[figure 4]{RY}).

\begin{figure}[t]
\labellist
\small\hair 2pt

\pinlabel {$\triangleleft$} [c] at 245 335
\pinlabel {$\triangle$} [c] at 496 168
\pinlabel {$\triangle$} [c] at 497 168
\pinlabel {$\triangleleft$} [c] at 245 0
\pinlabel {$\triangle$} [c] at 1 168
\pinlabel {$\triangle$} [c] at 0 168
\pinlabel {$\nwarrow$} [c] at 67 270
\pinlabel {$\nearrow$} [c] at 429 268
\pinlabel {$\searrow$} [c] at 435 62
\pinlabel {$\swarrow$} [c] at 67 67
\pinlabel {$\triangleleft$} [c] at 245 209
\pinlabel {$\triangle$} [c] at 370 168
\pinlabel {$\triangleleft$} [c] at 245 127
\pinlabel {$\triangle$} [c] at 129 168

\pinlabel {$\hu(xv\otimes w)$} [l] at 5 168
\pinlabel {$D$} at 74 273
\pinlabel {$\gamma$} at 245 321
\pinlabel {$D$} at 425 273
\pinlabel {$\hu(Mxv\otimes Dw)$} [r] at 487 168
\pinlabel {$D$} at 425 60
\pinlabel {$\gamma$} at 245 12
\pinlabel {$D$} at 74 60
\pinlabel {$u(xv\otimes Dw)$} [l] at 132 168
\pinlabel {$\gamma$} at 245 220
\pinlabel {$u(Mxv\otimes Dw)$} [r] at 368 168
\pinlabel {$\gamma$} at 245 116

\endlabellist
\centering
\centerline{\psfig{file=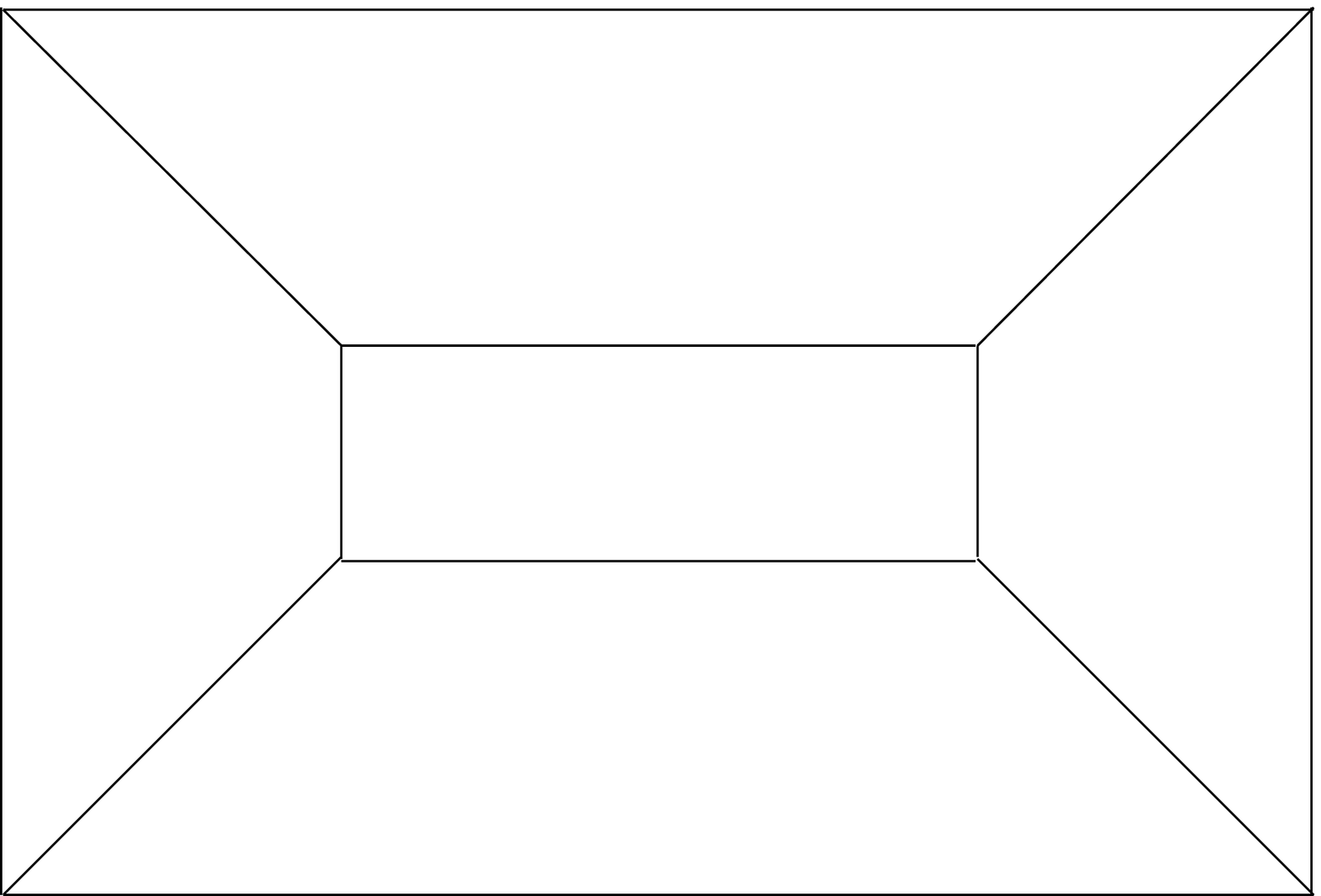,scale=75}}
\caption{A homotopy from $\gamma\hu(xv\otimes w)\gamma^{-1}$ to $\hu(Mxv\otimes w)$}
\label{figure:rectanglehomotopy1}
\end{figure}

We fill the top and bottom trapezoids because $D$ is a word in $\TT_{T}$ and $\gamma$ is a word in $\SL(S)$.  We can fill the left and right trapezoids because $\SpH_{S,T}$ has quadratic Dehn function.  We can fill the middle rectangle because it is skinny (for each prefix $\gamma_{i}$ of $\gamma$, we know that $\gamma_{i}u(xv\otimes Dw)\gamma_{i}^{-1}$ represents a matrix of norm less than $1$ by our choice of $D$).\\

We obtain a product of a bounded number of $\hu$, which we can homotope to the desired product of shortcuts because $\SpH_{S,T}$ has quadratic Dehn function.

\item[(b)] The proof is similar, but we must choose $D$ from $\TT_{S}$ instead of $\TT_{T}$, and similarly, instead of $\hu$ as defined above, we must choose representatives of $u(V)$ which use $\TT_{T}$ matrices for distortion instead of $\TT_{S}$ matrices.

\end{itemize}\end{proof}

\subsection{Steinberg relations}
\label{subsection:steinberg}
Now we will show that shortcuts can be moved past each other at quadratic cost (sometimes leaving more shortcuts behind, of course).
\begin{lemma}
\label{lemma:steinberg}
\begin{itemize}
\item[(a)] For $\alpha\in\Phi$, we have
$$\delta(\he_{\alpha}(x)\he_{\alpha}(y),\he_{\alpha}(x+y))=O(\log|x|^{2}+\log|y|^{2}).$$
\item[(b)] For $\alpha+\beta\neq 0$, we have
$$\delta([\he_{\alpha}(x),\he_{\beta}(y)],\he_{\alpha+\beta}(\kappa xy))=O(\log|x|^{2}+\log|y|^{2}),$$
where $\kappa=\pm 1$ for $\alpha+\beta$ short, $\kappa=\pm 2$ for $\alpha+\beta$ long, and $\kappa=0$ if $\alpha+\beta\notin \Phi$, and we interpret $\he(0)$ as the empty word.
\end{itemize}
\end{lemma}
\begin{proof}
\begin{itemize}
\item[(a)] This follows by switching into any $\SpH_{S,T}$.
\item[(b)] It is a tedious excercise to verify that $e_{\alpha}$ and $e_{\beta}$ both live in some $\SpH_{S,T}$ or $\SL(S)$ with quadratic Dehn function unless $\alpha=2s$ and $\beta=-s+t$ for some $t\neq \pm s$.  (There are eleven cases to check, corresponding the fact that the Weyl group of $\Sp$ has twelve orbits on the set of unordered pairs of roots).\\
In this case, choose $S^{\prime}$ containing $s$ and $T^{\prime}$ containing $t$ with $S=S^{\prime}\cup T^{\prime}$ isotropic, and further choose $T$ symplectic disjoint from $S$.  Observe that $\he_{\beta;S^{\prime},T^{\prime}}(y)$ lives inside $\SL(S)$, while $\he_{\alpha;S,T,t^{\prime}}(x)$ lives in $\SpH_{S,T}$, so we can apply the previous lemma.
\end{itemize}\end{proof}

A simple corollary of this lemma, used in section \ref{section:parabtodiag}, is that we can fill products of shortcuts in any $N_{S,T}$ which represent the identity.
\begin{lemma}
\label{lemma:ptd3}
Suppose $S\subset \mH$ isotropic and $T\subset \mH$ symplectic, and we have a product $w=\he_{\alpha_{1}}(x_{1})\ldots\he_{\alpha_{n}}(x_{n})$
 representing the identity with each $\alpha_{i}\in\Phi N_{S,T}$.  Then we can fill $w$ at cost at most $C_{n}\ell(w)^{2}$, where $C_{n}$ depends only on $n$.
\end{lemma}
\begin{proof}
Fix $n$, so that we write $O(1)$ for quantities which depend only on $n$ and not on $\ell(w)$.  Observe that every commutator $[e_{\alpha},e_{\beta}]$ lives in $Z_{S,T}$, and that $w$ projects to a word representing the identity in $A_{S,T}$.  Suppose $\alpha\in S-T$, and $\he_{\alpha}(x_{i_{1}}),\ldots,\he_{\alpha}(x_{i_{k}})$ are the $\alpha$ shortcuts appearing in $w$.  Then because $w$ projects to the identity in $A_{S,T}$, we know $x_{i_{1}}+\ldots+x_{i_{k}}=0$.

For each $\alpha\in S-T$, move all the $\he_{\alpha}(x_{i})$ shortcuts to the left with Lemma \ref{lemma:steinberg}, possibly introducing $O(1)$ shortcuts of the form $\he_{\beta}(x)$ where $\beta\in\Phi Z_{S,T}$, then fill the resulting product $\he_{\alpha}(x_{i_{1}})\ldots\he_{\alpha}(x_{i_{k}})$ with the same lemma.  This process will result in a product of $O(1)$ shortcuts for elements of $Z_{S,T}$, which can be filled by the same process, i.e., for each $\alpha\in\Phi Z_{S,T}$, move all the $\he_{\alpha}(x_{i})$ to the left and fill using Lemma \ref{lemma:steinberg}.\end{proof}

\subsection{Lubotzky, Mozes, and Raghunathan}
\label{subsection:lmr}
The following theorem says that any matrix in the diagonal block $\Sp(T)$ can be efficiently expressed as a product of shortcuts of the form $\he_{\alpha}(x)$ where $\alpha\in \Phi\Sp(T)$ (i.e., $e_{\alpha}\in\Sp(T)$).  This is due to Lubotzky, Mozes, and Raghunathan\cite{lmr}, but in order to obtain the precise statement of the theorem, we give a proof based on \cite{zakiryanov}, which proves bounded generation for $\Sp(2p;\ZZ)$.  This result is used implicitly in the next subsection, when we define a normal form for elements of $P_{S,T}$.

\begin{theorem}
\label{theorem:LMRSp}
Let $T\subset \Phi$ be symplectic.  If $w$ is a word in the generators of $\Sp(2p;\ZZ)$ representing some $M\in\Sp(T;\ZZ)$, then we can find a sequence of roots $\alpha_{1},\ldots,\alpha_{k}\in T$ and integers $x_{1},\ldots,x_{k}$ such that
$$e_{\alpha_{1}}(x_{1})\ldots e_{\alpha_{k}}(x_{k})=M$$
and
$$\ell(\he_{\alpha_{1}}(x_{1})\ldots\he_{\alpha_{k}}(x_{k}))=O(\ell(w))$$
\end{theorem}

\begin{proof}
First note that there is some constant $C>0$ such that
$$\ell(w)\geq C\log(\| M \|_{\infty}),$$
as in the proof of \ref{lemma:efficientshortcuts}, so it suffices to show that our product of shortcuts has length $O(\log(\| M \|_{\infty}))$.

We begin with the case $\#T=2$, so that $T=\{\pm \alpha\}$ for some $\alpha\in \Phi$.  Recall that the Euclidean algorithm determines the greatest common divisor of two positive integers by iteratively reducing one by the other.  More precisely, given $a>b\in \NN$, define sequences $a_{i}$ and $b_{i}$ by letting $a_{0}=a$, $b_{0}=b$, and
$$a_{i+1}=b_{i}$$
$$b_{i+1}=a_{i}-b_{i}\lfloor\frac{a_{i}}{b_{i}}\rfloor$$
so that $b_{i+1}$ is the remainder when we divide $b_{i}$ into $a_{i}$.  Lam\'e \cite{lame} proved that for some $k=O(\log(b))$ we get $a_{k+1}=1$ and $b_{k+1}=0$.  Observe that
$$\begin{bmatrix}b_{i+1}\\a_{i+1}\end{bmatrix}=
\begin{bmatrix}1 & -  \lfloor\frac{a_{i}}{b_{i}}\rfloor\\ 0 & 1 \end{bmatrix}
\begin{bmatrix}a_{i} \\ b_{i}\end{bmatrix}.$$

Now, suppose that $M\in \Sp(T)$ is
$$\begin{bmatrix}a & c\\ b & d\end{bmatrix}$$
and take $z_{\alpha},z_{-\alpha}$ as a basis for $\RR^{T}$, so that
$$e_{\alpha}= \begin{bmatrix}1 & 1\\ 0 & 1\end{bmatrix}$$
$$e_{-\alpha}=e_{\alpha}^{T}$$
We see that
$$e_{(-1)^{k+1}\alpha}(\lfloor\frac{a_{k}}{b_{k}}\rfloor)
e_{(-1)^{k}\alpha}(\lfloor\frac{a_{k-1}}{b_{k-1}}\rfloor)
\ldots
e_{\alpha}(\lfloor\frac{a_{0}}{b_{0}}\rfloor)
\begin{bmatrix} a & c\\ b & d\end{bmatrix}$$
$$=\begin{bmatrix} 1 & c^{\prime}\\ 0 & 1\end{bmatrix}$$
for some integer $c^{\prime}$.

Since $a_{i+1}=b_{i}$, we know that
$\lfloor\frac{a_{i}}{b_{i}}\rfloor=O(\frac{a_{i}}{a_{i+1}})$.
Thus, $$\ell( e_{(-1)^{k+1}\alpha}(\lfloor\frac{a_{k}}{b_{k}}\rfloor)
e_{(-1)^{k}\alpha}(\lfloor\frac{a_{k-1}}{b_{k-1}}\rfloor)
\ldots
e_{\alpha}(\lfloor\frac{a_{0}}{b_{0}}\rfloor))$$
$$=\sum_{i=0}^{k}O(\log\frac{a_{i+1}}{a_{i}}+1)$$
$$=O(\log(\text{max}\{a,b,c,d\}))=O(\ell(w)).$$
This implies that $\log(c^{\prime})=O(\ell(w))$ as well, so that $w$ can be expressed as a product of shortcuts with total length $O(\ell(w))$.\\

Now we assume $\#T>2$ and reduce to the previous case (our attack is copied from \cite{zakiryanov}).  Let $w$ be a word representing some matrix $M$ in $\Sp(T)$.  We will multiply $M$ by $O(1)$ elementary matrices of the form $e_{\alpha}(O(\ell(w)))$ and obtain a matrix $M^{\prime}\in \Sp(\pm t)$ for some $t\in T$.  Then $\log\Vert M^{\prime}\Vert_{\infty}=O(\ell(w))$, so we will be done by the previous case.

Beginning with $M$, we repeat the following four step procedure until we obtain the desired $M^{\prime}$.  (Each step will be accomplished by multiplying by a bounded number of elementary matrices each of which has $\|\cdot\|_{\infty}$-norm of order $\exp(O(\ell(w)))$.  Furthermore, each step increases the $\|\cdot\|_{\infty}$ by at most some constant factor).
\begin{itemize}
\item[1:] Choose $t$ in $T$.  Reduce $M$ to a matrix $M_{1}$ such that the projection of $M_{1}z_{t}$ on $\RR^{T\setminus \{-t\}}$ is unimodular (i.e., its coordinates are coprime).
\item[2:] Reduce $M_{1}$ to a matrix $M_{2}$ such that $M_{2}z_{t}$ has $z_{t}$ coordinate $1$.
\item[3:] Reduce $M_{2}$ to a matrix $M_{3}$ such that $M_{3}z_{t}=z_{t}$
\item[4:] Reduce $M_{3}$ to a matrix in $\Sp(T\setminus\{\pm t\})$
\end{itemize}

{\it Step 1:} For each $s\in T$, define $m_{s}$ to be the $z_{s}$ coordinate of $Mz_{t}$.  Let $p_{1},\ldots,p_{a}$ be the primes which divide all of $\{m_{s}:s\in T\setminus \{-t\}\}$ and let $q_{1},\ldots,q_{b}$ be the primes which divide all of $\{m_{s}:s\in T\setminus \{\pm t\}\}$ but not $m_{t}$.  Note that $m_{-t}$ cannot be divisible by any $p_{i}$, because $Mz_{t}$ is unimodular.  The Chinese remainder theorem provides a $c\in \NN$ of size $O(\| M \|_{\infty}^{2p})$ such that $c$ is divisible by all the $q_{i}$, but congruent to $1$ mod all of the $p_{i}$.  Then we know that the collection
$$\{m_{s}:s\in T\setminus \{\pm t\}\}\cup \{m_{t}+cm_{-t}\}$$
is coprime (any common factor must be either a $p_{i}$ or $q_{i}$, but $m_{t}+cm_{-t}$ cannot be divisible by these primes).  Hence we can multiply $M$ on the left by $e_{-2t}(\omega(z_{t},z_{-t})c)$ and obtain the desired $M_{1}$.\\

{\it Step 2:} Now, for each $s \in T$, let $m_{s}$ be the $z_{s}$ coordinate of $M_{1}z_{t}$.  We know that there exist constants $a_{s}$ of size $O(\| M_{1}\|_{\infty}^{2p})$ such that
$$\sum_{s\in T\setminus\{-t\}}a_{s}m_{s}=1.$$
Hence we have that
$$m_{-t}+\sum_{s\in T\setminus\{-t\}}a_{s}m_{s}(1-m_{-t})=1.$$
Letting $\kappa_{s}=\omega(z_{-s},z_{s})$, we see that
$$M_{1}^{\prime}:=\prod_{s\in T\setminus\{-t\}}e_{-s-t}(\kappa_{s}a_{s}(1-m_{-t}))M_{1}z_{s}$$
has entries of size $O(\|M_{1}\|_{\infty}^{4p^{2}})$ and has $z_{-t}$ coordinate equal to
$$m_{-t}+\sum_{s\in T\setminus\{-t\}}a_{s}m_{s}(1-m_{-t})
+\sum_{s\in (T\setminus\{\pm t\})^{+}}\omega(z_{s},z_{-s}) a_{s}a_{-s}m_{t}.$$
$$=1+\sum_{s\in (T\setminus\{\pm t\})^{+}}\pm a_{s}a_{-s}m_{t}$$
(the sign $\omega(z_{s},z_{-s})$ is only correct if one uses the natural order of multiplication, but it is entirely irrelevant).  Thus if we let
$$M_{1}^{\prime\prime}=e_{-2t}(\sum_{s\in (T\setminus\{-t\})^{+}}\pm a_{s}a_{-s})M_{1}^{\prime},$$
we see that $M_{1}^{\prime\prime}z_{t}$ has $z_{-t}$ coordinate $1$ (for the appropriate choice of signs) and
$$\|M_{1}^{\prime\prime}\|_{\infty}=O(\|M_{1}\|_{\infty}^{8p^{2}}).$$
Finally, we let
$$M_{2}=e_{2t}(1-x)M_{1}^{\prime\prime},$$
where is the $z_{t}$ coordinate of $M_{1}^{\prime\prime}$.\\

{\it Step 3:} For each $s \in T$, let $m_{s}$ be the $z_{s}$ coordinate of $M_{2}z_{t}$, noting that $m_{t}=1$. Letting
$$M_{2}^{\prime}=\prod_{s\in T\setminus\{t\}} e_{t-s}(m_{s})M_{3}$$
we see that
$$M_{2}^{\prime}z_{t}
=z_{t}+\sum_{s\in (T\setminus\{\pm t\})^{+}}\pm m_{s}m_{-s}m_{t}z_{-t}$$
(where the signs are irrelevant and depend on the order of multiplication).  As expected, we multiply this matrix by
$$e_{-2t}(\pm \sum_{s\in (T\setminus\{\pm t\})^{+}}\pm m_{s}m_{-s}m_{t}z_{-t})$$
to obtain $M_{3}$.\\

{\it Step 4:} We know that a matrix $A\in \Sp(T)$ will lie in $\Sp(T\setminus \{\pm t\})$ if and only if $Az_{t}=z_{t}$ and $z_{t}^{\text{T}}A=z_{t}^{\text{T}}$.  If $m_{s}$ denotes the $z_{s}$ coordinate of $z_{t}^{\text{T}}A$, we can multiply $M_{3}$ on the right by matrices of the form $e_{t-s}(m_{s})$ and then by one matrix of the form $e_{2t}(x)$ to produce $M_{4}$ with the desired property.\end{proof}

\subsection{Defining a normal form for $P_{S,T}$.}
\label{subsection:omegapst}
One of the most important roles of shortcuts is that they allow us to efficiently represent elements of the unipotent radical of $P_{S,T}$, which, combined with the following representation for words in the diagonal blocks, lets us define a normal form $\Omega_{P_{S,T}}$ for $P_{S,T}$ itself.

\begin{lemma}
We say a word $w$ is a $C$-efficient representation of $g\in \Sp(2p,\ZZ)$ if $\ell(w)<C\ell(w^{\prime})$ for any word $w^{\prime}$ representing $g$.  There is some $C>0$ such that the following hold.
\begin{itemize}
\item If $S\subset \mH$ isotropic and $\# S\geq 3$, then any $g\in \GL(S)$ has a $C$-efficient representation by a word in $\GL(S)$. (Recall definition
\ref{definition:shortcutword}).
\item If $T\subset \mH$ symplectic and $\# T\geq 4$, then $T$ has a $C$-efficient representation by a word in $\Sp(T)$.
\item If $S\subset \mH$ isotropic and $\# S=2$, then any $g\in \GL(S)$ has a $C$-efficient representation by a shortcut word in $\GL(S)$.
\item If $T\subset\mH$ symplectic and $\# T=2$, than any $g\in \Sp(T)$ has a $C$-efficient representation by a shortcut word in $\Sp(T)$.
\end{itemize}
\end{lemma}
\begin{proof}
This follows from the result of Lubotzky Mozes and Raghunathan \cite{lmr}.  Of course, Lemma \ref{theorem:LMRSp} suffices for the $\Sp(T)$ cases.  The reader may read the proof of \cite[Theorem 4.1]{riley} for the $\GL(S)$ cases.\end{proof}

Now we define $\Omega_{P_{S,T}}$.  Let $S,T\subset \mH$ with $S$ isotropic and $T$ symplectic.  Let $\phi_{\GL}:P_{S,T}\rightarrow \GL(S)$ denote projection and similarly for $\phi_{\Sp}:P_{S,T}\rightarrow \Sp(T)$.\\

If $\#S\geq 3$ and $\#T\geq 4$, define the normal form
$$\Omega_{P_{S,T}}:P_{S,T}(\ZZ)\rightarrow \Sigma_{\Sp(2p;\ZZ)}^{\ast}$$
by setting $\Omega(g)=den$, where $d$ and $e$ and geodesics in $\GL(S)$ and $\Sp(T)$ respectively representing $\phi_{\GL}(g)$ and $\phi_{\Sp}(g)$, and $n$ is a product
$$\prod_{\alpha\in \Phi N_{S,T}}\he_{\alpha}(x_{\alpha})$$
which represents $(\phi_{\GL}(g)\phi_{\Sp}(g))^{-1}g$ efficiently.\\

If $\#S=2$, we define $\Omega$ similarly, except we must take $d$ to be a shortcut word in $\GL(S)$. Similarly, if $\# T=2$, we must take $e$ to be a shortcut word in $\Sp(T)$.

\section{From diagonal blocks to parabolics}
\label{section:diagtoparab}
The goal of this section is to prove Theorem \ref{theorem:diagtoparab}. For simplicity, here is a restatement of this theorem.

\begin{nolabeltheorem}
Suppose $p\geq q\geq 2$.  If $w$ a relation in $\Sp(2q;\ZZ)$ (i.e., a word in the generators of $\Sp(\{\pm [1],\ldots,\pm [q]\};\ZZ)$ representing the identity), then $w$ can be quadratically broken into a collection of relations $w_{1},\ldots,w_{n}$ such that each $w_{i}$ is either a bounded length relation in $\Sp(2q)$ or an $\Omega$ triangle in some maximal parabolic of $\Sp(2q)$.
\end{nolabeltheorem}

This looks a little different than the original statement of the theorem, but it is actually the same because we can immediately fill all of the bounded length relations at cost $O(\ell(w)^{2})$.

\paragraph{An outline of the proof of Theorem \ref{theorem:diagtoparab}.}
The proof of the theorem proceeds by taking a relation $w$, representing it by a loop in the symmetric space $\mathcal{E}=\Sp(2q;\RR)/U(q)$, filling this loop with a Lipschitz disk (since $\mathcal{E}$ is CAT(0)), triangulating this disk so that each triangle lies inside some horoball (or in the thick part), and using this triangulation to break $w$ as desired.

We begin by proving some results about the basic structure of $\mathcal{E}$. Theorem \ref{theorem:siegelsets} describes a fundamental domain $\mathcal{S}$ for the action of $\Sp(2p;\ZZ)$ on $\mathcal{E}$, and hence a map $\rho:\mathcal{E}\rightarrow \Sp(2p;\ZZ)$.  We then prove a series of technical results culminating in \ref{corollary:parabolicedges}, which says that if the distance between points $x,y\in\mathcal{E}$ is small relative to their distance from the thick part, then $\rho(x)^{-1}\rho(y)$ lies in some maximal parabolic. (Essentially, $x$ and $y$ lie in a horoball centered at some $\xi\in\partial X$, and $\rho(x)^{-1}\rho(y)$ must fix $\rho(y)^{-1}\xi$).  We then apply the adaptive template lemma \ref{lemma:adaptivetemplates}\cite[Corollary 5.3]{RY} to show that a Lipschitz disk in $\mathcal{E}$ can be triangulated in such a way that the edges have length short enough for corollary \ref{corollary:parabolicedges} to apply, but long enough that
$$\ell(\Omega(\rho(x)^{-1}\rho(y)))=O(d(x,y)).$$
This implies that $w$ can be broken into a product of $\Omega$ triangles as desired.

It is worth noting that this theorem is analogous to \cite[Lemma 3.1]{RY}, and the proof is nearly identical. The results of subsection \ref{subsection:reductiontheory} are due to various authors including Ji and Macpherson \cite{jimac}.  The adaptive template concept is due to Young.

\subsection{Siegel Sets.}
\label{subsection:siegelsets}
Let $\Phi^{+}$ be the set of positive roots, i.e., all roots which are nonnegative integer linear combinations of the set of simple roots
$$\Sigma=\{ [1]-[2],\ldots,[q-1]-[q],2[q]  \}.$$
Let $N$ be the subgroup generated by $\{e_{\alpha}|\alpha\in \Phi^{+}\}$.  For example, if $q=3$, then $N$ consists of all symplectic matrices of the form
$$\begin{bmatrix}
1 & \ast & \ast & \ast & \ast & \ast\\
0 & 1    & \ast & \ast & \ast & \ast\\
0 & 0    & 1    & \ast & \ast & \ast\\
0 & 0    & 0    & 1    & 0    & 0   \\
0 & 0    & 0    & \ast & 1    & 0   \\
0 & 0    & 0    & \ast & \ast & 1   
\end{bmatrix}$$
Let $A_{\epsilon}=\{\diag_{\Sp}(a_{1},\ldots,a_{q})|a_{i}>\epsilon a_{i+1}; a_{q}>\sqrt{\epsilon}\}$.  (Note that $\diag_{\Sp}(a_{1},\ldots,a_{q})\in A_{\epsilon}$ if and only if $\sigma(\log(a))>\log(\epsilon)$ for all $\sigma\in \Sigma$).  The following theorem (see \cite[Prop. 4.4]{jimac}) is classical.

\begin{theorem}
\label{theorem:siegelsets}
For some $\epsilon\in(0,1)$ and compact $N^{+}\subset N$, the set $\mathcal{S}=[N^{+}A_{\epsilon}]_{\mathcal{E}}$ is a coarse fundamental domain, i.e.,
\begin{itemize}
\item $\Gamma\mathcal{S}=\mathcal{E}$
\item For $x\in \mathcal{E}$, there are only finitely many $\gamma\in \Gamma$ with $x \in \gamma\mathcal{S}$.
\end{itemize}
\end{theorem}

\begin{definition}
\label{def:siegelrhodepth}
\begin{itemize}
\item Fix $N^{+}$ and $\epsilon$ as in the theorem.  We call $\mathcal{S}=[N^{+}A_{\epsilon}]_{\mathcal{E}}$ a Siegel set.
\item Define $\rho:\mathcal{E}\rightarrow \Gamma$ to be some function such that $x\in \rho(x)\mathcal{S}$ for all $x$.  (For each $x$, there are only finitely many choices).
\item Define $\phi:\mathcal{E}\rightarrow A$ by $[x]_{\mathcal{E}}=[\rho(x)n\phi(x)]_{\mathcal{E}}$ for some $n\in N^{+}$.  (This is well defined because the map $A_{\epsilon}\times N^{+}\rightarrow \mathcal{S}$ is injective, since a product of $a\in A_{\epsilon}$ and $n\in N^{+}$ has the same diagonal entries as $a$).
\item Define $\mathcal{M}$ to be the metric space $\mathcal{E}/\Gamma$, and define the depth function $r:\mathcal{E}\rightarrow [0,\infty)$ by $r(x)=d_{\mathcal{M}}([x]_{\mathcal{M}},1)$.  Observe that this function is $1$-Lipschitz by definition.
\end{itemize}
\end{definition}

As we shall see, for any $x\in\mathcal{E}$, we have that $\phi(x)$ describes the geometry of the lattice $\ZZ^{2q}x\subset\RR^{q}$.  Thus, the following lemma implies that the geometry of this lattice cannot change very fast as we vary $x$.

\begin{lemma}
\label{lemma:philip}
$d_{A}(\phi(x),\phi(y))<d(x,y)+O(1)$.
\end{lemma}
\begin{proof}
This follows from the discussion in \cite[\S 5]{jimac}.\end{proof}

\subsection{Technical results of reduction theory.}
\label{subsection:reductiontheory}
\begin{definition}
\label{def:rshort}
Given $x\in X$ and $r\in\mathbb{R}$, let $V(x,r)\subset \mathbb{Z}^{2p}$ denote the subgroup of $\mathbb{Z}^{2p}$ generated by vectors $v$ such that $\Vert v\tilde{x} \Vert\leq r$, where $\tilde{x}\in G$ is any lift of $x$.
\end{definition}

The importance of this definition is that if $r(x)\gg 0$ and $d(x,y)$ is less than some constant times $r(x)$, then $V(x,r)$ and $V(y,r)$ will be the same for some $r$ determined by $x$.  Our starting point for showing this is the following lemma (the reader is advised to write out the matrices explicitly if the proof seems confusing):
\begin{lemma}
\label{lemma:rshort}
There is a constant $C>1$ such that
\begin{itemize}
\item[(a)] Given $n\in N^{+}$, a real number $r>0$, and $\diag_{\Sp}(a_{1},\ldots,a_{p})\in A^{+}$ with $a_{i}C^{-1}>r>a_{i+1}C$ for some $i$, we have that $V(na,r)=\langle z_{i+1},\ldots,z_{2p}\rangle$, i.e., the span of the last $2p-i$ standard basis vectors in $\mathbb{Z}^{2p}$.
\item[(b)] If $n\in N^{+}$ and $\diag_{\Sp}(a_{1},\ldots,a_{p})\in A^{+}$ with $a_{p}>C$, then $V(na,1)=\langle z_{p},\ldots,z_{2p}\rangle$, i.e., the span of the last $p$ standard basis vectors in $\mathbb{Z}^{2p}$.
\end{itemize}
\end{lemma}

\begin{proof}
This is an analogue of \cite[Lemma 4.6]{RY}, and in fact it follows from that lemma, but we give a proof anyways, starting with part (b).  We are thinking of $z_{j}$ as a row vector because we are multiplying by $z_{j}$ on the left.  Note that $z_{j}na$ is the $j$th row of $na$.  Suppose that $j>p$, and the $j$th row of $n$ is
$$[0,\ldots, 0,1,n_{j1},\ldots, n_{j2p}]$$
so the $j$th row of $na$ will be
$$[0,\ldots, 0,a_{j-p}^{-1},n_{j1}a_{j-p+1}^{-1},\ldots, n_{j2p}a_{p}^{-1}]$$
Since $a_{2p}=a_{p}^{-1}$, we know by the defining inequalities for elements of $A^{+}$ that $a_{k}<a_{p}^{-1}\epsilon_{\mathcal{S}}^{-p}$.  If we choose $C$ larger than $p\epsilon_{\mathcal{S}}^{-p}\sup_{\tilde{n}\in N^{+}}\Vert\tilde{n}\Vert_{\infty}$, then each entry of $\Vert z_{j}na \Vert$ will have absolute value less than $\frac{1}{p}$, and hence this vector will have norm less than $1$.

On the other hand, if $j<p$, then $z_{j}na$ will have $j$-th coordinate $a_{j}$.  Hence, given $v\notin \langle z_{i+1},\ldots,z_{2p}\rangle$ we have that the $j$-th coordinate of $vna$ is at least $a_{j}$ in absolute value, where $j<p$ is the first nonzero coordinate of $v$.  Consequently,
$\Vert vna \Vert > a_{j} > \epsilon_{\mathcal{S}}^{p}a_{p}$,
so if $C>\epsilon_{\mathcal{S}}^{-p}$, then $\Vert vna \Vert$ will be greater than $1$ as desired.

The proof of (a) is similar.\end{proof}

The following corollary follows immediately because for $\gamma\in\Gamma$, $x\in\mathcal{E}$ and $r>0$, we have $V(\gamma x,r)=V(x,r)\gamma^{-1}$ by definition. 
\begin{corollary}
\label{corollary:rshort}
There is a constant $C>1$ such that
\begin{itemize}
\item[(a)] If $x\in \mathcal{E}$ and $\phi(x)=\diag_{\Sp}(a_{1},\ldots,a_{p})\in A^{+}$ with $a_{i}C^{-1}>r>a_{i+1}C$ for some $i$, then $V(x,r)=\langle z_{i+1},\ldots,z_{2p}\rangle\rho(x)^{-1}$.
\item[(b)] If $x\in \mathcal{E}$ and $\phi(x)=\diag_{\Sp}(a_{1},\ldots,a_{p})\in A^{+}$ with $a_{p}>C$, then $V(x,1)=\langle z_{p},\ldots,z_{2p}\rangle\rho(x)^{-1}$.
\end{itemize}
\end{corollary}

For $j=1,...,p-1$, and $c>0$, let $B_{j}(c)$ be the set of points in $\mathcal{E}$ such that $([j]-[j+1])\log\phi(x)>\log(c)$
Similarly, let $B_{p}(c)$ be the set of $x\in\mathcal{E}$ such that
$2[p]\log\phi(x)>\log(c)$.

\begin{lemma}
\label{lemma:horoballlatticegeometry}
\begin{itemize} There is a constant $C$ such that if $c\gg C$, then we have the following.
\item[(a)] Suppose $1\leq j \leq p-1$.  Let $s(x)$ be the geometric mean of the $j$ and $j+1$st entries of $\phi(x)$, i.e.,
$$s(x)=\exp(\frac{1}{2}([j+1]+[j])\log\phi(x))$$
Then $V(x,s(x))$ is constant on each connected component of $B_{j}(c)$.
\item[(b)] $V(x,1)$ is constant on each connected component of $B_{p}(c)$.
\end{itemize}
\end{lemma} 
\begin{proof}
We will prove (a) first.  Fix any constant $C_{1}>0$.  It suffices to find $c$ such that $V(x,s(x))$ is constant on a ball of radius $C_{1}$ around any point of $B_{j}(c)$.  Take $x,y\in\mathcal{E}$ with $d(x,y)<C_{1}$, so that the distance between $\phi(x)$ and $\phi(y)$ is at most some constant $C_{0}$ by Lemma \ref{lemma:philip}.  Then for all $i$, we have $|[i]\log(\phi(x))-[i]\log(\phi(y))|$ is bounded by some constant that does not depend on $c$.  In particular, taking $i=j,j+1$, we see that if $c$ is sufficiently large, $V(x,s(x))$ and $V(y,s(y))$ will be the same by corollary \ref{corollary:rshort}.

The proof of (b) is similar.\end{proof}

\begin{corollary}
\label{corollary:parabolicedges}
There are constants $0<\Cparab<1$ and $C_{2}>0$ such that if $x,y\in \mathcal{E}$ with $r(x)>C_{2}$ and $d(x,y)<\Cparab r(x)$, then $\rho(y)\in\rho(x)P$ for some maximal parabolic $P$ of $\Gamma$.
\end{corollary}
\begin{proof}
Of course, the goal is to find $C_{1}$ and $C_{2}$ which force $x$ and $y$ to be in the same component of some $B_{j}(c)$.  Recall that $r(x)=\|\log\phi(x)\|_{2}$, and let
$$g(x)=\sup\{\epsilon_{\mathcal{S}}\}\cup\{|\sigma\log\phi(x)|:\sigma\in\Sigma\}$$
so that the entries of $\log\phi(x)$ are bounded by $pg(x)$, and thus $r(x)^{2}\leq 2p(pg(x))^{2}$.  It follows that $r(x)=O(g(x))$, and so there must be a constant $\varepsilon>0$ such that for all $x$ with $r(x)>\epsilon_{\mathcal{S}}$, there is some $\sigma_{x}\in\Sigma$ with $|\sigma_{x}\log\phi(x)|>\varepsilon r(x)$.  We are done because $\sigma\log\phi(y)\geq \sigma\log\phi(x)-d(x,y)+O(1)$ for any $\sigma\in\Sigma$ and $x,y\in \mathcal{E}$, so if we choose $C_{2}$ large enough and $C_{1}$ small enough, we can ensure that $\sigma_{x}\log\phi(x)$ and $\sigma_{x}\log\phi(y)$ are both large enough that we can apply the lemma.\end{proof}

\begin{corollary}
\label{corollary:parabolictriangles}
There are constants $0<C_{1}<C_{2}<1$ and $C_{3},C_{4}>0$ such that if $x,y,z\in\mathcal{E}$ with
\begin{itemize}
\item $r(x)>C_{3}$
\item $diam(\{x,y,z\})<C_{2}r(x)$
\item each of $d(x,y),d(x,z),d(y,z)$ is at least $C_{1}r(x)$
\end{itemize}
then $\rho(x)^{-1}\rho(y),\rho(y)^{-1}\rho(z),\rho(z)^{-1}\rho(x)\in P$ for some maximal parabolic $P$, and the perimeter in $\Gamma$ of $\rho(x),\rho(y),\rho(z)$ is at most $C_{4}$ times the perimeter of $x,y,z$.
\end{corollary}
\begin{proof}
By corollary\ref{corollary:parabolicedges}, we can find $C_{3}$ and $C_{2}$ which ensure that each edge lies in a maximal parabolic.  Now observe that
$$d([\rho(x)]_{\mathcal{E}},[\rho(y)]_{\mathcal{E}})\leq d([\rho(x)]_{\mathcal{E}},x)+d(x,y)+d(y,[\rho(y)]_{\mathcal{E}})\\
\leq r(x)+r(y)+d(x,y)+O(1)=O(r(x)+1)$$
so $d([\rho(x)]_{\mathcal{E}},[\rho(y)]_{\mathcal{E}})$ will be $O(d(x,y))$ (since $d(x,y)>C_{1}r(x)$) and similarly for the other edges.  This implies the desired result by Lubotzky-Mozes-Raghunathan\cite{lmr}.\end{proof}

\subsection{The proof of Theorem \ref{theorem:diagtoparab}.}
\label{subsection:diagtoparabproof}
The following lemma allows us to produce triangulations in which triangles either satisfy the conditions of the previous corollary, or have bounded perimeter.

\begin{lemma}
\label{lemma:adaptivetemplates}
If $N$ is a power of $2$ and
$h:[0,N]^{2}\rightarrow [1,\infty)$ is $1$-Lipschitz then $[0,N]^{2}$ admits a triangulation such that the following conditions all hold.
\begin{itemize}
\item If $x$ and $y$ are the endpoints of some edge, then
$$\min(\frac{1}{6}h(x),\frac{N}{2})\leq d(x,y)\leq \sqrt{2}h(x).$$
\item There are $O(N^{2})$ triangles.
\item The sum of the squared perimeter of all the triangles plus the sum of the squared edge lengths of all the triangles is $O(N^{2})$.
\end{itemize}
\end{lemma}

\begin{proof}
This is \cite[Corollary 5.3]{RY}.\end{proof}

We are finally in a position to prove Theorem \ref{theorem:diagtoparab}.  Let $w$ be a relation in $\Gamma$.  Choose $t$ the smallest power of $2$ greater than $\ell(w)$.  Define a map from the boundary of $[0,t]^{2}$ to $\mathcal{E}$ constant on the top and sides, and sending the bottom to a closed $1$-Lipschitz curve representing $w$.  Extend this map over $[0,t]^{2}$ by mapping line segments going from the midpoint of the top to other points of the boundary to geodesic segments in $\mathcal{E}$.  This yields a $O(1)$-Lipschitz map.  Now applying the previous lemma to an appropriate multiple of $r$ yields a triangulation such that each triangle has edges of length as in the corollary.  But then applying the corollary and representing edges in parabolics by $\Omega$-words gives use the desired $w_{i}$.

\section{From parabolics to diagonal blocks}
\label{section:parabtodiag}
This section is devoted to proving the Theorem \ref{theorem:parabtodiag}, as promised in section \ref{section:theorem}.  Recall that a shortcut word in a subgroup $H$ of $\Sp(2p;\ZZ)$ is a product of shortcuts $\he_{\alpha}(x)$ such that $e_{\alpha}(x)\in H$.  We now recall the statement of \ref{theorem:parabtodiag}.

\begin{nolabeltheorem}
Suppose $\emptyset\neq S\subset \mH$ is isotropic, and $T\subset \mH$ is symplectic with $S,T$ disjoint.  Let $\Delta$ be an $\Omega$-triangle in $P_{S,T}$.  Then we can homotope $\Delta$ at cost $O(\ell(\Delta)^{2})$ as follows (note that all these homotopies have quadratic cost).
\begin{itemize}
\item[(a)] If $\# T > 2(p-3)$, then $\Delta$ can be quadratically broken into a collection of relationss $\{ w_{i}\}$ with each $w_{i}$ an $\Omega$ triangle in some $P_{S_{i},T_{i}}$ with $\# T_{i}$ strictly smaller than $\# T$.
\item[(b)] If $4\leq \# T \leq 2(p-3)$, then $\Delta$ can be homotoped at cost $O(\ell(\Delta)^{2})$ to a relation of length $O(\ell(\Delta))$ in $\Sp(T)$.
\item[(c)] If $\# T = 2$, then $\Delta$ can be homotoped at cost $O(\ell(\Delta)^{2})$ to an identity-representing shortcut word of length $O(\ell(\Delta))$ in $Sp(T)$.
\item[(d)] If $\# T = 0$, then $\Delta$ can be filled at cost $O(\ell(\Delta)^{2})$.
\end{itemize}
\end{nolabeltheorem}

The proof is organized as follows.  In subsection \ref{subsection:movingdiag}, we explain how the results of \S\ref{section:shortcuts} allow us to move words or shortcut words in $\GL(S)\times \Sp(T)$ past normal form words in $N_{S,T}$.  The remaining subsections proves parts (a) through (d) of the theorem successively.  Part (a) is proved using the adaptive template ideas of \cite{RY} (much as we did in \S\ref{section:diagtoparab}).  Parts (b) and (c) are proved using just the techniques of \S\ref{section:shortcuts}.  Part (d) is also proved using adaptive template techniques, combined with Young's result on the Dehn function of $\GL(n;\ZZ)$.  The use of these techniques in parts (a) and (d) is justified by the results of section \ref{section:lipfill}.

\subsection{Moving diagonal block words}
\label{subsection:movingdiag}
$\Omega$-words have the form $den$ where $d$ is a word (or shortcut word) in $\GL(S)$, $e$ is a word or shortcut word in $\Sp(T)$ and $n$ is a product of shortcuts $\prod_{\alpha\in \Phi_{S,T}}\he_{\alpha}(x_{\alpha})$ (in some order).  In order to prove the theorem, we will need the following lemma, which allow us to move these words past each other.  The main tool used in proving this result is Lemma \ref{lemma:blockshortcut}.

\begin{lemma}
\label{lemma:ptd1}
Let $S,T\subset \mH$ be disjoint with $S$ isotropic and $T$ symplectic.  Suppose $\alpha\in \Phi N_{S,T}$, and $d$ is a word representing some $M\in \GL(S)\times\Sp(T)$ with
$$d(\prod_{\alpha\in \Phi N_{S,T}}e_{\alpha}(x_{\alpha}))d^{-1}=\prod_{\alpha\in \Phi N_{S,T}}e_{\alpha}(x_{\alpha}^{\prime}).$$
Assume that one of the following conditions holds.

\begin{itemize}
\item[(a)] $d$ is a word in the generators of $\GL(S)$, with $3 \leq \# S \leq p-1$.
\item[(b)] $d$ is a word in the generators of $\Sp(T)$, with $4 \leq \# T \leq 2(p-3)$.
\item[(c)] $d$ is a shortcut word in $\GL(S)$, with $\# S=2$.
\item[(d)] $d$ is a shortcut word in $\Sp(T)$, with $\# T=2$.
\end{itemize}

Then
$$\delta\left(d\left(\prod_{\alpha\in \Phi N_{S,T}}e_{\alpha}(x_{\alpha})\right)d^{-1},\prod_{\alpha\in \Phi N_{S,T}}e_{\alpha}(x_{\alpha}^{\prime})\right)=O\left(\ell(d)^{2}+\left(\sum\log\|x_{\alpha}\|\right)^{2}\right).$$
\end{lemma}

\begin{proof}
We will tackle the conditions one at a time.

\begin{itemize}
\item[(a)] Let $x_{\alpha}^{\beta}$ be such that
$$de_{\alpha}(x_{\alpha})d^{-1}=_{G}\prod_{\beta\in \Phi N_{S,T}}\he_{\beta}(x_{\alpha}^{\beta}).$$
Homotope as follows (with each step being quadratic cost for the given reason).\\

\begin{tabular}{l l}
$d(\prod_{\alpha\in \Phi N_{S,T}}\he_{\alpha}(x_{\alpha}))d^{-1}$ & \\
$\leadsto\prod_{\alpha\in \Phi N_{S,T}}d\he_{\alpha}(x_{\alpha})d^{-1}$ & free insertion\\
$\leadsto\prod_{\alpha\in \Phi N_{S,T}}\prod_{\beta\in \Phi N_{S,T}}\he_{\beta}(x_{\alpha}^{\beta})$ & Lemma \ref{lemma:blockshortcut}\\
$\leadsto\prod_{\alpha\in \Phi N_{S,T}}\he_{\alpha}(x_{\alpha}^{\prime})$ & Lemma \ref{lemma:ptd3}.\\
\end{tabular}\\

\item[(b)] We may assume that $\pm S\cup T=\mH$ (since a word in the generators of $\GL(S)$ is, without loss of generality, also a word in the generators of $\GL(S^{\prime})$ if $S^{\prime}\supset S$).  Since $\# T\leq 2(p-3)$, we can use the block-shortcut Lemma \ref{lemma:blockshortcut} exactly as in part (a).

\item[(c)] For each $\alpha\in\Phi N_{S,T}$, we can move $d$ past $\he_{\alpha}(x_{\alpha})$ by switching the shortcuts involved in $d$ to some pair $(S, \{s^{\prime}\})$ where $\alpha\notin\Phi\GL(S\cup \{s^{\prime}\})$, then applying the block shortcut lemma.  By doing this in turn for each $\alpha$, then filling the resulting product of $O(1)$ shortcuts with \ref{lemma:ptd3}, we obtain the desired homotopy.

\item[(d)] The proof is the same as part (c).
\end{itemize}\end{proof}

\subsection{Proof of Theorem \ref{theorem:parabtodiag}, part (a).}
Now we will prove part (a) of the theorem, which is perhaps the hardest part. Recall that we are given $\Delta$, an $\Omega$ triangle in $P_{S,T}$ with $T>2(p-3)$, and we wish to quadratically break $\Delta$ into a collection of words $\{ w_{i}\}$, so that each $w_{i}$ is a word in a parabolic with smaller $\Sp$ block and the total squared length of the $w_{i}$ is at most quadratic in $\ell(w)$.  We proceed through the following steps.

\begin{itemize}
\item Homotope $\Delta$ by moving all the $\GL$ part words to the left with the previous lemma and filling their product.
\item Homotope the resulting word into $K=ker(P_{S,T}\rightarrow \GL(S))$.
\item Thinking of the resulting word as a loop in $\mathcal{E}=K(\RR)/U(T)$, apply the adaptive template Lemma \ref{lemma:adaptivetemplates} to break it into a collection $\{w_{i}\}$ of $\Omega$ triangles in parabolics $P_{S_{i},T_{i}}$ where $T_{i}$ is a proper subset of $T$ (as desired).
\end{itemize}

\paragraph{Handling the $\GL$-block words.} Let
$$\Delta=d_{1}e_{1}n_{1}d_{2}e_{2}n_{2}d_{3}e_{3}n_{3},$$
where each $d_{i}$ is a word (or shortcut word) in $\GL(S)$, each $e_{i}$ is a word in $\Sp(T)$ and each $n_{i}$ is a product $\prod_{\alpha\in \Phi_{S,T}}\he_{\alpha}(x_{\alpha})$.  We will homotope $\Delta$ to a word of the form $e_{1}n_{1}^{\prime}e_{2}n_{2}^{\prime}e_{3}n_{3}^{\prime}$ (where the $n_{i}^{\prime}$ are also products of a bounded number of shortcuts for elements of $N_{S,T}$) as follows.\\

If $\# S\geq 3$, each $d_{i}$ is a word in $\GL(S)$, so by repeatedly applying Lemma \ref{lemma:ptd1} and the fact that generators of $\GL(S)$ commute with generators of $\Sp(T)$, we can move the $d_{i}$ to the left, homotoping $\Delta$ to a word
$$d_{1}d_{2}d_{3}e_{1}n_{1}^{\prime}e_{2}n_{2}^{\prime}e_{3}n_{3}^{\prime},$$
where each $n_{i}^{\prime}$ is a product of a bounded number of shortcuts $\he_{\alpha}(x_{\alpha})$ where $\alpha\in \Phi N_{S,T}$.  We know $d_{1}d_{2}d_{3}$ represents the identity and lives in some copy of $\GL(p)$, so it can be filled by \cite{RY}.

If $\# S=2$, switch each shortcut $\he_{s-s^{\prime}}(x)$ involved in $d_{i}$ to a special shortcut of the form $\he_{\alpha;\{s\},\{\pm s^{\prime}\}}(x)$.  This special shortcut is contained in $\Sp(\pm S)$, whose generators commute with those of $\Sp(T)$, so we can commute each $d_{i}$ with each $e_{i}$.  Furthermore, for each $\alpha\in N_{S,T}$. switch each shortcut in $d_{i}$ 

If $\# S=1$, there is nothing to do.\\

\paragraph{Homotoping into $K$.} We have homotoped $\Delta$ to a word $\Delta^{\prime}=e_{1}n_{1}^{\prime}e_{2}n_{2}^{\prime}e_{3}n_{3}^{\prime}$.  We wish to homotope $\Delta^{\prime}$ to a word in $K=\ker(P_{S,T}\rightarrow\GL(S)$ (i.e., a word in some finite generating set of $K$, say the one given by taking the union of a generating set of $\Sp(T)$ and a generating set of $N_{S,T}$).  Obviously the $e_{i}$ are already in $K$, so it suffices to put the $n_{i}^{\prime}$ into it as well.  Suppose $\he_{\alpha}(x)$ is a shortcut appearing in some $n_{i}^{\prime}$.  If $\alpha$ is short, we can use Lemma \ref{lemma:shortrootswitching} to switch $\he_{\alpha}(x)$ to either $\he_{T^{+},S}(x)$ or $\he_{T^{-},S}(x)$.  The resulting shortcut lies in $K$ because the distorting matrices will lie in $\GL(T^{\pm})\subset\Sp(T)$.  If $\alpha$ is long, we can just switch to a special shortcut $\he_{\alpha;\{s\},\{\pm t\},t}(x)$, whose distorting matrices lie in $\Sp(\pm t)\subset\Sp(T)$.  By abuse of notation, let $\Delta$ denote the word in $K$ resulting from our manipulations.\\

\paragraph{Applying adaptive templates.} Let
$$\mathcal{E}=\Sp(T;\RR)\ltimes N_{S,T}(\RR)/U(T),$$
where $U(T)\subset \Sp(T;\RR)$ is a maximal compact subgroup, and let $X$ be the symmetric space $\Sp(T)/U(T)$.  Let $r:\mathcal{E}\rightarrow \RR^{\geq 0}$ be the composition of projection $\mathcal{E}\rightarrow X$ with the depth function $X\rightarrow \RR^{\geq 0}$.  By Theorem \ref{theorem:lipfillsp}, proved in section \ref{section:lipfill}, any $C$-Lipschitz loop in $\mathcal{E}$ has a $O(C)$-Lipschitz filling.

Pad $\Delta$ so that its length is a power of $2$.  Let $\gamma:[0,N]\rightarrow\mathcal{E}$ be a ($1$-Lipschitz) loop representing $\Delta$.  Extend $\gamma$ over the boundary of $[0,N]^{2}$ by making it constant on the other three sides, let $f:[0,N]^{2}\rightarrow\mathcal{E}$ be a $O(1)$-Lipschitz filling of $\gamma$, and choose $0<C_{0}<\frac{\Cparab}{\Lip(f)\sqrt{2}}$ independent of $f$ (where $\Cparab$ is the constant of corollary \ref{corollary:parabolicedges}, as it applies to $X$).

Let $h:[0,N]^{2}\rightarrow[1,\infty)$ be given by
$$x\mapsto\max(1,C_{0}r(x)).$$
Then $h$ is certainly $1$-Lipschitz (since $\Cparab<1$) so we can apply Lemma \ref{lemma:adaptivetemplates} to obtain a triangulation of $[0,N]^{2}$ with $O(N^{2})$ triangles such that an edge with vertices $x$ and $y$ has length at most $\Cparab r(x)$ (when $r(x)>1$) and at least $\frac{1}{6}$, and the sum of the squares of all the edge lengths plus the diameters of all triangles is $O(N^{2})$.

Let $\phi:K\rightarrow \Sp(T)$ be projection.  Extend the function $\rho:X\rightarrow \Sp(T;\ZZ)$ of definition \ref{def:siegelrhodepth} to
$$\rho:\mathcal{E}\rightarrow K(\ZZ)$$
$$\rho:g\rightarrow \rho(\phi(g))n$$
where $n\in N_{S,T}(\ZZ)$ is the closest element to $\phi(g)^{-1}g\in N_{S,T}(\RR)$.  If $x_{1},x_{2},x_{3}$ are the vertices of a triangle in our triangulation, then by corollary \ref{corollary:parabolicedges}, either all three $\phi(\rho(x_{i})^{-1}\rho(x_{j}))$ lie in some (proper) maximal parabolic of $\Sp(T)$, or they are all bounded.  In order to break $\Delta$ into words as desired, we need a word representing each edge.  If some edge has vertices $x$ and $y$ and $g=\rho(x)^{-1}\rho(y)$ with $\phi(g)$ bounded, we represent $g$ by a word $\Omega(\phi(g))\prod_{\alpha\in\Phi N_{S,T}}\he_{\alpha}(x_{\alpha})$ where $\Omega(\phi(g))$ is a product of a bounded number of elementary generators in $\Sp(T)$. On the other hand, if $g$ lives in some parabolic $P_{S^{\prime},T^{\prime}}\subset \Sp(T)$, we represent $g$ as a word $den\tilde{n}$ where $d$ is a word (or shortcut word) in $\GL(S^{\prime})$, $e$ is a word in $\Sp(T^{\prime})$, $n$ is a product $\prod_{\alpha\in\Phi N_{S^{\prime},T^{\prime}}}\he_{\alpha}(x_{\alpha})$, and $\tilde{n}$ is a product $\prod_{\alpha\in\Phi N_{S,T}}\he_{\alpha}(x_{\alpha})$.  (In either case, our representation has length bounded by some constant times $d_{\mathcal{E}}(x,y)$).

We have broken $\Delta$ into a product of triangles (i.e., words representing the identity which have the form $w_{1}w_{2}w_{3}$, where each $w_{i}$ is a representative of some $\rho(x)^{-1}\rho(y)\in K(\ZZ)$ as above).  Furthermore by corollary \ref{corollary:parabolictriangles}, for each triangle $w_{1}w_{2}w_{3}$, we either have that the $w_{i}$ lie in some $P_{S^{\prime},T_{\prime}}$ or that they have bounded length.  We can fill any triangle whose edges have bounded $\phi$ image by using Lemma \ref{lemma:steinberg} to move the elementary generators in $\Sp(T)$ to the left, filling the resulting bounded word in $\Sp(T)$, then filling the remaining product of shortcuts in $N_{S,T}$ with Lemma \ref{lemma:ptd3}.

Any triangle whose edges have $\phi$ image in $P_{S^{\prime},T^{\prime}}$ looks like 
$$d_{1}e_{1}n_{1}\tilde{n}_{1} d_{2}e_{2}n_{2}\tilde{n}_{2} d_{3}e_{3}n_{3}\tilde{n}_{3}.$$
(Where $d_{i},e_{i},n_{i},\tilde{n}_{i}$ are as explained above).  We can move the $d_{i}$ to the left and fill $d_{1}d_{2}d_{3}$ by the standard argument.  This leaves a word of the form
$$e_{1}n_{1}^{\prime}\tilde{n}_{1}^{\prime} e_{2}n_{2}^{\prime}\tilde{n}_{2}^{\prime} e_{3}n_{3}^{\prime}\tilde{n}_{3}^{\prime}$$
with $n_{i}^{\prime}$ of the form $\prod_{\alpha\in\Phi N_{S^{\prime},T^{\prime}}}\he_{\alpha}(x_{\alpha})$ and $\tilde{n}_{i}^{\prime}$ of the form $\prod_{\alpha\in\Phi N_{S,T}}\he_{\alpha}(x_{\alpha})$.  If we can homotope this word to an $\Omega$-triangle in $P_{S\cup S^{\prime},T^{\prime}}$, we will be finished.

Suppose $\he_{\alpha}(x)$ is a shortcut appearing in $n_{i}^{\prime}$ with $\alpha\notin \Phi N_{S\cup S^{\prime},T^{\prime}}$.  Then $\alpha$ must lie in $S-S^{\prime}$, so write $\alpha=s-s^{\prime}$.  We will move $\he_{\alpha}(x)$ to the left as follows.  We can use Lemma \ref{lemma:steinberg} to move $\he_{\alpha}(x)$ past any shortcut in $n_{j}$ or $n_{j}^{\prime}$, perhaps leaving behind another shortcut $\he_{\beta}(y)$ with $\beta\in \Phi N_{S\cup S^{\prime},T^{\prime}}$.  If we switch $\he_{\alpha}(x)$ to a special shortcut $\he_{\alpha;\{s\},\{\pm s^{\prime}\}}$, we can move it past $e_{j}$ because this special shortcut is a word in the generators of $\Sp(\{\pm s, \pm s^{\prime}\})$, which commute with the generators of $\Sp(T^{\prime})$.

By moving all of these $\he_{\alpha}(x)$ (such that $\alpha\in S-S^{\prime}$ to the left, we end up with a product of a bounded number of shortcuts representing unipotents in $\GL(S\cup S^{\prime})$ followed by an $\Omega$-triangle in $P_{S^{\prime},T^{\prime}}$.  We will be done if we can fill this product of shortcuts.  Choose a Lagrangian $S^{\prime\prime}\supset S\cup S^{\prime}$ and switch each shortcut to lie in $\GL(S^{\prime\prime})$.  We can fill the resulting word by \cite{RY}.

\subsection{Proof of Theorem \ref{theorem:parabtodiag}, part (b).}
We now turn to part (b) of the theorem, which is by far the easiest part.  Without loss of generality, $2\# S + T=2p$ (otherwise we could just expand $S$).  By our assumptions on $S,T$, we know $\Delta$ is of the form
$$d_{1}e_{1}n_{1}d_{2}e_{2}n_{2}d_{3}e_{3}n_{3},$$
where $d_{i}$ a word in $GL(S)$, $e_{i}$ a word in $Sp(T)$, and $n_{i}$ is a product of a bounded number of shortcuts for elements of $N_{S,T}$.  By moving the $d_{i}$ and $e_{i}$ to the left (applying Lemma \ref{lemma:ptd1}), we homotope at quadratic cost to
$$d_{1}d_{2}d_{3}e_{1}e_{2}e_{3}n^{\prime}$$
where $n^{\prime}$ is a product of a bounded number of shortcuts for elements of $N_{S,T}$.  We can fill $d_{1}d_{2}d_{3}$ because it lives inside some $\GL(S^{\prime})$ where $S^{\prime}\supset S$ has size at least $5$.  We can fill $n^{\prime}$ by Lemma \ref{lemma:ptd3}.  Hence we are left with $e_{1}e_{2}e_{3}$, a word in $\Sp(T)$, as desired.

\subsection{Proof of Theorem \ref{theorem:parabtodiag}, part (c)}
The proof of (c) is essentially the same, except that the $e_{i}$ are now shortcut words instead of actual words.  Assume, without loss of generality that $\#S\geq 3$, and let $T=\{\pm t\}$, so that $e_{i}$ is a product of words of the form $\he_{2t}(x)$ and $\he_{-2t}(x)$.  If we wish to proceed as in part (b), it must be shown that we can move the $e_{i}$ past the $d_{j}$ and $n_{j}$. We will need the following lemma.\\

\begin{lemma}
\label{lemma:commutingshortcuts}
If $d$ is a word in the generators $\SL(S)$ then
$$\delta(d\he_{2t}(x),\he_{2t}(x)d)=O(\ell(d)^{2}+\ell(\he_{2t}(x))^{2}).$$
\end{lemma}

\begin{proof}
Homotope as follows:\\

\begin{tabular}{l l}
$\he_{2t}(x) d$ & \\
$\leadsto [\he_{s-t}(\lfloor\frac{x}{2}\rfloor),e_{s+t}]e_{2t}(2\{\frac{x}{2}\})d$ &
Lemma \ref{lemma:steinberg}\\
$\leadsto [\he_{s-t}(\lfloor\frac{x}{2}\rfloor),e_{s+t}]de_{2t}(2\{\frac{x}{2}\})$ &
\\
$\leadsto d[\prod_{\alpha\in S-t}\he_{\alpha}(x_{\alpha}),\prod_{\alpha\in S+t}\he_{\alpha}(y_{\alpha})]$ & Lemma \ref{lemma:ptd1}\\
$\leadsto d\he_{2t}(x)$ & Lemma \ref{lemma:ptd3}\\
\end{tabular}\\

where we have
$$\prod_{\alpha\in S-t}e_{\alpha}(x_{\alpha})=de_{s-t}(\lfloor\frac{x}{2}\rfloor)d^{-1}$$
and
\[\prod_{\alpha\in S+t}e_{\alpha}(y_{\alpha})=
de_{s+t}d^{-1}.\qedhere\]
\end{proof}

If we tried to naively apply the lemma to move $e_{i}$ past $d_{j}$, it would cost $O(\ell(e_{i})^{2}+n\ell(d_{j})^{2})$, where $n=O(\log(\ell(e_{i})))$ is the number of shortcuts involved in $e_{i}$.  To improve from $O(\ell^{2}\log(\ell))$ to $O(\ell^{2})$, we employ the following trick to move $d_{j}$ past $\he_{2t}(x)$ at cost $O(\ell(d_{j})\ell(\he_{2t}(x))$. Divide $d_{j}$ into $O(\ell(d_{j})/\ell(\he_{2t}(x)))$ segments $d_{jk}$ of length at most $\ell(\he_{2t}(x))$.  Then the lemma tells us that we can move any $d_{jk}$ past $\he_{2t}(x)$ at cost
$$O(\ell(\he_{2t}(x))^{2}+\ell(d_{jk})^{2})=O(\ell(\he_{2t}(x))^{2}).$$
Repeating the trick for all the segments moves $d_{j}$ past $\he_{2t}(x)$ at cost $O(\ell(\he_{2t}(x))\ell(d_{j}))$, and hence we can move $d_{j}$ past $e_{i}$ at cost
$$O(\ell(e_{i})\ell(d_{j}))=O((\ell(e_{i})+\ell(d_{j}))^{2})$$
as desired.

Now all we have to do is homotope $n_{j}e_{i}$ to $e_{i}n_{j}^{\prime}$, where $n_{j}^{\prime}$ is a product of a bounded number of shortcuts for unipotents in $N_{S,T}$.  We move $e_{i}$ past each $\he_{\alpha}(x)$ appearing in $n_{j}$ by switching all of the shortcuts $\he_{\pm 2t}(y)$ appearing in $e_{i}$ to lie in some $\Sp(T^{\prime})$ and switching $\he_{\alpha}(x)$ to lie in some $\SpH_{S^{\prime},T_{\prime}}$ with $\#S^{\prime}\geq 3$, then applying Lemma \ref{lemma:ptd1}.  To do this, just choose some $s^{\prime}\in S$ not appearing in $\alpha$, and take $T^{\prime}=\{\pm t, \pm s^{\prime}\}$ and $S^{\prime}=S\setminus \{s^{\prime}\}$.  (We can apply the lemma because we know that $\alpha\in S^{\prime}-T$ or $\alpha\in Z_{S^{\prime}}$).

\subsection{Proof of Theorem \ref{theorem:parabtodiag}, part (d).}
Now we prove part (d) of the theorem.  If $\#S < p$, the proof is similar to part (b).  We know $\Delta$ has the form
$$d_{1}n_{1}d_{2}n_{2}d_{3}n_{3}$$
where $d_{i}$ is a word in $S$, and by the Lemma \ref{lemma:blockshortcut} we can move the $d_{i}$ to the left, and fill the entire resulting word using Lemma \ref{lemma:ptd3} and the fact that $\GL(q)$ has quadratic Dehn function for $q\geq 5$.

When $\#S = p$, we can switch all the shortcuts involved in the $n_{i}$ to live inside $P_{S,\emptyset}$, so that without loss of generality $\Delta$ is a word in $P_{S,\emptyset}$.  Theorem \ref{theorem:lipfillsl} states that $P_{S,\emptyset}(\RR)$ is Lipschitz $1$-connected, so we can use proceed as in part (a).  The results we used about $\Sp(T)/U(T)$ all have analogues for the symmetric space $\SL(S)/SO(S)$ proven in \cite[\S 4, \S 5]{RY}.  In particular, we quadratically break $\Delta$ into a collection of words of the form $w_{1}w_{2}w_{3}$ representing the identity, where each $w_{i}$ is of the form $dn$, with $d$ a word in the diagonal blocks of a parabolic subgroup $P$ of $\GL(S)$ and $n$ a product of a bounded number of shortcuts for elements of the unipotent radical of $P$.  (A parabolic subgroup of $\GL(S)$ is the stabilizer of some $\RR^{S^{\prime}}$, where $S^{\prime}\subset S$).  It is clear that such $w_{1}w_{2}w_{3}$ have quadratic fillings by just applying Lemma \ref{lemma:ptd1}, so we have proved the theorem.
\section{$\SpH_{S,T}$ has quadratic Dehn function}
\label{section:solvable}
In this section we will show the following theorem:

\begin{theorem}
\label{theorem:solvable2} Assume that $S,T\subset \mathcal{H}$ are disjoint sets of half roots, with $S\perp T$, $S$ isotropic, and $T$ symplectic.  If $\# S \geq 3$ and $\#T \geq 1$, then the group $\SpH_{S,T}(\ZZ)$ has quadratic Dehn function.
\end{theorem}

Our attack proceeds by first showing that $\SpH_{S,T}(\ZZ)$ is cocompact in $\SpH_{S,T}(\RR)$. A powerful theorem of Cornulier and Tessera\cite{dct3} gives some simple criteria for a solvable Lie group to have quadratic Dehn function, and we show that these criteria hold for $\SpH_{S,T}(\RR)$ (or rather, some conjugate).

This section is organized as follows. \S\ref{subsection:solpreliminaries} explains the concept of compact presentation and Dehn functions for compactly presented groups, and provides some lemmas explaining our approach.  Lemma \ref{lemma:cocompactdehnfunction} states that a compactly presented group has the same Dehn function as any cocompact discrete subgroup, and Lemma \ref{lemma:hstcocompact} confirms that $\SpH_{S,T}(\ZZ)$ is cocompact in $\SpH_{S,T}(\RR)$. Lemma \ref{lemma:hstdiagonal} describes a group $G$ conjugate to $\SpH_{S,T}(\RR)$ in $\Sp(2p;\RR)$ which will be slightly easier to work with. Finally, \S\ref{subsection:dct} explains Cornulier and Tessera's theorem and verifies that it applies to $G$. Throughout this section $S\subset \mH$ is taken to be isotropic of size at least $3$ and $T\subset \mH$ is symplectic of size at least $1$.

\subsection{Basic setup}
\label{subsection:solpreliminaries}
We now explain Dehn functions for compactly presented groups, prove that $\SpH_{S,T}(\ZZ)$ is cocompact in the solvable group $\SpH_{S,T}(\RR)$, and define a conjugate $G$ of $\SpH_{S,T}(\RR)$ which is slightly easier to work with.
\begin{definition}
\label{def:compactlypresented}
A locally compact group $G$ is said to be compactly presented if we can find a compact generating set $\Sigma$ and a natural number $k$ so that every relation in $\Sigma^{\ast}$ is a product of conjugates of relations of length at most $k$.  We call a relation of length at most $k$ a relator.  The Dehn function of $G$ is defined by letting $\delta_{G}(\ell)$ be the maximum number of conjugates of relators needed to fill a relation of length $\ell$ in $\Sigma^{\ast}$.  Up to the usual constants, this is well defined and does not depend on our choice of $\Sigma$ or $k$.
\end{definition}

\begin{lemma}
\label{lemma:cocompactdehnfunction}
Suppose $\Gamma$ is a discrete cocompact subgroup of a compactly presented group $G$.  If $G$ has quadratic Dehn function, then so does $\Gamma$.
\end{lemma}

\begin{proof}
This is proved in \cite{dct}.\end{proof}

Let $n_{S}=\#S$ and $2n_{T}=\#T$, and let $\mH_{0}=S\cup -S\cup T$ We start with the following observation.
\begin{lemma}
\label{lemma:hstcocompact}
$\SpH_{S,T}(\ZZ)$ is a cocompact subgroup of $\SpH_{S,T}(\RR)$.
\end{lemma}

\begin{proof}
Observe that $\TT_{S}(\ZZ)\times \TT_{T}(\ZZ)$ is a cocompact subgroup of $\TT_{S}(\RR)\times \TT_{T}(\RR)$, and $N_{S,T}(\ZZ)$ is a cocompact subgroup of $N_{S,T}(\RR)$.  Hence the map
$$\SpH_{S,T}(\RR)/\SpH_{S,T}(\ZZ)\rightarrow
\TT_{S}(\RR)\times\TT_{T}(\RR)/\TT_{S}(\ZZ)\times\TT_{T}(\ZZ)$$
is a fiber bundle over a compact space with compact fiber $N_{S,T}(\RR)/N_{S,T}(\ZZ)$.  Such a bundle most have compact total space, which proves the lemma.\end{proof}

We would like to show that $\SpH_{S,T}(\RR)$ has quadratic Dehn function.  For notational reasons, it is easier to work with a certain conjugate of this group.
\begin{lemma}
\label{lemma:hstdiagonal}
$\SpH_{S,T}(\RR)$ is isomorphic to the group $G$ given by all $A\in P_{S,T}(\RR)$ which map to positive diagonal matrices in $\SL(S;\RR)\times\Sp(T;\RR)$ under the map
$P_{S,T}(\RR)\mapsto \GL(S;\RR)\times \Sp(T;\RR)$. 
\end{lemma}

\begin{proof}
By basic linear algebra, there exists $M\in \GL(S)\times\Sp(T)\subset \Sp(\mH_{0})$ with the following properties.
\begin{itemize}
\item $M$ carries the standard basis for $\RR^{S}$ to an eigenbasis for the action of $\TT_{S}$ on $\RR^{S}$.
\item $M$ carries the standard basis for $\RR^{T}$ to a symplectic eigenbasis for the action of $\TT_{T}$ on $\RR^{T}$.
\end{itemize}
Observe that conjugation by $M$ preserves $N_{S,T}$, and takes
$\TT_{S}\times\TT_{T}$ to $\Diag_{\Sp}\cap\GL_{S}\times\Sp_{T}$.  Hence conjugation by $M$ yields an isomorphism from $\SpH_{S,T}(\RR)$ to $G$.\end{proof}

We will refer to the group of positive diagonal matrices in $\SL(S;\RR)$ as $D$, and refer to the group of diagonal matrices in $\Sp(T;\RR)$ as $E$, so that $G=(D\times E)\ltimes N_{S,T}(\RR)$.

\subsection{The statement of Cornulier and Tessera's theorem}
\label{subsection:dct}
We wish to show that, by a theorem of Cornulier and Tessera, the group $G$ of Lemma \ref{lemma:hstdiagonal} has quadratic Dehn function. In order to understand the statement of the theorem, we must introduce some notation involving weights. Throughout this subsection, $D$ will denote the diagonal subgroup of $\SL(S;\RR)$ and $E$ the diagonal subgroup of $\Sp(T;\RR)$. Essentially all of the forthcoming definitions can be found in \cite[\S 1]{dct3}.

Suppose $U\rtimes A$ is a solvable real Lie group, with $U$ nilpotent, connected, and simply connected, and $A\cong \RR^{n}$ for some $n$. Let $\mathfrak{u}$ denote the Lie algebra of $U$ (see \cite[\S 8]{fh}). It is well known that in this case, exponentiation yields a diffeomorphism from $\mathfrak{u}$ to $U$.  Given $v\in\mathfrak{u}$ (the Lie algebra of $U$) and $a\in A$, we define $a\cdot v$ by
$$\exp(a\cdot v)=a\exp(v)a^{-1}.$$

\begin{definition}
For $\alpha\in \Hom(A,\RR)$, let
$$\mathfrak{u}_{\alpha}:=\{v\in\mathfrak{u}:a\cdot v=\exp(\alpha(a))v\}.$$
If $\mathfrak{u}_{\alpha}\neq\{0\}$, we call $\alpha$ a weight, and denote the set of weights by $\mathcal{W}_{\mathfrak{u}}$.
\end{definition}

A weight $\alpha\in\mathcal{W}_{\mathfrak{u}}$ is called principal if $(\mathfrak{u}/[\mathfrak{u},\mathfrak{u}])_{0}\neq \{0\}$. The group $U\rtimes A$ is called standard solvable if $0$ is not a principal weight. We say that $\alpha,\beta\in \mathcal{W}_{\mathfrak{u}}\setminus\{0\}$ are quasi-opposite if $\alpha=t\beta$ for some $t<0$. Observe that weights give a grading on $\mathfrak{u}$, in the sense that
$$\mathfrak{u}=\bigoplus_{\alpha\in\mathcal{W}_{\mathfrak{u}}}\mathfrak{u}_{\alpha}.$$

\begin{definition}
Suppose $\mathfrak{g}$ a Lie algebra graded in a vector space $\mathcal{W}$ (so that $\mathfrak{g}=\bigoplus_{\alpha\in\mathcal{W}}\mathfrak{g}_{\alpha}$). We define the grading on $\mathfrak{g}\otimes\mathfrak{g}$ as
$$(\mathfrak{g}\otimes\mathfrak{g})_{\alpha}=
\bigoplus_{\beta+\gamma=\alpha}\mathfrak{g}_{\beta}\otimes\mathfrak{g}_{\gamma}.$$
\end{definition}

As addition is associative, this grading induces a grading on $\mathfrak{g}^{\otimes n}$, which descends to a grading on $\mathfrak{g}\wedge\mathfrak{g}$ and $\Sym^{2}\mathfrak{g}$ and so on.

\begin{definition}
Let $H_{2}(\mathfrak{g})$ be the vector space $\ker(d_{2})/\text{im}(d_{3})$, where the maps
$$\mathfrak{g}\wedge\mathfrak{g}\wedge\mathfrak{g}
{\buildrel d_{3}\over\rightarrow}\mathfrak{g}\wedge\mathfrak{g}
{\buildrel d_{2}\over\rightarrow}\mathfrak{g}$$
are given by
$$d_{3}(x\wedge y\wedge z)=[x,y]\wedge z+[y,z]\wedge x+[z,x]\wedge y,$$
$$d_{2}(x\wedge y)=[x,y].$$
\end{definition}

Note that $H_{2}(\mathfrak{g})$ is graded because $d_{2}$ and $d_{3}$ respect the grading on $\mathfrak{g}$.

\begin{definition}
For a graded Lie algebra $\mathfrak{g}$, let $Kill(\mathfrak{g})$ denote $\Sym^{2}\mathfrak{g}$ modulo the subspace spanned by all vectors of the form $[x,y]\odot z-[x,z]\odot y$. 
\end{definition}

Note again that $Kill(\mathfrak{g})$ inherits a grading. We are now prepared to state the theorem.

\begin{theorem}
\label{theorem:dctbig}
Suppose $U\rtimes A$ standard solvable, and $\mathfrak{u}/[\mathfrak{u},\mathfrak{u}]$ has no quasi-opposite weights. If
$$H_{2}(\mathfrak{u})_{0}=Kill(\mathfrak{u})_{0}=0,$$
then $U\rtimes A$ has quadratic Dehn function (as a compactly presented group).
\end{theorem}

\begin{proof}
This is \cite[Theorem E]{dct3}.
\end{proof}

\subsection{Application of Cornulier and Tessera's theorem}
\label{subsection:dctapplied}
We now show that Cornulier and Tessera's theorem applies to the group $G$ with $U=N_{S,T}(\RR)$ and $A=D\times E$ (recall that $D=G\cap\SL(S;\RR)$ and $E=G\cap\Sp(T;\RR)$). We will first describe the weights $\mathcal{W}_{\mathfrak{u}}$, then show that there are no quasi-opposite weights and that $Kill(\mathfrak{u})_{0}$ and $H_{2}(\mathfrak{u})_{0}$ vanish.

\paragraph{Computation of weights.}
We begin by describing $\Hom(A,\RR)$, which is a real vector space in the obvious way. It is easily seen that $\Hom(E,\RR)$ is spanned by the $t\in T$, in the sense that any $\alpha\in\Hom(E,\RR)$ can be written as
$$a\mapsto\sum c_{t}t(\log(a))$$
for some constants $c_{t}\in \RR$. In fact, this decomposition is unique, so $T$ gives a basis for $\Hom(E,\RR)$.  It is also true that $S$ spans $\Hom(D,\RR)$, but it is not quite a basis because there is a relation given by
$$\sum s(\log(a))=0$$
for all $a\in D$ (because $D\subset\SL(S;\RR)$, so $\log(a)$ is trace free for $a\in D$).

As a subset of $\sp(2p;\RR)$, the Lie algebra $\mathfrak{u}$ consists of matrices of the form $n-1$ where $n\in N_{S,T}(\RR)$. Write $X_{\alpha}$ for $e_{\alpha}-1$. Any common eigenvector of $A$ in $\mathfrak{u}$ is a multiple of some $X_{\alpha}$ with $\alpha\in \Phi N_{S,T}$ (this follows from the fact that such an $X_{\alpha}$ is in fact an eigenvector of the $A$ action with eigenvalue $\alpha$, and $\{X_{\alpha}:\alpha\in\Phi N_{S,T}\}$ spans $\mathfrak{u}$). It follows that the set of weights $\mathcal{W}_{\mathfrak{u}}$ is given by
$$\Phi N_{S,T}=\{s-t:s\in S, t\in T\}\cup \{s+s^{\prime}:s,s^{\prime}\in S\}.$$
The set of weights of $\mathfrak{u}/[\mathfrak{u},\mathfrak{u}]$ is given by
$$\{s-t:s\in S, t\in T\}.$$

\paragraph{Verifying the assumptions.}
We now show that Theorem \ref{theorem:dctbig} applies to our situation, thus proving Theorem \ref{theorem:solvable2}. Recall that we assumed $\#S\geq 3$ and $\#T\geq 1$. Clearly, $0$ is not a weight of $\mathfrak{u}/[\mathfrak{u},\mathfrak{u}]$, so $G$ is in fact standard solvable. Furthermore, $\mathfrak{u}/[\mathfrak{u},\mathfrak{u}]$ has no quasi-opposite weights. If $\# S\neq 4$, then there do not exist $\alpha,\beta\in\mathcal{W}_{\mathfrak{u}}$ such that $\alpha+\beta=0$, so $(\mathfrak{u}\otimes\mathfrak{u})_{0}=0$, and thus the subquotients $H_{2}(\mathfrak{u})_{0}$ and $Kill(\mathfrak{u})_{0}$ vanish as desired. It follows that from Theorem \ref{theorem:dctbig} that $G$ has quadratic Dehn function.

Now suppose $\#S=4$. It is now the case that $(\mathfrak{u}\otimes\mathfrak{u})_{0}$ is spanned by vectors of the form $X_{s_{1}+s_{2}}\otimes X_{s_{3}+s_{4}}$, where $s_{1},\ldots,s_{4}\in S$ are distinct (recall that $s_{1}+s_{2}+s_{3}+s_{4}$ vanishes in $\Hom(A,\RR)$). But for any $t\in T\neq \emptyset$, we see that
$$X_{s_{1}+s_{2}}\wedge X_{s_{3}+s_{4}}
=d_{3}(X_{s_{1}+t}\wedge X_{s_{2}-t}\wedge X_{s_{3}+s_{4}})$$
$$X_{s_{1}+s_{2}}\odot X_{s_{3}+s_{4}}
=[X_{s_{1}+t},X_{s_{2}-t}]\odot X_{s_{3}+s_{4}}
-[X_{s_{1}+t},X_{s_{3}+s_{4}}]\odot X_{s_{2}-t},$$
since the last term is zero. It follows that 
$$H_{2}(\mathfrak{u})_{0}=Kill(\mathfrak{u})_{0}=0,$$
and we have established Theorem \ref{theorem:solvable2}.

\section{Lipschitz fillings}
\label{section:lipfill}
In this section, we prove two theorems promised in \S\ref{section:parabtodiag}, each of which states that $\ell$-Lipschitz loops in a certain homogeneous space have $O(\ell)$-Lipschitz fillings. In each theorem statement, the reader may think of $S^{1}$ and $D^{2}$ as the unit circle and disk respectively in the Euclidean plane.

\begin{theorem}
\label{theorem:lipfillsp}
Suppose $S$ and $T$ are disjoint subsets of $\mathcal{H}$ with $S$ isotropic, $T$ symplectic, and $\#T\geq 6$.  Let $G=ker(P_{S,T}\rightarrow GL(S))$, and let $U(T)\subset Sp(T)\subset G$ be a maximal compact subgroup.  Let $X$ be the homogeneous space $G/U(T)$. There is a constant $c$ such that (for $L\geq 1$) any $L$-Lipschitz map $S^{1}\rightarrow X$ can be extended to a $cL$-Lipschitz map of the disk $D^{2}\rightarrow X$.
\end{theorem}

\begin{theorem}
\label{theorem:lipfillsl}
Suppose $S\subset \mH$ isotropic with $\#S=p$.  Let $G=P_{S,\emptyset}\cong \SL(p)\ltimes\Sym^{2}\RR^{p}$, and let $X$ be the homogeneous space $G/SO(S)$. There is a constant $c$ such that (for $L\geq 1$) any $L$-Lipschitz map $S^{1}\rightarrow X$ can be extended to a $cL$-Lipschitz map of the disk $D^{2}\rightarrow X$.
\end{theorem}

The section is organized as follows.  Subsection \ref{subsection:rules} details some basic rules for manipulating Lipschitz curves and Lipschitz homotopies in homogeneous spaces.  Subsection \ref{subsection:lipfillsp} will detail the proof of Theorem \ref{theorem:lipfillsp}.  Subsection \ref{subsection:lipfillsl} will prove Theorem \ref{theorem:lipfillsl}.

\subsection{Basic rules.}
\label{subsection:rules}
Let $X=G/K$ be a complete, simply connected Riemannian space homogeneous under $G$, with isotropy subgroup $K\subset G$ a maximal compact subgroup.  We say that $X$ is large-scale Lipschitz 1-connected if every unit speed closed curve of length $\ell\geq 1$ admits a $O(\ell)$-Lipschitz filling, as we wish to show in the theorems above (actually, since $X$ is homogeneous, this property is equivalent to the a priori stronger property of Lipschitz 1-connectedness defined in \cite{RYlip}). Similarly, we say that a class of loops $\mathcal{C}$ in $X$ has Lipschitz fillings if there exists a constant $C$ such that every $\gamma\in \mathcal{C}$ has a filling $f$ with $\Lip(f)\leq C\ell(\gamma)$.  (If $\mathcal{C},\mathcal{D}$ are two sets of curves $[0,1]\rightarrow X$, and we have some function $f:\mathcal{C}\rightarrow\mathcal{D}$ such that $f(\gamma)$ always has the same endpoints as $\gamma$, then we say that we can Lipschitz homotope from $\gamma$ to $f(\gamma)$ if we can always find a $O(\Lip(\gamma))$ filling of the loop formed by concatenating $\gamma$ and the reverse of $f(\gamma)$). In this section we describe the types of manipulations needed to prove that such an $X$ is large-scale Lipschitz 1-connected.  (Essentially all of these ideas are taken from \cite[\S 8.3]{RY}). Throughout this section, $D^{2}(\ell)$ denotes the Euclidean square $[0,\ell]\times[0,\ell]$ and $S^{1}(\ell)$ denotes its boundary (also, $S^{1}$ denotes $S^{1}(1)$).  Here is a brief summary of the results of this subsection.
\begin{itemize}
\item {\it Normal form triangles:\\} To show that $X$ is large-scale Lipschitz $1$-connected, it suffices to show that so-called normal form triangles have Lipschitz fillings (Proposition \ref{lemma:nftriangles}).
\item {\it Combinatorialization:\\} Lipschitz curves and homotopies in $G$ descend to Lipschitz curves and homotopies in $X$ (Lemma \ref{lemma:combinatorialization}).
\item {\it Reparameterization:\\} Any Lipschitz reparameterization of the boundary of a Lipschitz disk admits a Lipschitz filling (Lemma \ref{lemma:reparameterization}).
\item {\it Insertion:\\} Given a concatenation $\alpha\beta\gamma$ of curves in $X$, and an endpoint-fixing Lipschitz homotopy from $\beta$ to $\beta^{\prime}$, one can construct a Lipschitz homotopy from $\alpha\beta\gamma$ to $\alpha\beta^{\prime}\gamma$ (Lemma \ref{lemma:insertion}).
\item {\it Stacking:\\} One can use templates to produce Lipschitz homotopies (Lemma \ref{lemma:liptemplate}), and hence ``stack" a bounded number of Lipschitz homotopies and obtain a Lipschitz homotopy (Corollary \ref{corollary:stacking}).
\end{itemize}

\paragraph{Normal form triangles.} Suppose that $\Omega$ is a normal form for $X$, i.e., for $x,y\in X$, we have that $\Omega(x,y)$ is a unit speed curve in $X$ connecting $x$ and $y$ with length at most some constant times $d(x,y)$ (we do not require $\Omega(x,y)$ to fellow travel with any particular geodesic).
\begin{definition}
An $\Omega$ triangle is a triangle connecting points $x,y,z\in X$ by edges given by $\Omega$.
\end{definition}
Let $\Delta_{0}$ denote the Euclidean equilateral triangle of side length $1$. Then \cite[Proposition 8.14]{RY} says the following.

\begin{lemma}
\label{lemma:nftriangles}
Suppose there is a constant $c$ such that for any $x,y,z\in X$ there is a Lipschitz $f_{x,y,z}: \Delta_{0}\rightarrow X$ with $f$ taking the sides of $\Delta_{0}$ to $\Omega(x,y)$, $\Omega(y,z)$, and $\Omega(x,z)^{-1}$, and $Lip(f)\leq c\diam(x,y,z)+c$.  Then there is a constant $C$ such that any unit speed rectifiable curve mapping the circle of radius $\ell\geq 1$ into $X$ has a $C$-Lipschitz extension over the disk of radius $\ell$.
\end{lemma}

\paragraph{Combinatorialization.} It is generally more convenient with actual matrices in $G$ rather than equivalence classes of matrices (i.e., points of $X$). For instance, there is a canonical way to concatenate a finite sequence of curves in $G$ which are based at the identity.

\begin{definition}
Fix a left-invariant metric on $G$ and a base point $\ast\in X$.  If $\gamma:[0,\ell]\rightarrow G$ is a curve in $G$, let $[\gamma]_{X}$ denote the curve
$$[\gamma]_{X}:[0,\ell]\rightarrow X$$
$$[\gamma]_{X}:t\mapsto \gamma(t)\cdot\ast$$
Often, by abuse of notation we will let $\gamma$ stand for $[\gamma]_{X}$.\\

Given two curves $\gamma_{1}:[0,\ell_{1}]:\rightarrow G$ and $\gamma_{2}:[0,\ell_{2}]\rightarrow G$, let $\gamma_{1}\gamma_{2}$ be the concatenation
$$\gamma_{1}\gamma_{2}:[0,\ell_{1}+\ell_{2}]\rightarrow G$$
given by $t\mapsto \gamma_{1}(t)$ for $t\leq \ell_{1}$ and $t\mapsto \gamma_{1}(\ell_{1})\cdot\gamma_{2}(t-\ell_{1})$ for $t\geq \ell_{1}$.
\end{definition}

We now show that Lipschitz curves and homotopies in $G$ descend to Lipschitz curves and homotopies in $X$.

\begin{lemma}
\label{lemma:combinatorialization}
There exists a constant $C$ such that, if $\gamma:[0,\ell]\rightarrow G$ is a Lipschitz curve in $G$, then
$$\Lip([\gamma]_{X})\leq C\Lip(\gamma),$$
and furthermore, if $f:D^{2}(\ell)\rightarrow G$ is a Lipschitz disk in $G$, then
$$\Lip([f]_{X})\leq C\Lip(f).$$
\end{lemma}

\begin{proof}
Consider the projection $\pi:G\rightarrow X$.  For every $g\in G$, the operator norm of the derivative
$$d\pi|_{g}:T_{g}G\rightarrow T_{[g]_{X}}X$$
is the same, because the map $G\rightarrow X$ is $G$-equivariant, and $G$ acts on both $G$ and $X$ by isometries.  Taking $C$ to be this norm, the result follows.\end{proof}

When we say that a curve $\gamma$ in $G$ represents $g\in G$, we mean that $\gamma$ begins at the identity and ends at $g$.\\

\paragraph{Reparameterization.}
The following is \cite[Lemma 8.13]{RY}.
\begin{lemma}
\label{lemma:reparameterization}
There exists a constant $C$ such that, if $\beta:S^{1}(\ell)\rightarrow X$ is a reparameterization of $\gamma:S^{1}(\ell)\rightarrow X$, and $f:D^{2}(\ell)\rightarrow X$ is a filling of $\beta$, then $\gamma$ admits a filling $\tilde{f}:D^{2}(\ell)\rightarrow X$ with
$$\Lip(\tilde{f})\leq C\max\{\Lip(f),\Lip(\gamma)\}+C$$
\end{lemma}

There are two natural notions of a Lipschitz homotopy between two curves in $X$ having the same endpoints.  A corollary of the lemma is that these are equivalent.
\begin{definition}
If $\beta,\gamma:[0,\ell]\rightarrow X$ are curves in $X$, an endpoint preserving Lipschitz homotopy from $\beta$ to $\gamma$ is a Lipschitz map $f:D^{2}(\ell)\rightarrow X$ such that
\begin{itemize}
\item $f(x,y)=\beta(x)$ if $y=0$,
\item $f(x,y)=\gamma(x)$ if $y=1$,
\item and $f$ is constant on $\{0\}\times [0,\ell]$ and $\{1\}\times [0,\ell]$.
\end{itemize}
\end{definition}

\begin{corollary}
\label{corollary:liphomotopies}
There is some constant $C$ with the following properties.  Given any Lipschitz homotopy $f$ from $\beta$ to $\gamma$, there exists a map $\tilde{f}:D^{2}(2\ell)\rightarrow X$ with $\tilde{f}$ constant on the three sides
$$\{0,1\}\times[0,2\ell]\cup[0,2\ell]\times\{1\}$$
and $\tilde{f}|\{0\}\times[0,2\ell]$ given by $\beta$ followed by the reverse of $\gamma$ and
$$\Lip(\tilde{f})\leq C\Lip(f)$$
\end{corollary}
\begin{proof}
This follows immediately from the Lemma \ref{lemma:reparameterization} since the boundary of the homotopy is a reparameterization of the concatenation of $\beta$ with the reverse of $\gamma$.\end{proof}

\paragraph{Insertion.}
The following lemma is obvious, but we state it here because it will be freely used throughout.
\begin{lemma}
\label{lemma:insertion}
Given a concatenation $\alpha\beta\gamma:[0,3\ell]\rightarrow X$ of curves $\alpha,\beta,\gamma:[0,\ell]\rightarrow X$ and an endpoint-fixing Lipschitz homotopy $f:D^{2}(\ell)\rightarrow X$ between $\beta$ and some curve $\tilde{\beta}:[0,\ell]\rightarrow X$, we can produce an endpoint-fixing Lipschitz homotopy $\tilde{f}:D^{2}(3\ell)\rightarrow X$ between $\alpha\beta\gamma$ and $\alpha\tilde{\beta}\gamma$ with
$$\Lip(\tilde{f})\leq\max\{\Lip(\alpha),\Lip(f),\Lip(\gamma)\}.$$
\end{lemma}

\begin{proof}
Just take $\tilde{f}(x,y)$ to be $\alpha(x)$ for $x\leq \ell$, $f(x-\ell,\frac{y}{3})$ for $\ell\leq x\leq 2\ell$, and $\gamma(x-2\ell)$ for $x\geq 2\ell$.\end{proof}

\paragraph{Stacking.}
The following lemma says that if we have a finite triangulation (or more general polygonal decomposition) of $D^{2}$ with sides labeled by curves in $X$, together with a Lipschitz filling for each cell, we get a Lipschitz filling of the boundary curve.  Let $\tau$ be a decomposition of $D^{2}=D^{2}(1)$ into a finite collection of Euclidean polygons.  A map $\gamma$ from the one-skeleton of $\tau$ into $X$ induces a loop $\gamma|_{S^{1}}:S^{1}\rightarrow X$ and a loop $\partial\Delta:S^{1}\rightarrow X$ for each 2-cell $\Delta$ of $\tau$.  (For the sake of precision, take $\partial\Delta$ to be constant speed).
\begin{lemma}
\label{lemma:liptemplate}
There is a constant $C=C_{\tau}>0$ such that, if each $\partial\Delta:S^{1}\rightarrow X$ has an $\ell$-Lipschitz filling $f_{\Delta}:D^{2}\rightarrow X$, then $\gamma_{S^{1}}$ has a $C\ell$-Lipschitz filling $\tilde{f}:D^{2}\rightarrow X$.
\end{lemma}

\begin{proof}
Define $\tilde{f}$ on each $\Delta$ by precomposing $f_{\Delta}:D^{2}\rightarrow X$ with a bi-Lipschitz map $\Delta\rightarrow D^{2}$.\end{proof}

We will often use the following obvious corollary:
\begin{corollary}
\label{corollary:stacking}
Fix a natural number $n$.  Suppose we are given a collection of curves
$$\alpha_{i}:[0,1]\rightarrow X\quad\quad(i=0,\ldots,n)$$
and a collection of endpoint-fixing $\ell$ Lipschitz homotopies $\{f_{i}\}_{i=1,\ldots,n}$ from $\alpha_{i}$ to $\alpha_{i+1}$.  Then there exists an $O(\ell)$ Lipschitz homotopy from $\alpha_{0}$ to $\alpha_{n}$.
\end{corollary}

\begin{proof}
Decompose $D^{2}$ as a stack of $n$ rectangles of height $1/n$.  Map each horizontal edge to the appropriate $\alpha_{i}$ and apply the lemma.\end{proof}

\subsection{Proof of Theorem \ref{theorem:lipfillsp}}
\label{subsection:lipfillsp}
We will now prove Theorem \ref{theorem:lipfillsp}.  By Lemma \ref{lemma:nftriangles}, it suffices to define a normal form $\Omega$, then provide uniformly Lipschitz fillings for $\Omega$-triangles.  By abuse of notation, we define $\Omega$ as a map from $G$ to the space of curves in $G$ which start at the identity, with the understanding that for $x,y\in X$, we have $\Omega(x,y)=\Omega(g)\cdot x$ where $g\in G$ is such that $gx=y$ (there are many $g$ with this property, but one should choose $g$ as close as possible to the identity).  It will be reasonably obvious from the way we define $\Omega$ that it is a normal form for elements $g\in G$ in the sense that $\ell(\Omega(g))=O(d(\Id,g))$, so it follows that $\Omega$ is a normal form on $X$ as well.

\paragraph{Defining $\Omega$.} Let $N=N_{S,T}\subset G$ and $Z=Z_{S,T}\subset N$. We will define $\Omega$ via the following steps.
\begin{itemize}
\item {\it Step 1:} Define a curve $\hu(V)$ in $G$ running from the identity to an element of the form $u(z_{s}\otimes v)\in N_{S,T}$ (see \S\ref{section:subgroups}).
\item {\it Step 2:} Define a curve $\huZ(V)$ in $G$ running from the identity to a matrix of the form $u_{Z}(V)\in Z$.
\item {\it Step 3:} Define $\Omega$ as a product of a finite number of $\hu(V)$ and curves in $\Sp(T)$.
\end{itemize}

\paragraph{Defining $\hu(V)$.} We begin by defining a curve in $G$ running from the identity to $u(V)\in N=N_{S,T}$, where $V=w_{s}\otimes v$ for some $s\in S$ and $v\in \RR^{T}$.  Define a curve $\hat{u}(V)$ representing $u(V)$ in $X$ as follows.
\begin{definition}
\begin{itemize}
\item If $\Vert V \Vert\leq 1$, take $\hat{u}(V)$ to be $t\mapsto u(tV)$.  (By abuse of notation, we sometimes write such a curve as $u(V)$).
\item If $\Vert V \Vert > 1$, take $\hat{u}(V)$ to be a curve of the form $\gamma u(\frac{V}{\Vert v \Vert})\gamma^{-1}$ where $\gamma$ is a geodesic in $\Sp(T)$ representing some matrix $D_{v}\in \Sp(T)$ with $v$ an eigenvector of $D_{v}$ with eigenvalue $\Vert v \Vert$ (and $\Vert D_{v} \Vert_{2}=O(\Vert v \Vert)$). 
\end{itemize}
\end{definition}

\paragraph{Defining $\huZ(V)$.} We now We now define a curve from the identity to $\huZ(ww^{\prime})$ (where $w,w^{\prime}\in\RR^{S}$.
\begin{definition}
For $w,w^{\prime}\in \RR^{S}$, with $\|w^{\prime}\|=1$, let $\huZ(ww^{\prime})$ be some $[\hat{u}(w\otimes v_{t}),\hat{u}(w_{\prime}\otimes v_{-t})]$ where $t\in T$.
\end{definition}

For $\alpha=s+s^{\prime}\in\Phi Z$, we let $\he_{\alpha}(x)$ denote $\huZ(xz_{s}\otimes z_{s^{\prime}})$.  Similarly, for $\alpha=s-t\in S-T$, let $\he_{\alpha}(x)$ denote $\hu(xz_{s}\otimes z_{t})$.\\

\paragraph{Defining $\Omega$.}
\begin{itemize}
\item Suppose $z\in Z$, then there are real numbers $x_{s+s^{\prime}}$ such that $z=\prod_{s+s^{\prime}\in S+S}u_{Z}(x_{s+s^{\prime}}w_{s}w_{s^{\prime}})$.  Define $\Omega(z)$ to be 
$$\prod_{s+s^{\prime}\in S+S}\huZ(x_{s+s^{\prime}}w_{s}w_{s^{\prime}}).$$
\item Suppose $n\in N$, and let $Ab(n)=\sum_{s\in S} n_{s} \otimes w_{s}$ for some $n_{s}\in \RR^{T}$.  Then we define $\Omega(n)$ to be
$$\left(\prod_{s\in S}\hat{u}(n_{s}\otimes w_{s})\right)\Omega(z)$$
 where $z=(\prod_{s\in S}u(n_{s}\otimes w_{s}))^{-1}n\in Z$.
\item Suppose $g\in G$.  Define $\Omega(g)$ to be
$$\gamma_{M}\Omega(M^{-1}g),$$ where $M$ is the $\Sp(T)$ part of $g$, so that $M^{-1}g\in N$, and $\gamma_{M}$ is a geodesic in $\Sp(T)$ representing $M$.  Observe that $\Omega$ always defines a curve consisting of a path in $\Sp(T)$ followed by a bounded number curves of the form $\hat{u}(V)$.\\
\end{itemize}

\paragraph{Filling $\Omega$ triangles.}
We will now describe an algorithm which produces an appropriately Lipschitz disk filling any $\Omega$-triangle $\Delta_{X}$ in $X$.  First, observe that we can lift $\Delta_{X}$ to an $\Omega$-triangle $\Delta$ in $G$ with $\diam(\Delta)=O(\diam(\Delta_{X})+1)$ (since $K$ is compact).  We shall produce a $O(\diam(\Delta)+1)$-Lipschitz filling in $G$, which descends to a $O(\diam(\Delta_{X})+1)$-Lipschitz filling in $X$ as desired by Lemma \ref{lemma:combinatorialization}.\\

Observe that $\Delta$ has the form
$$\gamma_{1}\Omega(n_{1})\gamma_{2}\Omega(n_{2})\gamma_{3}\Omega(n_{3})$$
where $\gamma_{i}$ are curves in $\Sp(T)$ and $n_{i}\in N$.  We fill this $\Omega$-triangle via the following procedure, which relies on several lemmas proved below.
\begin{itemize}
\item Use the conjugation lemma (\ref{lemma:lipconjugation}) to homotope to a curve of the form
$$\gamma_{1}\gamma_{2}\gamma_{3}
\prod_{j=1}^{O(1)}\hu(z_{s_{j}}\otimes v_{j})$$
where $s_{j}\in S$, and $v_{j}\in \RR^{T}$, and $\prod_{j=1}^{O(1)}\hu(z_{s_{j}}\otimes v_{j})$ denotes a product of a bounded number of terms (at most
$3(\#S + \#S^{2})$) of the form $\hu(z_{s_{j}}\otimes v_{j})$ with $s_{j}\in S$, and $v_{j}\in \RR^{T}$.
\item Fill $\gamma_{1}\gamma_{2}\gamma_{3}$, since $\Sp(T)/U(T)$ is CAT(0).
\item Homotope $\prod_{j=1}^{O(1)}\hu(z_{s_{j}}\otimes v_{j})$ to a product
$$\prod_{j=1}^{O(1)}\he_{\alpha_{j}}(x_{j})$$
by corollary \ref{corollary:factoring}, which converts each $\hu(z_{s_{j}}\otimes v_{j})$ into a product of at most $3\# T$ words of the form $\he_{\alpha}(x)$ with $\alpha=s-t$ for some $s\in S$ and $t\in T$.
\item Fill this product using the Steinberg relations (Lemma \ref{lemma:lipsteinberg}).  One can do this by enumerating the elements of $S-T$ as $\alpha_{1},\ldots,\alpha_{n}$ and using the lemma to successively move each $\alpha_{i}$ all the way to the left, leaving behind a bounded number of $\he_{\beta}$ (with $\beta\in\Phi Z$), which can be filled in the same way.  (Enumerate the elements of $\Phi Z$, successively move each to the left with Lemma \ref{lemma:lipsteinberg}, and fill with Lemma \ref{lemma:lipsteinberg}).
\end{itemize}

It is clear that, given the following lemmas, the procedure above works to produce the desired filling.
\begin{itemize}
\item {\bf Conjugation lemma: (Lemma \ref{lemma:lipconjugation})} Suppose $v,v^{\prime}\in \RR^{T}$ and $\gamma$ is a geodesic in $\Sp(T)$ representing some matrix $M$ with $Mv=v^{\prime}$.  Then there is a Lipschitz homotopy from
$\gamma\hu(z_{s}\otimes v)\gamma^{-1}$ to $\hu(z_{s}\otimes v^{\prime})$.
\item {\bf Factoring lemma: (Lemma \ref{corollary:factoring})} Suppose $v\in \RR^{T}$ and $s\in S$ with $v=\sum_{t\in T}x_{t}z_{t}$.  Then we can Lipschitz homotope from
$\hu(z_{s}\otimes v)$ to
$$\prod_{{\pm t}\subset T}\he_{s-t}(x_{t})\he_{s+t}(x_{-t})
[\he_{s-t}(\frac{-1}{2}x_{t}x_{-t}),\he_{s+t}(1)]$$
(which is a concatenation of at most $3\#T$ curves of the form $\he_{s-t}(x)$).
\item {\bf Steinberg relations: (Lemma \ref{lemma:lipsteinberg})} We can Lipschitz fill any of the basic relations involving curves of the form $\he_{s-t}(x)$ for $s\in S, t\in T, x\in \RR$.
\end{itemize}

We now proceed to prove these lemmas.  This is more difficult than the special linear case \cite[lemma 8.8]{RY} because $\Sp(T)$ cannot contract an arbitrary pair of vectors in $\RR^{T}$.  That is, if $v,w\in\RR^{T}$ and $\omega(v,w)\neq 0$, then for $M\in \Sp(T)$ we have $\omega(Mv,Mw)=\omega(v,w)$, so if $v,w$ are eigenvectors of $M$ with positive eigenvalues, and $M$ shrinks one of $v,w$, then it must expand the other.  Thus, in addition to the standard strategy of contract-and-fill, we will need ideas from Allcock's proof that higher Heisenberg groups have quadratic isoperimetric inequality \cite{allcock} and Cornulier and Tessera's proof that Abels's group has quadratic isoperimetric inequality \cite{dct2}.

\begin{lemma}
\label{lemma:lipconjugation}
Suppose $v,v^{\prime}\in \RR^{T}$ and $\gamma$ is a geodesic in $\Sp(T)$ representing some matrix $M$ with $Mv=v^{\prime}$.  Then for $s\in S$, we can Lipschitz homotope from 
$\gamma\hu(z_{s}\otimes v)\gamma^{-1}$ to $\hu(z_{s}\otimes v^{\prime})$.
\end{lemma}

\begin{proof}
The proof is quite similar to that of \cite[Lemma 8.16]{RY}.

\begin{figure}[t]
\labellist
\small\hair 2pt

\pinlabel {$\triangleright$} [c] at 48 223
\pinlabel {$\triangleright$} [c] at 122 223
\pinlabel {$\triangleleft$} [c] at 196 223
\pinlabel {$\triangledown$} [c] at 242 114
\pinlabel {$\triangleleft$} [c] at 196 0
\pinlabel {$\triangleright$} [c] at 122 0
\pinlabel {$\triangleright$} [c] at 48 0
\pinlabel {$\triangledown$} [c] at 1 114
\pinlabel {$\triangledown$} [c] at 98 114
\pinlabel {$\triangledown$} [c] at 147 114

\pinlabel $\gamma_{D_{v^{\prime}}}$ [] at 48 212
\pinlabel $u_{\overline{v^{\prime}}}$ [] at 120 212
\pinlabel $\gamma_{D_{v^{\prime}}}$ [] at 195 212
\pinlabel $\gamma$ [] at 232 114
\pinlabel $\gamma_{D_{v}}$ [] at 195 9
\pinlabel $u_{\overline{v}}$ [] at 120 9
\pinlabel $\gamma_{D_{v}}$ [] at 48 9
\pinlabel $\gamma$ [] at 11 114
\pinlabel $\delta$ [] at 89 114
\pinlabel $\delta$ [] at 157 114

\endlabellist
\centering
\centerline{\psfig{file=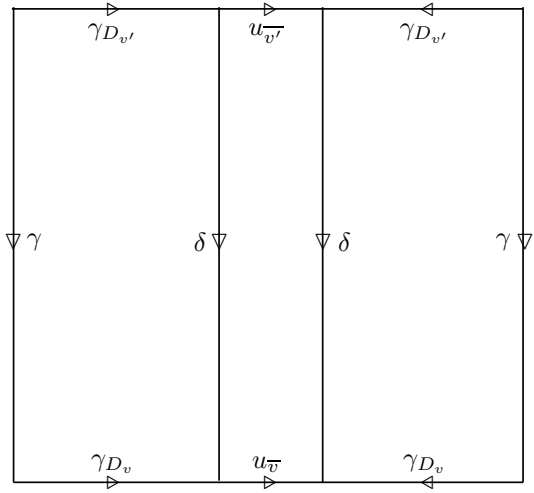,scale=80}}
\caption{A homotopy from $\gamma\hu(z_{s}\otimes v)\gamma^{-1}$ to 
$\hu(z_{s}\otimes v^{\prime})$}
\label{figure:conjugation}
\end{figure}

Recall that $\hu(z_{s}\otimes v)$ has the form $\gamma_{D_{v}}\hu(z_{s}\otimes \bv)\gamma_{D_{v}}^{-1}$, where $\bv$ is $v$ if $\Vert v \Vert \leq 1$ and $\frac{v}{\| v\|}$ if $\| v \| > 1$. Suppose we could choose a curve $\delta$ in $\Sp(T)$ as indicated in the diagram \ref{figure:conjugation} such that $\ell(\delta)=O(\ell(\gamma_{D_{v}}^{-1}\gamma^{-1}\gamma_{D_{v^{\prime}}}))$ and the middle rectangle remains ``skinny", i.e, the distance between $\delta(t)$ and $u(\overline{v})\delta(t)$ is bounded.  Then we could Lipschitz fill the left and right rectangles because $\Sp(T)/U(T)$ is CAT(0), and we could Lipschitz fill the middle rectangle by mapping the horizontal segment at height $t$ to the curve given by
$$s\mapsto\delta(t)u(sw(t))$$
where $w(t)\in\RR^{T}$ is the (bounded) vector $w(t)=\delta(t)^{-1}\bv$.  Hence it suffices to find $\delta$.  We proceed through several cases of increasing difficulty.\\

{\it Case 1:} Observe that if $\bv=\bvp$, we can find $\delta$ as desired, because the stabilizer $G_{v}$ of $v$ in $\Sp(T)$ is undistorted, so there is a curve of length $O(\ell(D_{v}^{-1}\gamma^{-1}D_{v}))$ inside $G_{v}$ representing $ D_{v}^{-1}\gamma^{-1}D_{v^{\prime}}$.\\

{\it Case 2:} If $\bv$ is a positive real multiple of $\bvp$, we can reduce to the previous case by making $\delta$ begin by scaling $v$ to norm $1$ and end by scaling $v$ down to norm $\|\bvp\|$.\\

{\it Case 3:} Finally, we can always reduce to the previous cases by making $\delta$ begin by rotating $v$ to the direction of $v_{\prime}$ (i.e., the initial segment of $\delta$ will be a geodesic in $U(T)$ representing a matrix which carries $v$ to $w$, and the length of this segment is $O(1)$ because $U(T)$ is compact).\end{proof}

The following two lemmas are proved using the usual contraction technique and will be useful throughout.
\begin{lemma}
\label{lemma:isorect}
If $v,w\in\RR^{T}$ with $\omega(v,w)=0$, and $s,s^{\prime}\in S$ then we can Lipschitz homotope from
$\hu(z_{s}\otimes v)\hu(z_{s^{\prime}}\otimes w)$
to
$\hu(z_{s^{\prime}}\otimes w)\hu(z_{s}\otimes v)$.
\end{lemma}

\begin{proof}
First, note that by Lemma \ref{lemma:reparameterization} it suffices to provide a filling of $\hu(z_{s}\otimes v)\hu(z_{s^{\prime}}\otimes w)\hu(z_{s}\otimes v)^{-1})\hu(z_{s^{\prime}}\otimes w)^{-1}$.\\

\begin{figure}[t]
\labellist
\small\hair 2pt

\pinlabel {$\triangleleft$} [c] at 245 335
\pinlabel {$\triangle$} [c] at 496 168
\pinlabel {$\triangle$} [c] at 497 168
\pinlabel {$\triangleleft$} [c] at 245 0
\pinlabel {$\triangle$} [c] at 1 168
\pinlabel {$\triangle$} [c] at 0 168
\pinlabel {$\nwarrow$} [c] at 67 270
\pinlabel {$\nearrow$} [c] at 429 268
\pinlabel {$\searrow$} [c] at 435 62
\pinlabel {$\swarrow$} [c] at 67 67
\pinlabel {$\triangleleft$} [c] at 245 209
\pinlabel {$\triangle$} [c] at 370 168
\pinlabel {$\triangleleft$} [c] at 245 127
\pinlabel {$\triangle$} [c] at 129 168

\pinlabel {$\hu(z_{s^{\prime}}\otimes w)$} [l] at 5 168
\pinlabel {$\gamma$} at 74 273
\pinlabel {$\hu(z_{s}\otimes v)$} at 245 321
\pinlabel {$\gamma$} at 425 273
\pinlabel {$\hu(z_{s^{\prime}}\otimes w)$} [r] at 487 168
\pinlabel {$\gamma$} at 425 60
\pinlabel {$\hu(z_{s}\otimes v)$} at 245 12
\pinlabel {$\gamma$} at 74 60
\pinlabel {$\hu(z_{s^{\prime}}\otimes Mw)$} [l] at 132 168
\pinlabel {$\hu(z_{s}\otimes Mv)$} at 245 220
\pinlabel {$\hu(z_{s^{\prime}}\otimes Mw)$} [r] at 368 168
\pinlabel {$\hu(z_{s}\otimes Mv)$} at 245 116

\endlabellist
\centering
\centerline{\psfig{file=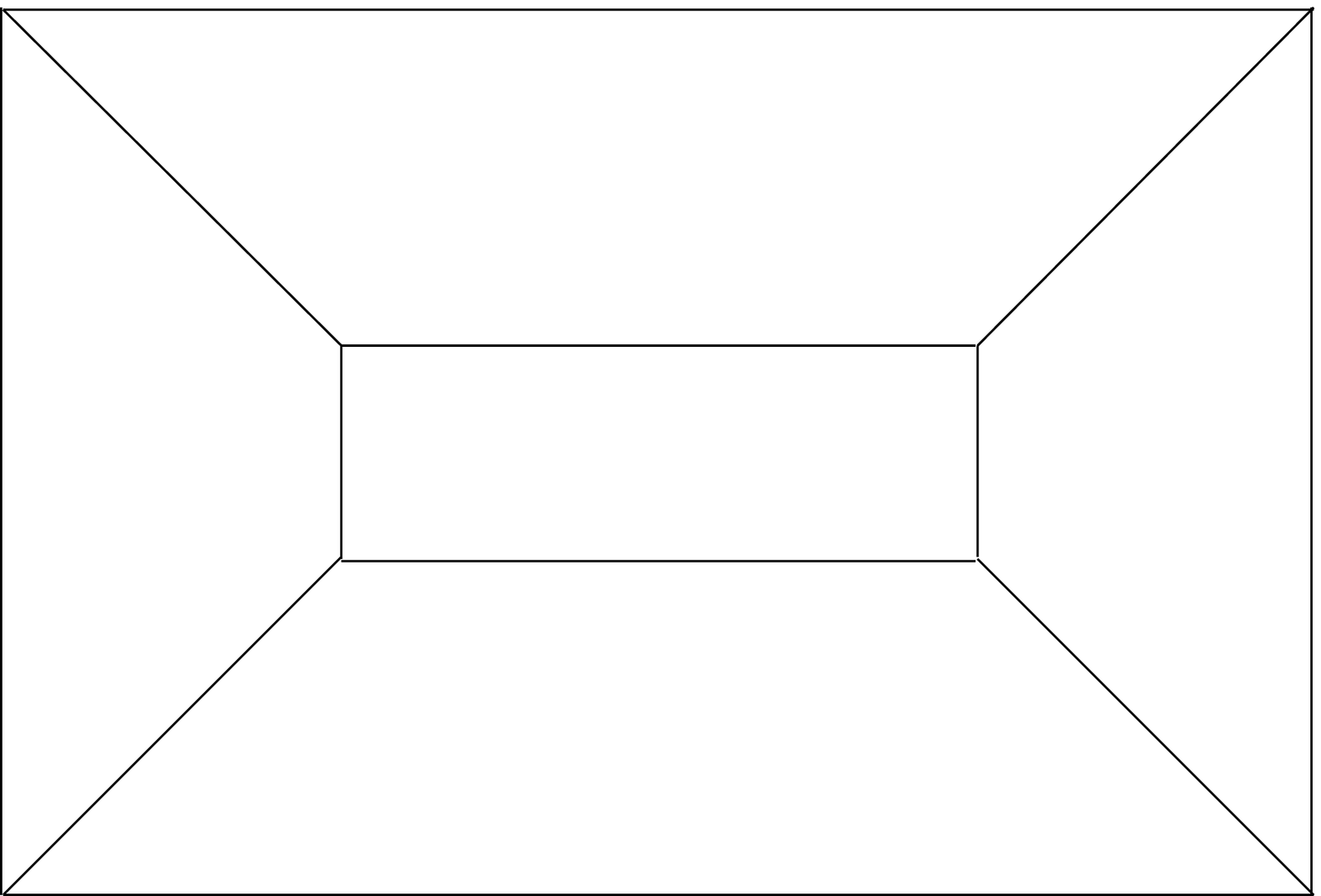,scale=75}}
\caption{A filling of $[\hu(z_{s}\otimes v),\hu(z_{s^{\prime}}\otimes w)]$}
\label{figure:rectanglehomotopy2}
\end{figure}

Choose a geodesic $\gamma$ in $\Sp(T)$ representing $M\in \Sp(T)$ with $\| Mv\|\leq 1$ and $\| Mw\|\leq 1$ and $\ell(\gamma)=O((\log|v|+\log|w|)^{2})$.  Then we apply Lemma \ref{lemma:liptemplate} to the diagram.  The outer trapezoids can all be filled by the conjugation lemma (\ref{lemma:lipconjugation}).  The middle rectangle can be filled because it has four sides of bounded length.\end{proof}

\begin{lemma}
\label{lemma:isotri}
If $v,w\in\RR^{T}$ with $\omega(v,w)=0$, and $s\in S$ then we can Lipschitz homotope from
$\hu(z_{s}\otimes v)\hu(z_{s}\otimes w)$
to
$\hu(z_{s}\otimes (v+w))$.
\end{lemma}

We omit the proof because it is essentially the same as the previous lemma.  Note however that if we wish to apply this lemma to $\hu(V_{1})\hu(V_{2})$ we must have that $V_{1},V_{2}\in \{z_{s}\}\otimes \RR^{T}$ for some $s\in S$ (the previous lemma had no such restriction).  Following Allcock \cite{allcock}, we will now use these lemmas to obtain fillings for loops obtained by concatenating $\hu(V_{i})$ where all the $V_{i}$ are contained in some set of the form $\{z_{s}\}\otimes \RR^{T^{\prime}}$ where $T^{\prime}$ is a proper symplectic subset of $T$.

\begin{lemma}
\label{lemma:allcock}
Let $T^{\prime}\subset T$ be a proper symplectic subset and let $s\in S$.  Then for fixed $n$, there is a Lipschitz filling of any loop of the form
$$\hu(z_{s}\otimes v_{1})\ldots\hu(z_{s}\otimes v_{n})$$
where each $v_{i}$ is in $\RR^{T^{\prime}}$
\end{lemma}

A typical concatenation of this form will {\it not} form a loop.

\begin{proof}

\begin{figure}[t]
\labellist
\small\hair 2pt

\pinlabel {$\hu(z_{s}\otimes v_{1})$} [] at 16 120
\pinlabel {$\hu(z_{s}\otimes v_{2})$} [] at 61 120
\pinlabel {$\ldots$} [] at 127 120
\pinlabel {$\hu(z_{s}\otimes v_{n-1})$} [] at 197 120
\pinlabel {$\hu(z_{s}\otimes v_{n})$} [] at 241 120
\pinlabel {$\he_{s+t}(-1)$} [l] at 246 60
\pinlabel {$\he_{s-t}(x_{n-2})$} [] at 195 9
\pinlabel {$\ldots$} [] at 128 9
\pinlabel {$\he_{s-t}(x_{1})$} [] at 60 9
\pinlabel {$\he_{s+t}(1)$} [r] at 9 60
\pinlabel {$\hu(z_{s}\otimes V_{1})$} [l] at 33 60
\pinlabel {$\hu(z_{s}\otimes (V_{2}-x_{2}z_{t}))$} [l] at 90 60
\pinlabel {$\ldots$} [r] at 221 60

\endlabellist
\centering
\centerline{\psfig{file=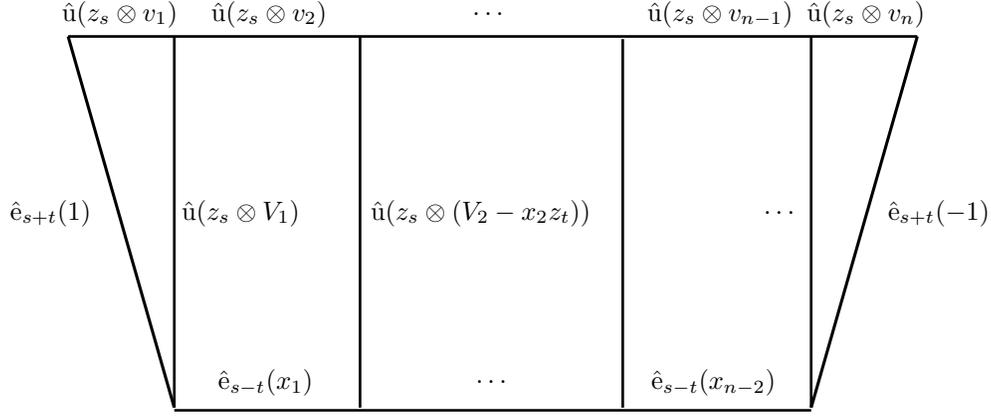,scale=125}}
\caption{A homotopy from  $\hu(z_{s}\otimes v_{1})\ldots\hu(z_{s}\otimes v_{n})$ to $\he_{s+t}(1)\he_{s-t}(x_{1})\ldots\he_{s-t}(x_{n-2})\he_{s+t}(-1).$}
\label{figure:allcock}
\end{figure}

Since $T^{\prime}\neq T$, we can choose $t\in T$ such that $\pm t\notin T^{\prime}$.  Let $V_{i}=\sum_{j\leq i} v_{j}$ and $x_{i}=\omega(V_{i},v_{i+1})$ so that
$$u(z_{s}\otimes v_{1})\ldots u(z_{s}\otimes v_{i})
u(z_{s}\otimes(v_{1}+\ldots+v_{i}))^{-1}
=u((x_{1}+\ldots+x_{i-1})z_{s}^{2})$$
We will start applying Lemma \ref{lemma:liptemplate} to diagram \ref{figure:allcock} in order to homotope from
$\hu(z_{s}\otimes v_{1})\ldots\hu(z_{s}\otimes v_{n})$ to
$$\he_{s+t}(1)\he_{s-t}(x_{1})\ldots\he_{s-t}(x_{n-2})\he_{s+t}(-1).$$
(The use of the lemma is justified by the fact that $n$ is fixed; otherwise, the Lipschitz constant of the resulting homotopy would depend on $n$ as well as the length of the original loop).  In the diagram, take each vertical edge to have the form $\hu(z_{s}\otimes (V_{i}-x_{i}z_{t}))$.  Each rectangle (as well as each of the triangles) in figure \ref{figure:allcock} has the form indicated in diagram \ref{figure:holonomysquare}, where we abbreviate $\hu(z_{s}\otimes v)$ as $v$, and
$$\omega(v,a)=\omega(u,b).$$
(We think of the triangles of figure \ref{figure:allcock} as squares with one degenerate side).

\begin{figure}[t]
\labellist
\small\hair 2pt

\pinlabel{$\triangleright$} [c] at 104 223
\pinlabel{$\triangle$} [c] at 208 115
\pinlabel{$\triangleright$} [c] at 104 0
\pinlabel{$\triangle$} [c] at 1 115
\pinlabel{$\nwarrow$} [c] at 46 175
\pinlabel{$\nearrow$} [c] at 163 176
\pinlabel{$\nwarrow$} [c] at 165 47
\pinlabel{$\nearrow$} [c] at 44 48

\pinlabel {$b$} [] at 104 214
\pinlabel {$v+b-u-a$} [r] at 204 115
\pinlabel {$a$} [] at 104 12
\pinlabel {$v-u$} [l] at 4 115
\pinlabel {$-u$} [l] at 55 176
\pinlabel {$b-u$} [r] at 156 176
\pinlabel {$v-a$} [r] at 156 47
\pinlabel {$v$} [l] at 55 47

\endlabellist
\centering
\centerline{\psfig{file=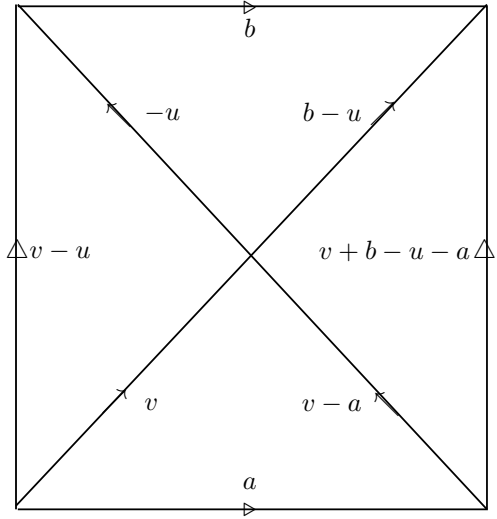,scale=85}}
\caption{A filling for a rectangle from diagram \ref{figure:allcock}. Since $\omega(u,a)=\omega(v,b)$, each triangle is isotropic.}
\label{figure:holonomysquare}
\end{figure}

Now, figure \ref{figure:holonomysquare} shows how to decompose the boundary rectangle into isotropic triangles which can be filled by \ref{lemma:isotri}.  We have thus reduced to the desired product
$$\he_{s+t}(1)\he_{s-t}(x_{1})\ldots\he_{s-t}(x_{n-2})\he_{s+t}(-1).$$
We can fill the product of $\he_{s-t}(x_{i})$ by Lemma \ref{lemma:isotri}, then fill the remaining curve $\he_{s+t}(1)\he_{s+t}(-1)$ because it is small.\end{proof}

\begin{corollary}
\label{corollary:factoring}
Suppose $v\in\RR^{T}$ is equal to $\sum_{\{\pm t\}\subset T}x_{t}z_{t}+x_{-t}z_{-t}$ and $s\in S$.  Then we can Lipschitz homotope from $\hu(z_{s}\otimes v)$ to a product
$$\prod_{{\pm t}\subset T}\he_{s-t}(x_{t})\he_{s+t}(x_{-t})
[\he_{s-t}(\frac{-1}{2}x_{t}x_{-t}),\he_{s+t}(1)]$$
\end{corollary}

\begin{proof}
Let $v_{\pm t}=x_{t}z_{t}+x_{-t}z_{-t}$ be the projection of $v$ on $\RR^{\{\pm t\}}$.  Then by Lemma \ref{lemma:isotri} we can homotope from $\hu(v)$ to
$\prod_{{\pm t}\subset T}\hu(z_{s}\otimes v_{\pm t})$.
Hence it suffices to homotope each $\hu(z_{s}\otimes v_{\pm t})$ to a product of $\he$ of the specified form.  But this follows immediately by applying the previous Lemma \ref{lemma:allcock} with $T^{\prime}=\{\pm t\}$.\end{proof}

We will describe how to fill all relations between $\he$.  This lemma is the only place where we use the fact that $\# T\geq 3$.\\

\begin{lemma}
\label{lemma:lipsteinberg}
Suppose that $s,s^{\prime}\in S$ and $t,t^{\prime}\in T$.  We can Lipschitz homotope as follows.
\begin{itemize}
\item[(a)] From
$\he_{s-t}(x)\he_{s-t}(y)$ to $\he_{s-t}(x+y)$.
\item[(b)] If $t\neq \pm t^{\prime}$, from
$\he_{s-t}(x)\he_{s^{\prime}-t^{\prime}}(y)$ to
$\he_{s^{\prime}-t^{\prime}}(y)\he_{s^{\prime}-t^{\prime}}(x)$.
\item[(c1)] If $t\neq \pm t^{\prime}$, from
$[\he_{s-t}(x),\he_{s^{\prime}+t}(y)]$ to
$[\he_{s^{\prime}-t^{\prime}}(xy),\he_{s+t^{\prime}}(\pm 1)]$.\\
(Here the sign is determined by $s,s^{\prime},t,t^{\prime}$).
\item[(d)] From
$\he_{s-t}(x)[\he_{s^{\prime}-t^{\prime}}(y),\he_{s^{\prime\prime}+t^{\prime}}(z)]$ to
$[\he_{s^{\prime}-t^{\prime}}(y),\he_{s^{\prime\prime}+t^{\prime}}(z)]\he_{s-t}(x)$.
\item[(c2)] From $[\he_{s-t}(x),\he_{s^{\prime}+t}(y)]$ to $[\he_{s-t^{\prime}}(xy),\he_{s^{\prime}+t^{\prime}}(1)]$
\item[(e)] From $[\he_{s-t}(x),\he_{s^{\prime}+t}(y)][\he_{s-t}(x^{\prime}),\he_{s^{\prime}+t}(y^{\prime})]$ to $[\he_{s-t}(xy+x^{\prime}y^{\prime}),\he_{s^{\prime}+t}(1)]$
\item[(f)] From $[\he_{s-t}(x),\he_{s^{\prime}+t}(y)][\he_{s-t^{\prime}}(x^{\prime}),\he_{s^{\prime}+t^{\prime}}(y^{\prime})]$

to $[\he_{s-t^{\prime}}(x^{\prime}),\he_{s^{\prime}+t^{\prime}}(y^{\prime})][\he_{s-t}(x),\he_{s^{\prime}+t}(y)]$
\end{itemize}
\end{lemma}

\begin{proof}
(a) follows by Lemma \ref{lemma:isotri}.  (b) follows by Lemma \ref{lemma:isorect}.\\

To prove (c1), it suffices to Lipschitz fill curves of the form
$$[\he_{s-t}(x),\he_{s^{\prime}+t}(y)][\he_{s^{\prime}-t^{\prime}}(xy),\he_{s+t^{\prime}}(\pm 1)]^{-1}$$
Noting that 
$[\he_{s^{\prime}-t^{\prime}}(xy),\he_{s+t^{\prime}}(\pm 1)]^{-1}
=[\he_{s+t^{\prime}}(\pm 1),\he_{s^{\prime}-t^{\prime}}(xy)]$, 
we homotope as follows.\\
\begin{tabular}{l l}
$[\he_{s-t}(x),\he_{s^{\prime}+t}(y)][\he_{s+t^{\prime}}(\pm 1),\he_{s^{\prime}-t^{\prime}}(xy)]$ & \\
$[\he_{s-t}(x)\he_{s+t^{\prime}}(\pm 1),\he_{s^{\prime}+t}(y)\he_{s^{\prime}-t^{\prime}}(xy)]$ & Lemma \ref{lemma:isorect}\\
$[\hu(z_{s}\otimes(xz_{t}\pm z_{-t^{\prime}})),\hu(z_{s^{\prime}}\otimes(yz_{-t}+xyz_{t^{\prime}}))]$ & Lemma \ref{lemma:isotri}\\
\end{tabular}\\
This final curve can be filled by Lemma \ref{lemma:isorect}.\\

To prove (d), choose $t^{\prime\prime}\neq \pm t$ and homotope as follows.\\
\begin{tabular}{l l}
$\he_{s-t}(x)[\he_{s^{\prime}-t^{\prime}}(y),\he_{s^{\prime\prime}+t^{\prime}}(z)]$ & \\
$\he_{s-t}(x)[\he_{s^{\prime\prime}+t^{\prime\prime}}(y),\he_{s^{\prime}-t^{\prime\prime}}(z)]$ & part (c1)\\
$[\he_{s^{\prime\prime}+t^{\prime\prime}}(yz),\he_{s^{\prime}-t^{\prime\prime}}(\pm 1)]\he_{s-t}(x)$ & Lemma \ref{lemma:isorect}\\
$[\he_{s^{\prime}-t^{\prime}}(y),\he_{s^{\prime\prime}+t^{\prime}}(z)]\he_{s-t}(x)$ & part (c1)\\
\end{tabular}\\

To prove part (c2), choose $t^{\prime\prime}\notin \{\pm t,\pm t^{\prime}\}$ and homotope as follows.\\
\begin{tabular}{l l}
$[\he_{s-t}(x),\he_{s^{\prime}+t}(y)]$ & \\
$[\he_{s^{\prime}-t^{\prime\prime}}(xy),\he_{s+t^{\prime\prime}}(\pm 1)]$ & \\
$[\he_{s-t^{\prime}}(xy),\he_{s^{\prime}+t^{\prime}}(\pm 1)]$ & \\
\end{tabular}\\

To prove part (e), it suffices to fill the loop
$$[\he_{s-t}(x),\he_{s^{\prime}+t}(y)][\he_{s-t}(x^{\prime}),\he_{s^{\prime}+t}(y^{\prime})]([\he_{s-t}(xy+x^{\prime}y^{\prime}),\he_{s^{\prime}+t}(\pm 1)])^{-1}.$$
Choose $t^{\prime},t^{\prime\prime}$ so that $\pm t,\pm t^{\prime}, \pm t^{\prime\prime}$ are distinct and homotope as follows.\\
\begin{tabular}{l l}
$[\he_{s-t}(x),\he_{s^{\prime}+t}(y)][\he_{s-t}(x^{\prime}),\he_{s^{\prime}+t}(y^{\prime})]([\he_{s-t}(xy+x^{\prime}y^{\prime}),\he_{s^{\prime}+t}(\pm 1)])^{-1}$ & \\
$[\he_{s-t}(x),\he_{s^{\prime}+t}(y)][\he_{s-t^{\prime}}(x^{\prime}),\he_{s^{\prime}+t^{\prime}}(y^{\prime})][\he_{s-t^{\prime\prime}}(xy+x^{\prime}y^{\prime}),\he_{s^{\prime}+t^{\prime\prime}}(\pm 1)]$ & parts (c1), (c2)\\
$[\he_{s-t}(x)\he_{s-t^{\prime}}(x^{\prime})\he_{s-t^{\prime\prime}}(xy+x^{\prime}y^{\prime}),\he_{s^{\prime}+t}(y)\he_{s^{\prime}+t^{\prime}}(y^{\prime})\he_{s^{\prime}+t^{\prime\prime}}(\pm 1)]$ & Lemma \ref{lemma:isorect}\\
$[\hu(z_{s}\otimes(xz_{t}+x^{\prime}z_{t^{\prime}}+xyz_{t^{\prime\prime}}+x^{\prime}y^{\prime}z_{t^{\prime\prime}})),
\hu(z_{s^{\prime}}\otimes(yz_{-t}+y^{\prime}z_{-t}\pm z_{-t^{\prime\prime}}))]$ & Lemma \ref{lemma:isotri}\\
\end{tabular}\\
This loop can be fill via Lemma \ref{lemma:isorect}.\\

Part (f) follows trivially from part (d).\end{proof}

By proving all of the promised lemmas, we have proved Theorem \ref{theorem:lipfillsp}.

\subsection{Proof of Theorem \ref{theorem:lipfillsl}.}
\label{subsection:lipfillsl}
Suppose $S\subset \mH$ isotropic with $\#S=p$.  Let $G=P_{S,\emptyset}\cong \SL(p)\ltimes\Sym^{2}\RR^{p}$, and let $X$ be the symmetric space $G/SO(S)$.  Recall that Theorem \ref{theorem:lipfillsl} asserts the existence of a constant $c$ such that (for $L\geq 1$) any $L$-Lipschitz map $S^{1}\rightarrow X$ can be extended to a $cL$-Lipschitz map of the disk $D^{2}\rightarrow X$.  This subsection is devoted to the proof of this theorem.  As in the previous subsection, we proceed by defining a normal form $\Omega$ for $X$, finding appropriate fillings for $\Omega$-triangles, and appealing to Lemma \ref{lemma:nftriangles}.\\

\paragraph{Defining $\Omega$. }
As before, we define $\Omega$ in $G$ rather than $X$ (see the beginning of \S\ref{subsection:lipfillsp}).  We begin by defining a curve $\hu(vw)$ representing $u(vw)\in Z_{S}$.  Note that the factorization $vw$ is unique up to scaling.
\begin{definition}
For each pair of unit length vectors $v,w\in \RR^{S}$, choose a matrix $M_{v,w}\in \SL(\# S)$ with bounded entries which has $v,w$ as eigenvectors with eigenvalue $\exp(-1)$ and all other eigenvalues positive (so that real powers of $M$ are well defined.

Given $v,w\in \RR^{S}$, define $\hu(vw)$ as follows.  If $\| v \| \|w\|\leq 1$, simply take $\hu(vw)(t)=u(vwt)$ (up to scale).  Otherwise, let $M=M_{v/\|v\|,w/\|w\|}$ and let $\gamma$ be a geodesic in $\SL(S)$ representing $M^{x}$, where $M^{x}\frac{v}{\|v\|}M^{x}\frac{w}{\|w\|}=vw$, then define $\hu(vw)$ to be the concatenation
$$\gamma \hu\left(\frac{vw}{\| v \| \| w\|}\right) \gamma^{-1}$$
\end{definition}

We are now in a position to define $\Omega$.  Suppose $z\in Z_{S}$.  We can write
$$z=u\left(\sum_{s+s^{\prime}\in \Phi Z_{S}}x_{s+s^{\prime}}z_{s}z_{s^{\prime}}\right).$$  \begin{definition}
Define $\Omega(z)$ to be
$$\prod_{s+s^{\prime}\in \Phi Z_{S}}\he_{s+s^{\prime}}(x_{s+s^{\prime}})$$
\end{definition}
Let $\phi_{\SL}:G\rightarrow \SL(S)$ be projection.
\begin{definition} For $g\in G$, define $\Omega(g)$ as follows.  Let $\gamma$ be a geodesic in $\SL(S)$ representing $\phi_{\SL}(g)$, then $\phi_{\SL}(g)^{-1}g\in Z_{S}$ so we can set $\Omega(g)$ to be
$$\gamma \Omega(\phi_{\SL}(g)^{-1}g)$$
\end{definition}

\paragraph{Filling $\Omega$-triangles.} The procedure we use to fill $\Omega$-triangles is similar to that used in the proof of Theorem \ref{theorem:lipfillsp}.  An $\Omega$-triangle has the form
$$\gamma_{1}\Omega(n_{1})\gamma_{2}\Omega(n_{2})\gamma_{3}\Omega(n_{3})$$
We can fill this by the following procedure.
\begin{itemize}
\item Homotope to $\gamma_{1}\gamma_{2}\gamma_{3}\prod^{O(1)}\hu(vw)$ using the conjugation lemma (\ref{lemma:lipconjugationsl}).
\item Fill $\gamma_{1}\gamma_{2}\gamma_{3}$ using the fact that $\SL_{S}/SO(S)$ is CAT(0).
\item Use corollary \ref{corollary:lipbreaking} to break each $\hu(vw)$ into a product of at most $p^{2}$ curves of the form $\he_{\alpha}(x)$, yielding a product of a bounded number of $\he$.
\item Fill this product using the Steinberg relations (Lemma \ref{lemma:lipsteinsl}).
\end{itemize}

To do this we need several lemmas, developed below.  We begin by proving a weak version of Lemma \ref{lemma:lipconjugationsl}.

\begin{proposition}
\label{proposition:lipconjugationsl}
Suppose $M,M^{\prime}\in \SL(S)$ and $v,v^{\prime}\in \RR^{S}$ with $\|v\|,\|v^{\prime}\|\leq 1$ and $Mv=M^{\prime}v^{\prime}$.  Let $\gamma_{M}:[0,1]\rightarrow G$ be a curve connecting $1$ to $M$ and $\gamma_{M^{\prime}}:[0,1]\rightarrow G$ a curve connecting $1$ to $M^{\prime}$.  Then there is a homotopy $f:[0,3]\rightarrow G$ from $\gamma_{M}u(v^{2})\gamma_{M}^{-1}$ to $\gamma_{M^{\prime}}u(v^{\prime2})\gamma_{M^{\prime}}^{-1}$, with
$$\Lip(f)=O(\Lip(\gamma_{M})+\Lip(\gamma_{M^{\prime}})+1).$$
\end{proposition}

\begin{proof}
The proof is omitted as it follows the same lines as that of Lemma \ref{lemma:lipconjugation}.\end{proof}

\begin{lemma}
\label{lemma:lipconjugationsl}
If $\gamma$ is a curve in $\SL$ representing some matrix $M$, and $v,w,v^{\prime},w^{\prime}\in \RR^{S}$, with $Mv=v^{\prime}$ and $Mw=w^{\prime}$, then there exists a homotopy from $\gamma \hu(vw) \gamma^{-1}$ to $\hu(v^{\prime}w^{\prime})$ with Lipschitz constant $O(\Lip(\gamma\hu(vw)\gamma^{-1})$.
\end{lemma}

\begin{proof}
We homotope as follows.  (Here $\overline{v}$ denotes $\frac{v}{\|v\|}$ if $\|v\|\geq 1$ and $v$ otherwise).\\

\begin{tabular}{l l}
$\gamma \gamma_{M_{v,w}} u(\overline{v}\overline{w})\gamma_{M_{v,w}}^{-1}\gamma^{-1}$ & \\
$\gamma \gamma_{M_{v,w}} u((\frac{\overline{v}+\overline{w}}{2})^{2})u((\frac{\overline{v}-\overline{w}}{2})^{2})\gamma_{M_{v,w}}^{-1}\gamma^{-1}$ & \\
$\gamma \gamma_{M_{v,w}} u((\frac{\overline{v}+\overline{w}}{2})^{2})
\gamma_{M_{v,w}}^{-1}\gamma^{-1}\gamma\gamma_{M_{v,w}}
u((\frac{\overline{v}-\overline{w}}{2})^{2})\gamma_{M_{v,w}}^{-1}\gamma^{-1}$ & Lemma \ref{lemma:insertion}\\
$\gamma_{M_{v^{\prime},w^{\prime}}} u((\frac{\overline{v^{\prime}}+\overline{w^{\prime}}}{2})^{2})
\gamma_{M_{v^{\prime},w^{\prime}}}^{-1}\gamma_{M_{v^{\prime},w^{\prime}}}
u((\frac{\overline{v^{\prime}}-\overline{w^{\prime}}}{2})^{2})\gamma_{M_{v^{\prime},w^{\prime}}}^{-1}$ & proposition \ref{proposition:lipconjugationsl}\\
$\gamma_{M_{v^{\prime},w^{\prime}}} u((\frac{\overline{v^{\prime}}+\overline{w^{\prime}}}{2})^{2})u((\frac{\overline{v^{\prime}}-\overline{w^{\prime}}}{2})^{2})\gamma_{M_{v^{\prime},w^{\prime}}}^{-1}$ & Lemma \ref{lemma:insertion}\\
$\hu(u^{\prime}v^{\prime})$ & \\
\end{tabular}\\

The use of the proposition is justified by the fact that
$$MM_{v,w}\frac{\overline{v}+\overline{w}}{2}=M\frac{v+w}{2}$$
\[=\frac{v^{\prime}+w^{\prime}}{2}=
M_{v^{\prime},w^{\prime}}\frac{\overline{v^{\prime}}+\overline{w^{\prime}}}{2}.\qedhere\]
\end{proof}

\begin{lemma}
\label{lemma:lipbreakingsl}
We can lipschitz homotope from
$\hu(vw)\hu(vw^{\prime})$
to $\hu(v(w+w^{\prime}))$.
\end{lemma}

\begin{proof}
Let $\gamma$ be a geodesic of length $O(\ell(\hu(vw)\hu(vw^{\prime})))$ in $\SL(S)$ representing a matrix $M$ such that $Mv,Mw,Mw^{\prime}$ all have norm less than $1$.  Homotope as follows.\\

\begin{tabular}{l l}
$\hu(vw)\hu(vw^{\prime})$ & \\
$\gamma^{-1}\gamma\hu(vw)\gamma^{-1}\gamma\hu(vw^{\prime})\gamma^{-1}\gamma$ & free insertion\\
$\gamma^{-1}\hu(MvMw)\hu(MvMw^{\prime})\gamma$ & Lemma \ref{lemma:lipconjugationsl}\\
$\gamma^{-1}\hu(Mv(Mw+Mw^{\prime}))\gamma$ & \\
$\hu(v(w+w^{\prime}))$ & Lemma \ref{lemma:lipconjugationsl}\\
\end{tabular}
\end{proof}

As usual, let $\he_{s+s^{\prime}}(x)$ denote $\hu(x_{s}x_{s^{\prime}})$.

\begin{corollary}
\label{corollary:lipbreaking}
If $v=\sum_{s\in S}a_{s}z_{s}$
and $w=\sum_{s\in S}b_{s}z_{s}$
then we can homotope from
$\hu(vw)$ to $\prod_{s+s^{\prime}}\he_{s+s^{\prime}}(c_{s+s^{\prime}})$
where $c_{s+s^{\prime}}$ is $a_{s}b_{s^{\prime}}+a_{s^{\prime}}b_{s}$ is $s\neq s^{\prime}$ and $a_{s}b_{s}$ otherwise.
\end{corollary}

\begin{proof}
Using the Lemma \ref{lemma:lipbreakingsl}, we break $\hu(vw)$ into
$$\prod_{s\in S}\hu(b_{s}vz_{s})$$
then break this into the desired product.\end{proof}

\begin{lemma}
\label{lemma:lipsteinsl}
We have the following Lipschitz homotopies.
\begin{itemize}
\item[(a)] From $\he_{\alpha}(x)\he_{\beta}(y)$ to $\he_{\beta}(y)\he_{\alpha}(x)$ for $\alpha,\beta\in\Phi Z_{S}$.
\item[(b)] From $\he_{\alpha}(x)\he_{\alpha}(y)$ to $\he_{\alpha}(x+y)$.
\end{itemize}
\end{lemma}

\begin{proof}
Each part can be proved in the same manner as Lemma \ref{lemma:lipbreakingsl}.\end{proof}

By the outline above, this suffices to prove Theorem \ref{theorem:lipfillsl}.

\bibliographystyle{plain}
\bibliography{bibliography}

\noindent
David Bruce Cohen\\
Department of Mathematics\\
Rice University, MS 136 \\
6100 Main St.\\
Houston, TX 77005\\
E-mail: {\tt dc17@rice.edu}\\

\end{document}